\documentclass{article}
\usepackage[utf8]{inputenc}

\usepackage{hyperref}
\usepackage{cite}

\usepackage{geometry} 

\usepackage{amssymb} 
\usepackage{amsmath} 
\usepackage{amsthm} 
\usepackage{mathtools}
\usepackage{scrextend} 
\usepackage[framemethod=tikz]{mdframed}  
\usepackage{enumitem} 
\usepackage{slashed} 
\usepackage{verbatim} 
\setlist[enumerate]{itemsep=0.5pt, topsep=4pt}

\usepackage{multicol} 
\usepackage{cancel} 

\newcommand{\lv}{\vspace{0.3cm}}
\newcommand{\vv}{\vspace{0.5cm}}
\newcommand{\vvv}{\vspace{1cm}}

\newcommand{\ddt}{\left. \frac{d}{dt} \right|_{t=0}}

\newcommand{\R}{\mathbb{R}}
\newcommand{\N}{\mathbb{N}}
\newcommand{\Z}{\mathbb{Z}}

\newcommand{\ds}{\displaystyle}

\newcommand{\dd}[1]{\frac{\partial}{\partial #1}}
\newcommand{\norm}[1]{\left\lVert#1\right\rVert}
\newcommand{\nb}{\underline{n}}

\newcommand{\normmm}[4]{\left\lVert#1\right\rVert_{(#2 \to  #3),#4}}
\newcommand{\normmC}[4]{ \left\lVert#1\right\rVert_{C, (#2 \to  #3),#4}}
\newcommand{\normmH}[4]{ \left\lVert#1\right\rVert_{H, (#2 \to  #3),#4}}
\newcommand{\normmmC}[3]{ \left\lVert#1\right\rVert_{C, #2,#3}}
\newcommand{\normmmH}[3]{ \left\lVert#1\right\rVert_{H, #2,#3}}

\newcommand{\gsc}{\mathfrak{g_{sc}}}
\newcommand{\trksc}{trK_{\mathfrak{g_{sc}}}}
\newcommand{\trko}{trK_{\mathfrak{B}}}
\newcommand{\gammasc}{\gamma_{\mathfrak{g_{sc}}}}
\newcommand{\gammao}{\gamma_{\mathfrak{B}}}
\newcommand{\HtoH}[3]{H^{#1}_{#2}\left([nm_0,\infty); H^{#3}(S^2)\right)}
\newcommand{\HtoL}[2]{H^{#1}_{#2}\left([nm_0,\infty); L^{2}(S^2)\right)}
\newcommand{\LtoH}[2]{L^{2}_{#1}\left([nm_0,\infty); H^{#2}(S^2)\right)}
\newcommand{\CtoH}[3]{C^{#1}_{#2}\left([nm_0,\infty); H^{#3}(S^2)\right)}

\newcommand{\LtoHr}[2]{L^{2}_{#1}\left([r_0,\infty); H^{#2}(S^2)\right)}
\newcommand{\CtoHr}[3]{C^{#1}_{#2}\left([r_0,\infty); H^{#3}(S^2)\right)}

\newcommand{\HtoHH}[3]{H^{#1}_{#2}\left([nm_0,\infty); \mathcal{H}^{#3}(S^2)\right)}
\newcommand{\LtoHH}[2]{L^{2}_{#1}\left([nm_0,\infty); \mathcal{H}^{#2}(S^2)\right)}

\newcommand{\conflie}[1]{\mathcal{L}_{#1,conf}}

\newcommand{\AH}[3]{{\mathcal{A}_{H}}^{(#1,#2)}_{#3}}
\newcommand{\AC}[3]{{\mathcal{A}_{C}}^{(#1,#2)}_{#3}}
\newcommand{\AHC}[3]{\mathcal{A}^{(#1,#2)}_{#3}}

\numberwithin{equation}{section}
\newtheorem{thm}{Theorem}[section]
\newtheorem{lem}[thm]{Lemma}
\newtheorem{prop}[thm]{Proposition} 
\newtheorem*{prop*}{Proposition}
\newtheorem {cor}[thm]{Corollary}  

\newtheorem*{question*}{Question}

\newtheorem*{mainthm}{Main Theorem}{\bf}{\it}
\newtheorem*{redthm}{Reduction Theorem}{\bf}{\it}

\theoremstyle{definition}
\newtheorem{defn}[thm]{Definition}
\newtheorem{definition}{Definition}
 
\theoremstyle{remark}
 
\newtheorem{remark}[thm]{Remark}

\title{Local Well-posedness of the Bartnik Static Extension Problem near Schwarzschild spheres}

\author{Ahmed Ellithy\thanks{Department of Mathematics, University of Toronto. Email: \texttt{ahmed.ellithy@mail.utoronto.ca}}}
\date{}

\begin{document}

\maketitle
\begin{abstract}
 \noindent We establish the local well-posedness of the Bartnik static metric extension problem for arbitrary Bartnik data that perturb that of any sphere in a Schwarzschild $\{t=0\}$ slice. Our result in particular includes spheres with arbitrary small mean curvature. We introduce a new framework to this extension problem by formulating the governing equations in a geodesic gauge, which reduce to a coupled system of elliptic and transport equations. Since standard function spaces for elliptic PDEs are unsuitable for transport equations, we use certain spaces of Bochner-measurable functions traditionally used to study evolution equations. In the process, we establish existence and uniqueness results for elliptic boundary value problems in such spaces in which the elliptic equations are treated as evolutionary equations, and solvability is demonstrated using rigorous energy estimates. The precise nature of the expected difficulty of solving the Bartnik extension problem when the mean curvature is very small is identified and suitably treated in our analysis. 
 
\end{abstract}
\tableofcontents

\section{Introduction}

We consider the Bartnik static metric  extension problem, which 
originates in the Bartnik mass-minimization 
problem, \cite{Bartnik2, Bartnik3}: In the latter, one considers a topological $3$-ball 
$(B, g)$ 
equipped with a Riemannian 3-metric 
of positive scalar curvature. A natural example of such a metric arises on 
any compact  space-like maximal hypersurface $\Omega$ (with boundary) 
in a $(3+1)$-spacetime $(M, g)$ 
that satisfies the dominant energy condition, where $g$ is the restriction of the 
space-time metric $g$ to $\Omega$.

One wishes to assign a notion of mass to $(\Omega, g)$; in fact ideally, \cite{Bartnik3},  
the notion of mass should depend just on the restriction of $g$ to $\partial\Omega$, 
the second fundamental form of $\Omega\subset M$ at 
$\partial\Omega$, 
as well as the second fundamental form of $\partial\Omega$ inside $\Omega$. 

Bartnik's definition, see \cite{Bartnik3, corvino2}, considers such data on 
$\partial\Omega$ and associates to it a class $\mathcal{PM}$ 
of admissible asymptotically flat extensions $(M_{\rm ext}, g_{\rm ext})$,
and seeks to minimize the ADM
mass among all such extensions. There are many possibilities on how to define
the space $\mathcal PM$ of extensions, see \cite{Bartnik3}. The most ``minimal''  
requirements are that the 3-metrics $g_{\mathrm ext}\in\mathcal PM$ 
should be of positive scalar curvature, the metrics on $\partial M_{\rm ext}$ induced from the two sides
$\Omega$ (interior side) and   $M_{\rm ext}$ (exterior side) should \emph{match}:  
$g_{\rm ext}|_{\partial M_{\rm ext}}= g|_{\partial\Omega}$;  moreover, the mean curvature 
$H_{\rm ext}$ of $\partial M_{\rm ext}$ in $M_{\rm ext}$ should agree with the mean curvature 
$H_{\rm int}$
of $\partial M_{\rm ext}=\partial\Omega$ with respect to the interior metric  $g$ over $\Omega$.
Additional requirements, such as the non-existence of closed minimal surfaces in
the extension  $(M_{\rm ext}, g_{\rm ext})$
are very natural (see \cite{Bartnik3}) and are 
also 
frequently imposed. Once the class of admissible extensions has been chosen the Bartnik Mass
is defined to be the infimum of the ADM masses, among all admissible extensions.

 
 An important feature of the Bartnik 
 mass 
 is the result of Corvino, \cite{corvino1, corvino2} that  if 
this infimum is attained for some (asymptotically flat) 
metric $\mathfrak{g}$, on a manifold $M_{\rm ext}$ with 
$\partial M_{\rm ext}= \partial \Omega$,
then this extension $\mathfrak{g}$  
must satisfy the system of equations: 
  
  \begin{equation}
\label{static equations}    
\Delta_{\mathfrak{g}} f=0, {\quad}
\operatorname{Ric}_{\mathfrak{g}}=f^{-1}\operatorname{Hess}_{\mathfrak{g}}(f), 
  \end{equation}
  as well as the two imposed requirements 
  \begin{equation}
      \label{bdry condns}
      \mathfrak{g}|_{\partial M_{\rm ext}}= g|_{\partial\Omega}, {\quad} H_{\rm ext}(\mathfrak{g})= 
      H_{\rm int}.   
  \end{equation}
A solution to the system \eqref{static equations} implies that the metric 
${\bf g}=-f^2 dt^2+\mathfrak{g}$ on 
$M_{\rm ext}\times\mathbb{R}$ would satisfy the Einstein Vacuum
equations, and also be static, in the sense that
${\mathcal L}_{\partial_t}{\bf g}=0$. 
\vv

\begin{remark}
 We note further that if the Bartnik minimizer exists, it is known--see
 \cite{Bartnik3}--that the 
metric $g_{\rm total}$ defined over $M_{\rm total}= \Omega\bigcup M_{\rm ext}$ by joining $g$ with $\mathfrak{g}$ 
across $\partial M_{\rm ext}$
is generically expected to be merely Lipschitz 
across
the joining boundary $\partial M_{\rm ext}$: 
The traceless parts $\hat{K}|_{g}, \hat{K}|_{\mathfrak{g}}$ of the second fundamental 
forms  ${K}|_{g}, {K}|_{\mathfrak{g}}$
induced on  $\partial\Omega=\partial M_{\rm ext}$ from the two sides 
$(\Omega,g), (M_{\rm ext},\mathfrak{g})$ are 
generically expected to not match.
\end{remark}

\begin{remark}
In this paper, we define the mean curvature of a surface in a 3-manifold as half the trace of the second fundamental form of the surface. In particular, the mean curvature of the round unit sphere in $\R^3$ is $1$. 
\end{remark}

In view of the result of \cite{corvino1}
the question of the attainment of the Bartnik mass 
leads to the Bartnik static extension problem with data supported on a 2-sphere: 

\begin{question*}[Bartnik static metric extension problem]\label{def stat ext}
Consider a Riemannian 2-sphere $(S^2, \gamma)$  equipped with a function $H$
over $S^2$. We consider a (topological) manifold 
$M=\mathbb{R}^3\setminus B$ and  seek an asymptotically flat metric $\mathfrak{g}$
over $M$ which satisfies:
\begin{itemize}
\item $\mathfrak{g}|_{\partial M}=\gamma$, 
\item the mean curvature $H_{\rm ext}$  of 
$\partial M$ relative to $\mathfrak{g}$ equals $H$, and: 
\item There exists a 
positive function $f$ over $M$ with $f(x)\to 1 $ as $|x|\to \infty$ on $M$ so that the pair $(\mathfrak{g},f)$ satisfy the system of equations \eqref{static equations}. 
\end{itemize}
\end{question*}

\begin{definition}
 The system in \eqref{static equations} will be called the \emph{static vacuum equations}. The pair of prescribed data over $S^2$ (the metric $\gamma$ and the 
putative mean curvature function $H$) will be called 
\emph{Bartnik data}. A solution $(\mathfrak{g},f)$ to \eqref{static equations} to this prescribed data with $\mathfrak{g}$ being asymptotically
flat and $f$ going to $1$ at infinity will be called a \emph{static vacuum extension} with Bartnik data $(\gamma, H)$.  
\end{definition}

\vv

Two important examples of static vacuum extensions are: 

\begin{enumerate}
\item The Euclidean solution $(g_{euc}, 1)$ on $\R^3\setminus B_1$ with Bartnik data $(\gamma_{\mathbb{S}^2}, 1)$, where $\gamma_{\mathbb{S}^2}$ is the round metric on $S^2$.  
\item The Riemannian Schwarzschild solution $(\gsc, f_{sc})$ with mass $m_0$ on $\R^3 \setminus B_{r_0}$ and Bartnik data $(r_0^2 \gamma_{\mathbb{S}^2}, \frac{f_{sc}(r_0)}{r_0})$, where $r_0>2m_0$ and 
\[\gsc = f_{sc}^{-2}dr^2+ r^2 \gamma_{\mathbb{S}^2}, \qquad f_{sc} = \sqrt{1-\frac{2m_0}{r}} \]
\end{enumerate}


\vv

Since the static vacuum equations are highly nonlinear, one first hopes to achieve a local well-posedness result near arbitrary solutions. In fact, Anderson and Khuri in \cite{anderson-khuri} prove, by means of counterexamples, that global well-posedness does not hold (see also \cite{anderson2023}) . Nonetheless, there has been significant progress on establishing local well-posedness results. Miao in \cite{miao} confirmed that the extension problem is locally well-posed near Euclidean Bartnik data on the unit sphere under a triple reflectional symmetry assumption. This symmetry assumption was later removed by Anderson in \cite{anderson-local-exist} with the result generalized by Huang and An in \cite{an-huang} and \cite{an-huang3} for a large range of connected embedded surfaces in Euclidean $\R^3$. 

\vv

Subsequently, Huang and An in \cite{an-huang2} introduced a general criterion for local well-posedness near a given solution, which hinges on the triviality of the kernel of a particular operator. They identified a class of static vacuum extensions they call “static regular”, characterized by the linearized operator having a trivial kernel, as sufficient conditions for local well-posedness. They showed that static regularity is, in some sense, generic for smooth hyper surfaces which are already inside a static vacuum extension. However since this relies in a very essential way on the data already lying in the interior of a given solution (which must be analytic), this result does not guarantee genericity in any sense in the space of smooth Bartnik data. Their findings in particular implies that for any given $m_0>0$ and $\epsilon>0$, the set of radii $r_0\geq 2m_0+\epsilon$ for which the Schwarzschild manifold $\R^3 \setminus B_{r_0}$ with mass $m_0$ is static regular forms an open dense subset of $[2m_0+\epsilon, \infty)$.

\vv

\subsubsection*{Summary of Main Results}

Our main result in this paper is establishing local well-posedness for perturbations of {\textit every} Schwarzschild solution, hence strengthening Huang and An's result in \cite{an-huang2}. We present a new approach to this problem that can be applied to similar extension problems. In this approach, we write the putative solution $(\mathfrak{g}, f)$ with respect to a geodesic gauge, which was not used before for this problem. One benefit of this gauge is that the connection coefficients of the desired solution $\mathfrak{g}$ can be 
linked to $f$ by \emph{ordinary} differential equations, where $f$ provides 
the forcing terms. More precisely, we will be considering the metrics $g:=f^2\cdot \mathfrak{g}$, whose Ricci curvature must then satisfy: ${\mathrm Ric}_{ij}(g)= 2\nabla_i u\otimes \nabla_j u$ (with $u=\ln f$), and we will reduce the extension problem in \eqref{static equations} to an elliptic equation on $u$ coupled with Riccatti equations on the second fundamental form of $g$, with $u$ providing the forcing term, and constraint equations on the boundary coming from the contracted Gauss and Codazzi equations. We rigorously establish estimates for the linearized operator and its inverse, showing that the linearization of the reduced equations is an isomorphism on appropriate Banach spaces. We then invoke the implicit function theorem on Banach manifolds to conclude local well-posedness.

\vv

An interesting remark concerns the expected difficulty of solving Bartnik's extension problem when the Schwarzschild sphere is very close to the horizon, in which the mean curvature is positive but very small. This difficulty is anticipated by the black hole uniqueness theorem (see \cite{blackholeunique}), which in particular implies the following: for surfaces with zero mean curvature, the existence of static vacuum extensions fails unless the surface is a round sphere, in which case the Schwarzschild exteriors are the \textit{only} possible extensions. Therefore, one expects that the space of allowed perturbations of the Schwarzschild spheres $S_r$ must be shrinking as $r$ goes to $2m_0$. This also suggests that solving the linearized problem should be progressively harder as $r \to 2m_0$. We do capture this difficulty in our analysis and resolve it (see proposition \ref{prop-nonlocalpde-unique}), showing the solvability of Bartnik's extension problem near all spheres $S_r$, $r>2m_0$. 

\vv

The choice of gauge influences which Banach spaces are most appropriate to use. Due to our choice of gauge, the equations in \eqref{static equations} reduce to an elliptic PDE coupled with transport equations. Consequently, the standard spaces used for elliptic PDEs, such as weighted Sobolev and Hölder spaces, are not appropriate as they do not provide the correct setting to solve transport equations. Instead, we use spaces of Bochner-measurable functions that are traditionally used as the setting to study hyperbolic and parabolic PDEs (see \cite{evans}). More specifically, the spaces we use for $u$ are  $\AC{2}{k}{\delta}(M)$ and $\AH{2}{k}{\delta}(M)$ defined by (see definition \ref{spaces for u})

\begin{equation*} u \in \AH{2}{k}{\delta}(M) \quad \Longleftrightarrow \quad \begin{cases} \begin{aligned}  &u \in \LtoHr{\delta}{k}\\&\partial_r u \in \LtoHr{\delta-1}{k-1} \\&\partial_r^2 u \in \LtoHr{\delta-2}{k-2} \end{aligned} \end{cases} \end{equation*}
\begin{equation*} u \in \AC{2}{k}{\delta}(M) \quad \Longleftrightarrow \quad \begin{cases} \begin{aligned}  &u \in \CtoHr{0}{\delta}{k}\\&\partial_r u \in \CtoHr{0}{\delta-1}{k-1} \\&\partial_r^2 u \in \CtoHr{0}{\delta-2}{k-2} \end{aligned} \end{cases} \end{equation*}
where $r_0>0$, $k\geq 2$, and $\delta \in (-1,-\frac{1}{2})$ is a weight introduced appropriately in the norms of the above spaces to control the decay at infinity. These spaces are not traditionally used to study elliptic PDEs. In this paper, we establish solvability of a certain elliptic problem in the above spaces. More specifically, defining the operator $Q: u \mapsto (\Delta_g u, u|_{\partial M})$ with respect to a certain asymptotically flat metric $g$ on $M = \R^3 \setminus B_{r_0}$, we demonstrate that (see chapter \ref{elliptic-chapter} )

\[Q:\AH{2}{k}{\delta}(M) \to \LtoHr{\delta-2}{k-2}  \times H^{k-1/2}(\partial M) \qquad \text{is an isomorphism}\]
\[Q:\AC{2}{k}{\delta}(M) \to \CtoHr{0}{\delta-2}{k-2}  \times H^{k}(\partial M) \qquad \text{is an isomorphism} \]

This result can be generalized to arbitrary asymptotically flat metrics and more general elliptic boundary value problems, thereby establishing the solvability of such problems in the above spaces. 

\vv

An interesting comparison we can make to the above is the study of elliptic boundary value problems in $C^k$ spaces. It is well known that there is no general existence theorem for elliptic boundary value problems in $C^k(M)$ (see \cite{trudinger} problem 4.9 for a counterexample). In \cite{trudinger}, the authors demonstrate via the celebrated Schauder and Calderon-Zygmund estimates that Hölder spaces $C^{k,\alpha}(\Omega)$ and Sobolev spaces $H^k(\Omega)$ on a bounded open set $\Omega$, instead of $C^k(\Omega)$, have sufficiently nice properties allowing for general existence theorems for elliptic boundary value problems. Modification of those spaces by including wights generalizes these existence results to unbounded spaces (see for example \cite{Bartnik1}). In this paper, we establish an existence theorem in $\CtoHr{0}{\delta}{k}$ spaces, a mix of both Hölder and Sobolev spaces. Our work readily implies similar existence results in the spaces $C^0([a,b] : H^k(S^2))$ when the domain is bounded. 

\vv

\subsubsection*{ Comparison with the Framework in \cite{an-huang2} and \cite{anderson-khuri}}

Given the two different approaches to the Bartnik static metric extension problem, it is of interest to describe the key differences between our framework and the one developed in \cite{an-huang2} and \cite{anderson-khuri}. In the latter, the extension problem is formulated as an elliptic boundary value problem in a certain gauge called the Bianchi-harmonic gauge. In contrast, our formalism is based on a simpler equivalent conformal system (see equation \eqref{conformal-problem}) and leverages the fact that the Ricci curvature determines the full Riemann curvature tensor in 3 dimensions; the extension problem is then presented, in a geodesic gauge, as an elliptic boundary value problem on the lapse function coupled with transport equations on the second fundamental form of the leaves of the equidistant foliation. Our framework has several advantages:

\begin{itemize}
\item Owing to the vanishing of the Weyl tensor in 3 dimensions, the complicated geometric equation $f \mathrm{Ric} = \mathrm{Hess} f$ simplifies to transport equations governing the evolution of the second fundamental form (see equations \eqref{redeq2} - \eqref{redeq5}). This allows for a more tractable analysis of the linearized problem, enabling us to establish local well-posedness for perturbations of \textit{every} coordinate Schwarzschild sphere, which strengthens the results in \cite{an-huang2}. This more accessible analysis may also open the door to addressing more general extension problems, such as the generalized Bartnik metric extension problem for non-time-symmetric initial data sets, where one seeks a \textit{stationary} (rather than static) extension. 

\item The linearized problem in our setting reduces to a novel nonlocal elliptic system (see equation \eqref{Psc}) that seems to be fundamental to this problem. This new perspective may serve as a useful tool for studying the global solvability of the problem and could have further implications for the theory of quasi-local mass in general relativity.

\item Our framework is flexible and can be adapted to use gauge foliations other than the equidistant foliation considered here. For instance, one can apply the results in \cite{inverse-mean} to reformulate the geometric equations in terms of the foliation generated by inverse mean curvature flow. This freedom may prove useful in analyzing the global solvability of the extension problem and exploring new geometric perspectives. 

\item There are intrinsic obstructions to solvability that seem to be fundamental to this problem which, in our framework, manifest in the contracted Codazzi equation (see the introduction to chapter \ref{proof section}). We are then able to relate this space in a natural way to the conformal structure of the $2$-sphere, providing an insightful geometric interpretation of these peculiar apparent obstructions. We circumvent this difficulty by introducing an artificial object to the definition of a solution to our problem (see section \ref{artificial-section}). A similar argument appears in \cite{an-huang} and \cite{an-huang2}.  

\item The solvability of certain elliptic boundary value problems play a central role in this extension problem. While \cite{an-huang2} relies on standard elliptic theory in weighted Hölder spaces (see lemma 2.3 in \cite{an-huang} and lemma 3.3 in \cite{an-huang2}), our framework necessitates working in nonstandard function spaces, for which classical results do not apply. As part of this paper, we establish new solvability results in these settings (see Chapter~\ref{elliptic-chapter}). The main result, Theorem~\ref{elliptic thm}, may be of independent interest, with potential applications to other extension problems involving coupled elliptic and non-elliptic (e.g., transport, parabolic, or hyperbolic) PDE systems.

\end{itemize}

\begin{remark}
We note that a limitation of our approach is that it applies specifically to 3 dimensional manifolds, whereas the framework developed in \cite{an-huang2} extends to all higher dimensions. In particular, in dimensions larger than 3, it is not yet known whether the extension problem is locally well-posed near \textit{every} coordinate Schwarzschild sphere. This reflects a key strength of the approach by An and Huang as they are able to establish a local well-posedness result for all but possibly a meager set of radii in higher dimensions.

\end{remark}

\subsubsection*{ Comparison with the Proof in \cite{AhmedAn}}

As mentioned above, Huang and An in \cite{an-huang2} reduce the local well-posedness of the Bartnik problem to showing that a certain operator has trivial kernel (see theorem 5.1 in \cite{an-huang2}). The analysis of the current paper, namely proposition \ref{prop-nonlocalpde-unique}, readily implies the kernel is trivial for coordinate Schwarzschild spheres of \textit{every} radius in $(2m_0,\infty)$, thus providing another proof of the same result. This alternative proof is outlined in \cite{AhmedAn} by An, Huang, Alexakis and the present author. The key difference between the current paper and \cite{AhmedAn} is essentially how the surjectivity of the corresponding operator is proven. In \cite{AhmedAn}, the surjectivity follows from lemma 3.10 in \cite{an-huang2}, which is based on proposition 3.1 in \cite{anderson-khuri}, asserting that the operator is Fredholm of index $0$. In the current paper, surjectivity is shown explicitly in the framework we use, owing to the more tractable analysis of our approach.


\vv



\subsubsection*{Acknowledgements} The author is grateful to Spyros Alexakis for his meticulous verification of the arguments presented in this paper and for his numerous invaluable suggestions that significantly enhanced the clarity and quality of the writing. The author is also grateful to Yakov Shlapentokh-Rothman for many insightful and fruitful discussions on this work.


\section{Preliminaries}
Let $M:= \R^3\setminus B_{n \cdot m_0}$ where $n> 2$ and $m_0 >0$. Denote by $g_{euc}$ the Euclidean metric on $M$ and by $\gamma_{\mathbb{S}^2}$ the round metric on the unit sphere $S^2$. 

\subsection{Properties of Static Vacuum Extensions} \label{properties of static vacuum extensions}

In this section, we will discuss some decay and regularity properties of static vacuum extensions and demonstrate that they can be written in the geodesic gauge. More precisely, we will show that given a static vacuum extension $(\mathfrak{g}, f)$, we can globally write the metric $g:=f^{-2} \mathfrak{g}$ in geodesic coordinates so that $\mathfrak{g}$ takes the form 

\[ \mathfrak{g} = f^{-2} dr^2 + {\gamma_{\mathfrak{g}}}_r \]

where $r$ is the distance function from the boundary with respect to $g$ and ${\gamma_{\mathfrak{g}}}_r$ is the induced metric on the level sets of $r$. Note that this form is directly observable in the Schwarzschild solutions $(\gsc, f_{sc})$ as $\gsc$ is given by  
\[\gsc = f_{sc}^{-2} dr^2 + r^2 \gamma_{\mathbb{S}^2}\]

\vvv

\begin{defn} \label{af def}
 Let $\eta >0$. A $\mathcal{C}^2$ metric $\mathfrak{g}$ over $M$  is asymptotically flat of order $\eta >0$ if there exists a coordinate system $(x^1,x^2,x^3)$ near infinity in which the metric satisfies

  \begin{itemize}
 \item $\mathfrak{g}_{ij} -\delta_{ij} = \mathcal{O}(|x|^{-\eta})$
 \item $\partial_k \mathfrak{g}_{ij} = \mathcal{O}(|x|^{-\eta-1})$
 \item $\partial_l\partial_k \mathfrak{g}_{ij} = \mathcal{O}(|x|^{-\eta-2})$
 \end{itemize}
 
\noindent where $\partial_k := \dd{x^k}$ and $|x|=\sqrt{|x^1|^2+|x^2|^2+|x^3|^2}$. For conciseness, we will write 
$\mathfrak{g}_{ij} = \delta_{ij} + \mathcal{O}_2(|x|^{-\eta})$ if the above conditions are satisfied.

\end{defn}

\begin{defn} For a metric $\mathfrak{g}$ and a positive function $f$, we say that a pair $(\mathfrak{g}, f)$ is a strongly asymptotically flat if $\mathfrak{g}$ admits a coordinate system $(x^1,x^2,x^3)$ near infinity in which 

\begin{equation}
\mathfrak{g}_{ij} =  \left( 1+ \frac{2m}{|x|}\right) \delta_{ij} + \mathcal{O}_2(|x|^{-2}), \qquad f = 1 - \frac{m}{|x|} + \mathcal{O}_2(|x|^{-2}) 
\end{equation}
\end{defn}



The Schwarzschild solutions $(\gsc, f_{sc})$ discussed in the introduction are examples of smooth strongly asymptotically flat static vacuum extensions. 

\vv

The system in \eqref{static equations} is equivalent to a lower order system of equations. Letting $g := f^2\mathfrak{g}, u:= \ln f$, a direct computation shows that $(\mathfrak{g}, f)$ solves equation $\eqref{static equations}$ if and only if $(g,u)$ solve 

\begin{equation} \label{conformal-problem}
Ric_{g} = 2du \otimes du, \qquad
\Delta_{g} u = 0. 
\end{equation}

We will call the above equations the conformal static vacuum equations. We will call the pair $(g_{sc}, u_{sc}) := (f_{sc}^2\,\, \mathfrak{g}_{sc}, \ln f_{sc})$ the conformal Schwarzschild solution. 

\vv

By taking advantage of the form that the Ricci curvature takes for $g$, Murchadha in \cite{murchadha} shows that every static vacuum extension $(\mathfrak{g}, f)$ is strongly asymptotically flat and is smooth away from the boundary. For the rest of this section, we will describe how this strong decay and regularity of static vacuum extensions allows us to write the extension problem in a geodesic gauge. 

\vv


\vv

Let $(g,u)$ solve the conformal static vacuum equations in \eqref{conformal-problem}. We wish to write the metric $g$ in geodesic coordinates. Let $r(\cdot) = \text{dist}(\cdot, \partial M) + r_0$ be a shifted distance function from $\partial M$. Due to the compactness of $\partial M$, the function $r$ is smooth on a neighborhood of $\partial M$, with $\partial M$ excluded, and so defines a foliation near $\partial M$ with the leaves being the level sets of $r$. We can then write the metric with respect to this foliation as $dr^2 + \gamma_r$ where $\gamma_r$ is the induced metric on the level sets. This representation of the metric generally does not hold globally and is valid only whenever $r$ is differentiable. However, under a smallness assumption on the Ricci curvature of $g$, it will hold that $r$ is differentiable everywhere on $M\setminus \partial M$ and the metric can be written globally as $dr^2 + \gamma_r$. This follows from the next proposition, which follows from a straightforward adaptation of the argument in proposition 5.01 in \cite{christodoulou}. 

\begin{prop} \label{gauge}
There exists $\tau'=\tau'(n,m_0)>0$ small enough such that the following is true for any $0<\tau<\tau'$.\\
If an asymptotically flat metric $g$ on $M$ of order $\eta>0$ satisfies in Cartesian coordinates
\begin{equation} \label{smallness}
|x|^{\eta} |g - g_{sc}| + |x|^{\eta+1}|\partial g - \partial g_{sc}| + |x|^{\eta+2}| \partial\partial g - \partial\partial g_{sc}| < \tau
\end{equation}
where $|\cdot|$ is with respect to the Euclidean metric $\delta$,  then:   

\begin{enumerate}
\item The affine parameter $r(\cdot) = \text{dist}_g(\partial M,\cdot)+nm_0$ is differentiable everywhere on $M\setminus \partial M$ and defines a global radial foliation with leaves $S_r$ diffeomorphic to $S^2$. Moreover, given a coordinate system $(x^1,x^2,x^3)$ near infinity as described in definition \ref{af def}, $r$ and $|x|$ are comparable in the sense that  
\begin{equation}
    C^{-1} |x| \leq r \leq C |x|
\end{equation}
for some constant $C>0$. 

\item With respect to this foliation, we have 
\begin{equation}  \label{trK-hatK-inf}
trK = \frac{2}{r} + \mathcal{O}_1(r^{-1-\eta}), \qquad |\hat K| = \mathcal{O}_1(r^{-1-\eta})
\end{equation}
where $K = \text{Hess}(r)$ is the second fundamental form on the leaves $S_r$, $trK$ is the trace of $K$, and $\hat K$ is the traceless part of $K$. 
\item There exists a unique diffeomorphism $\Phi: M \to [nm_0, \infty) \times S^2 $ such that $\Phi|_{\partial M} = Id_{S^2}$, $r (\cdot) = \pi_r \circ \Phi (\cdot)$ where $\pi_r$ is the projection onto the first coordinate,  and $\Phi_* g = dr^2 + {\gamma_g}_r$ where ${\gamma_g}_r$ is the push forward of the induced metric on $S_r$. 
\end{enumerate}
\end{prop}

This allows us to globally express the static vacuum extension $(\mathfrak{g}, f)$ in geodesic coordinates as follows: 
\[\mathfrak{g} = f^{-2} dr^2 + {\gamma_{\mathfrak{g}}}_r, \quad \text{where ${\gamma_{\mathfrak{g}}}_r$ is the induced metric on $S_r$}\]
and, hence, justifies the space of metrics that we will be working in (see definition \ref{space-of-metrics} in the next section).

\subsection{Function Spaces}

In this section, we define the function spaces that we will be using. Fix $k \in \Z_{\geq0}$ and $\delta \in \R$. From here onwards, we will identify $M$ with the space $[nm_0,\infty) \times S^2$.

\begin{defn} 
 We define the weighted Sobolev space $H^k_{\delta} (M)$ with weight $\delta$ to be the space of all functions $u$ in $H^k_{loc}(M)$ such that $\norm{u}_{k,\delta}< \infty$ respectively, where 
\begin{equation} \label{norm1}
 \norm{u}_{k, \delta} = \sum^k_{l = 0} \left\{ 
\int_M \left( |D^l u| \cdot r^{l -\delta} \right)^2 r^{-3} dV \right\}^
\frac{1}{2}
\end{equation}
where $r = |x|$, $D$ is the connection with respect to the Euclidean metric on $M$, and $dV$ is the Euclidean volume form on $M$. We will also denote the space $H^k_{\delta}(M)$ by $L^2_{\delta}(M)$ when $k=0$.\\
\end{defn}

\begin{defn} \label{def of Xkdelta}
\noindent We define the space $\mathcal{X}^k_{\delta}(M)$ to be the space of vector fields $X$ on $M$ with components $X^i := X(x^i)$ in $H^k_{\delta}(M)$, where $(x^1,x^2,x^3)$ is the standard cartesian coordinates. The norm we use is 
\begin{equation}
\norm{X}_{k,\delta}:= \sum_{l=0}^k \norm{|D^l X|}_{0,\delta-l}
\end{equation}
\end{defn}

\vv


\vv

\begin{defn} 
Let $H^k(S^2)$ be the usual $L^2$ space, when $k=0$, and Sobolov space, when $k\geq 1$, on $(S^2, \gamma_{\mathbb{S}^2})$. Let $\mathcal{M}^k(S^2)$ and $\mathcal{H}^k(S^2)$ be the space of metrics on $S^2$ and symmetric tensors on $S^2$, respectively, with components in $H^k(S^2)$. The norm we will use is as follows:
\begin{equation}
\norm{h}^2_{\mathcal{H}^k(S^2)} := \sum_{l=0}^k \norm{|\slashed D^l h|}^2_{L^2(S^2)}
\end{equation}
where $\slashed D$ is the covariant derivative on $S^2$ with respect $\gamma_{\mathbb{S}^2}$. \\

\noindent Let $\Omega^k(S^2)$ be the space of 1-forms on $S^2$ with components in $H^k(S^2)$. The norm used on this space is as follows:
\begin{equation}
\norm{\omega}^2_{\Omega^k(S^2)} := \sum_{l=0}^k \norm{|\slashed D^l \omega|}^2_{L^2(S^2)} 
\end{equation} 
\end{defn}
\vv

\begin{defn} Let $t \in \Z_{\geq 0}$. We define the space $H^t_{\delta}\left([nm_0,\infty); H^k(S^2)\right)$ to be the space of functions $u$ in $H^t_{loc}\left([nm_0,\infty); H^k(S^2)\right)$ such that $\normmH{u}{t}{k}{\delta} <\infty$, where 

\begin{equation}
\normmH{u}{t}{k}{\delta}^2:= \sum_{t'=0}^t\int_{nm_0}^{\infty} r^{-2\delta-1+2t'}   \norm{ \partial_r^{(t')} u(r)}^2_{H^k(S^2)} dr
\end{equation}
We also define the space $\CtoH{t}{\delta}{k}$ to be the space of continuous $H^k(S^2)$-valued functions $u$ on $[nm_0,\infty)$ such that $\normmC{u}{t}{k}{\delta} <\infty$, where 

\begin{equation}
\normmC{u}{t}{k}{\delta}^2:= \sum_{t'=0}^t \sup_{r\geq nm_0} \left( r^{-2\delta+2t'}   \norm{ \partial_r^{(t')} u(r)}^2_{H^k(S^2)}\right)
\end{equation}
\end{defn}
\vv

\vv

We then define the space $H^t_{\delta}\left([nm_0,\infty); \mathcal{M}^k(S^2)\right)$ and $H^t_{\delta}\left([nm_0,\infty); \mathcal{H}^k(S^2)\right)$ similarly to the above with norm
\begin{equation}
\normmH{h}{t}{k}{\delta}^2 := \sum_{t'=0}^t\int_{nm_0}^{\infty} r^{-2\delta-1+2t'} \norm{\partial_r^{(t')} h(r)}^2_{\mathcal{H}^k(S^2)} dr
\end{equation}

\vv

\begin{defn}   \label{space-of-metrics}
Define $\mathcal{M}^k_{\delta}(M)$ to be the space of metrics on $[nm_0,\infty) \times S^2$ of the form $dr^2+ g(r)$ where $g(r) = r^2(\gamma_{\infty} + h(r))$, $\gamma_{\infty} \in \mathcal{M}^k(S^2)$, and $h \in \HtoHH{2}{\delta}{k}$. The space $\mathcal{M}^k_{\delta}(M)$ can be naturally identified with an open subset of the Banach space $\mathcal{H}^k(S^2) \oplus \HtoHH{2}{\delta}{k}$. This makes $\mathcal{M}^k_{\delta}(M)$ an open Banach submanifold of $\mathcal{H}^k(S^2) \oplus \HtoHH{2}{\delta}{k}$ and, in particular, a Banach manifold. Given $g_0 \in \mathcal{M}^k_{\delta}(M)$, the tangent space $T_{g_0}M^k_{\delta}$ is isomorphic to the space of tensors $\tilde g$ of the form $\tilde g = r^2(\tilde \gamma_{\infty}+ \tilde h(r))$, where $\tilde \gamma_{\infty} \in \mathcal{H}^k(S^2)$ and $\tilde h \in  H^2_{\delta}\left([nm_0,\infty); \mathcal{H}^k(S^2)\right)$, equipped with the norm 

\begin{equation}
\norm{\tilde g}_{\mathcal{M}^k_{\delta}} := \norm{\tilde \gamma_{\infty}}_{\mathcal{H}^k(S^2)} + \normmH{\tilde h}{2}{k}{\delta} 
\end{equation}

\end{defn}

\vv
\begin{defn} \label{spaces for u}
Let $t\geq 0$. Denote by $\AH{t}{k}{\delta}(M)$ and $\AC{t}{k}{\delta}(M)$ the spaces

\begin{equation}
\AH{t}{k}{\delta}(M) := \bigcap_{t'=0}^t \HtoH{t'}{\delta}{k-t'} , \qquad \AC{t}{k}{\delta}(M) :=  \bigcap_{t'=0}^t \CtoH{t'}{\delta}{k-t'}
\end{equation}
equipped with the norms

\begin{equation}
\norm{u}^2_{\AH{t}{k}{\delta}} :=\max_{0\leq t'\leq t}  \normmH{u}{t'}{k-t'}{\delta}^2 , \qquad \norm{u}^2_{\AC{t}{k}{\delta}} := \max_{0\leq t'\leq t}  \normmC{u}{t'}{k-t'}{\delta}^2 
\end{equation}

Note that 
\[ u \in \AH{t}{k}{\delta}(M) \quad \Longleftrightarrow \quad \text{for every $0\leq t'\leq t$,   } \,\,  \partial_r^{(t')} u \in \LtoH{\delta-t'}{k-t'} \]
\[ u \in \AC{t}{k}{\delta}(M) \quad \Longleftrightarrow \quad \text{for every $0\leq t'\leq t$, } \,\, \partial_r^{(t')} u \in \CtoH{0}{\delta-t'}{k-t'} \]
Denote the intersection of these spaces by $\AHC{t}{k}{\delta}(M)$ defined by
\begin{equation*}
\AHC{t}{k}{\delta}(M) := \AH{t}{k}{\delta}(M) \bigcap \AC{t}{k}{\delta}(M)
\end{equation*}
equipped with the norm 

\begin{equation}
\norm{u}_{\mathcal{A}^{(t,k)}_{\delta}}^2:= \max_{0\leq t'\leq t} ( \normmH{u}{t'}{k-t'}{\delta}^2 + \normmC{u}{t'}{k-t'}{\delta}^2 )
\end{equation}

\end{defn}

\vvv

In the next proposition, we list some important results regarding the spaces we defined that will be repeatedly used in the rest of the paper.

\begin{prop} \label{prop-spaces} \
\begin{enumerate}[label=\textbf{(\alph*)}]

\item Let $k\geq 0$, $t\geq 1$ and $\delta <0$. Every function $u \in \HtoH{t}{\delta}{k}$ has a representative in $ \CtoH{t-1}{loc}{k}$, which will also be denoted by $u$. Furthermore, there exists a constant $C>0$ such that for every $u \in \HtoH{t}{\delta}{k}$ 
\begin{equation}
\normmC{u}{t-1}{k}{\delta} \leq C\normmH{u}{t}{k}{\delta}
\end{equation}
If in addition $k\geq 2$, then $ \partial_r^{(t')}u(r) \in C^{k-2}(S^2)$ for every $r\in [nm_0, \infty)$ and $0\leq t'\leq t-1$. Also, for every $0\leq l \leq k-2$ and $0\leq t'\leq t-1$, 

\begin{equation}
|\slashed D^l \partial_r^{(t')} u| = o(r^{\delta-l - t'}) \quad \text{as $r \to \infty$}
\end{equation}
\item Let $dr^2+g(r) \in \mathcal{M}^k_{\delta}(M)$ with $g(r) = r^2 (\gamma_{\infty} + h(r))$. 


Then with respect to the foliation defined by the level sets of $r$, the trace and traceless part of the fundamental form satisfy

\begin{equation} \label{prop-space-trK-K}
\left|trK - \frac{2}{r}\right| \in \HtoH{1}{\delta-1}{k}, \qquad \hat K \in \HtoHH{1}{\delta-1}{k}
\end{equation}

Furthermore, the metric $dr^2+g(r)$ is asymptotically flat if and only if $\gamma_{\infty}$ is of constant curvature $1$. 
\item Let $k_1<k_2$, $t_1<t_2$, and $\delta_2 < \delta_1<0$. Then the space $\HtoH{t_2}{\delta_2}{k_2}$ is compactly embedded in $\HtoH{t_1}{\delta_1}{k_1}$. Furthermore, the space $\CtoH{t_2}{\delta_2}{k_2}$ is compactly embedded in $\CtoH{t_1}{\delta_1}{k_1}$.

\item Let $k\geq 1$ and $\delta\in \R$. Suppose a function $ u$ satisfies
\[  u \in \LtoH{\delta}{k}(M), \quad \partial_r u \in \LtoH{\delta-1}{k-1}(M) \]
Then for every $r\in [nm_0,\infty)$, 
\[u(r) \in H^{k-1/2}(S^2)\]
\end{enumerate}
\end{prop}

\begin{proof}
$(a)$, $(b)$ and $(d)$ follow immediately from standard results on Sobolev spaces (see \cite{Bartnik1} theorem 1.2 and lemma 1.4, and \cite{evans} section 5.9). 

\vv

We focus on proving $(b)$. The metric $g(r)$ evolves according to the equation

\begin{equation}
\partial_r g(r) = trK \, g(r) + 2 \hat K
\end{equation}

where $K = \text{Hess}(r)$ is the second fundamental form on the leaves $S_r$, $trK$ is the trace of $K$, and $\hat K$ is the traceless part of $K$. Since $g(r) = r^2(\gamma_{\infty} + h(r))$, it follows that

\begin{equation}
trK = \frac{2}{r} +  \frac{1}{2}tr_{r^{-2}g(r)} (\partial_r h(r)), \quad 2\hat K = r^2\Big(\partial_r h(r) - tr_{ g(r)} (\partial_r h(r))g(r) \Big)
\end{equation}

This directly implies equation \eqref{prop-space-trK-K}

In view of the Gauss and Codazzi equations, we get 

\begin{equation}
R = 2Ric(\dd{r},\dd{r}) + R_{S_r} -\frac{1}{2} trK^2 + |\hat K|^2, \qquad Ric(\dd{r}, \dd{r}) = -\partial_r trK - \frac{1}{2} (trK)^2- |\hat K|^2
\end{equation} 

It follows immediately using equation \eqref{prop-space-trK-K} and the fact that the scalar curvature $R_{\gamma_{\infty}}$ of $(S^2, \gamma_{\infty})$ is the limit of $r^2 R_{S_r}$ as $r$ goes to infinity that $R$ decays faster than $r^{-2}$ if and only if $R_{\gamma_{\infty}} = 2$. We conclude that the metric is asymptotically flat if and only if $\gamma_{\infty}$ is of constant curvature $1$. 

\end{proof}

\subsection{The Main Theorem}

\begin{defn}
In place of the mean curvature $H$, we will work with $\trko := 2H$ for convenience, which represents the trace of the hypothetical second fundamental form on $\partial M$. From this point forward, we will denote the Bartnik data on $\partial M$ by $(\gammao, \frac{1}{2} \trko)$. We will also denote by $(\gammasc, \frac{1}{2} \trksc)$ the Schwarzschild Bartnik data, which is given by 

\begin{equation}
\gamma_{\mathfrak{g_{sc}}} = (nm_0)^2 \gamma_{\mathbb{S}^2}, \quad trK_{\mathfrak{g_{sc}}} = \frac{2 \sqrt{1-\frac{2}{n}}}{n m_0}
\end{equation}

\end{defn}

The statement of the main theorem is as follows. 

\begin{mainthm}
Let $M:= \R^3\setminus B_{nm_0}$ where $m_0>0$ and $n>2$. Let $\delta \in (-1,-\frac{1}{2}]$ and $k\geq 5$. There exists a neighbourhood $\mathcal{U}$ of $(\gammasc, \frac{1}{2} \trksc)$ in $\mathcal{M}^{k+1}(\partial M) \times H^k(\partial M) $ and a unique $\mathcal{C}^1$ map $\mathcal{\bf H}: (\gammao, \frac{1}{2} \trko) \mapsto (g,u)$ on $\mathcal{U}$ to $\mathcal{M}^k_{\delta}(M) \times \mathcal{A}^{(2,k+1)}_{\delta}(M)$ in which $(\mathfrak g, f) := (e^{-2u} g, e^u)$ solves the static Einstein vacuum equations with Bartnik data $(\gammao, \frac{1}{2} \trko)$. 
\end{mainthm}

Given Bartnik data $(\gammao, \frac{1}{2} \trko) \in \mathcal{U}$, the pair $(g,u) = \mathcal{\bf H} (\gammao, \frac{1}{2} \trko)$ will then solve the conformal static vacuum equations written out in equation \eqref{conformal-problem}. Due to proposition \eqref{prop-spaces}, $\mathfrak{g}$ and $f$ are $\mathcal{C}^2$ on $M$ and satisfy, in some coordinates $(x^1, x^2, x^3)$, 

\begin{equation}
\mathfrak{g}_{ij} = \delta_{ij} + \mathcal{O}_2(|x|^{\delta}), \quad  f = 1+\mathcal{O}_2(|x|^{\delta})
\end{equation}
Moreover, the discussion in section \ref{properties of static vacuum extensions} implies that $(\mathfrak{g}, f)$ is strongly asymptotically flat and is smooth away from the boundary.

\section{Solvability of Elliptic BVP in $\AH{2}{k}{\delta}(M)$ and $\AC{2}{k}{\delta}(M)$ } \label{elliptic-chapter}

In this chapter, we will establish the well-posedness of the elliptic PDE $\Delta_{g_{sc}} \tilde u = 0$ on $(M,g_{sc})$, subject to Dirichlet boundary conditions, in the function spaces $\AH{2}{k}{\delta}(M)$ and $\AC{2}{k}{\delta}(M)$. Here, $g_{sc}$ is the conformal Schwarzschild metric on $M = \R^3 \setminus B_{n\cdot m_0}$ given by 

\begin{equation}
g_{sc} = dr^2 + r(r-2m_0) \gamma_{\mathbb{S}^2}
\end{equation}
and $n>2$, $m_0>0$. 

More precisely, we will prove the following theorem. 

\begin{thm} \label{elliptic thm}
Define the operator $\mathcal{Q}$ by: 
\[\mathcal{Q} (\tilde  u) := \begin{pmatrix} \Delta_{g_{sc}} \tilde u \lv \\ \left. \tilde u \right|_{\partial M} \end{pmatrix} \]
For $\delta \in (-1,-\frac{1}{2}]$ and $k\geq 1$,  
\[Q:\AH{2}{k+1}{\delta}(M) \to \AH{0}{k-1}{\delta-2}(M) \times H^{k+1/2}(\partial M) \qquad \text{is an isomorphism}\]
and
\[Q:\AC{2}{k+1}{\delta}(M) \to \AC{0}{k-1}{\delta-2}(M)  \times H^{k+1}(\partial M) \qquad \text{is an isomorphism} \]
\end{thm}

\vv

\begin{remark}
In particular, it holds that 
\[Q:\AHC{2}{k+1}{\delta}(M) \to \AHC{0}{k-1}{\delta-2}(M) \times H^{k+1}(\partial M) \qquad \text{is an isomorphism,}  \]
which will be used in section \ref{proofDPhi}. 
\end{remark}

\vv

The map $Q$ can be defined on the space $\AH{2}{k+1}{\delta}(M)$ and $\AC{2}{k+1}{\delta}(M)$ with codomain $\AH{0}{k-1}{\delta-2}(M)\times H^{k+1/2}(\partial M)$ and $\AC{0}{k-1}{\delta-2}(M)\times H^{k+1}(\partial M)$ respectively. Indeed, we deduce directly from the definition of our Banach spaces that for all $u \in \AH{2}{k+1}{\delta}(M)$, 

\[ \Delta_{g_{sc}} \tilde u  = \partial^2_r \tilde u + \frac{2(r-m_0)}{r(r-2m_0)} \partial_r \tilde u + \frac{1}{r(r-2m_0)}\slashed \Delta_{\gamma_{\mathbb{S}^2}} \tilde u (r) \in \LtoH{\delta-2}{k-1}= \AH{0}{k-1}{\delta-2}(M)\] \[ \tilde u(nm_0) \in H^{k+1/2}(S^2)  \qquad \text{(by proposition \ref{prop-spaces} (d))}\]

Similarly, for all $u \in \AC{2}{k+1}{\delta}(M)$, 
\[ \Delta_{g_{sc}} \tilde u  = \partial^2_r \tilde u + \frac{2(r-m_0)}{r(r-2m_0)} \partial_r \tilde u + \frac{1}{r(r-2m_0)}\slashed \Delta_{\gamma_{\mathbb{S}^2}} \tilde u (r) \in \CtoH{0}{\delta-2}{k-1} = \AC{0}{k-1}{\delta-2}(M)\] \[ \tilde u(nm_0) \in H^{k+1}(S^2) \]

We recall the following result from Maxwell in \cite{maxwell}:
\[Q:H^2_{\delta}(M) \to L^{2}_{\delta-2}(M) \times H^{3/2}(\partial M) \qquad \text{is an isomorphism}  \]
 To prove theorem \ref{elliptic thm}, it then suffices to prove the estimates 
\[\norm{\tilde u}_{\AH{2}{k+1}{\delta}} \leq C \norm{Q(\tilde u)}_{ \AH{0}{k-1}{\delta-1} \times H^{k+1/2}(\partial M)}\]
\[\norm{\tilde u}_{\AC{2}{k+1}{\delta}} \leq C \norm{Q(\tilde u)}_{ \AC{0}{k-1}{\delta-1} \times H^{k+1}(\partial M)}\]
for all $\tilde u$ in $\AH{2}{k+1}{\delta}(M)$ and $\AC{2}{k+1}{\delta}(M)$ respectively. These estimates will be the content of the next lemma.

\begin{lem} \label{main-est}
\begin{itemize}
\item There exist a constant $C>0$ such that for any $\tilde u \in \AH{2}{k+1}{\delta}(M)$, the following estimate holds. 

\begin{equation}
\norm{\tilde u}_{\AH{2}{k+1}{\delta}} \leq C \left( \norm{\Delta_{g_{sc}} \tilde u}_{\AH{0}{k-1}{\delta}} + \norm{\tilde u}_{H^{k+1/2}(\partial M)} \right)
\end{equation}
\item There exist a constant $C>0$ such that for any $\tilde u \in \AC{2}{k+1}{\delta}(M)$, the following estimate holds. 

\begin{equation}
\norm{\tilde u}_{\AC{2}{k+1}{\delta}} \leq C \left( \norm{\Delta_{g_{sc}} \tilde u}_{\AC{0}{k-1}{\delta}} + \norm{\tilde u}_{H^{k+1}(\partial M)} \right)
\end{equation}

\end{itemize}
\end{lem}
\begin{proof}
Since $\AHC{2}{k+1}{\delta}(M)$ is dense in both $\AH{2}{k+1}{\delta}(M)$ and $\AC{2}{k+1}{\delta}(M)$, it suffices to prove both estimates for all $\tilde u \in \AHC{2}{k+1}{\delta}(M)$. 

\vv

Let $\tilde u \in \mathcal{A}^{(2,k+1)}_{\delta}$. Define $F:= \Delta_{g_{sc}} \tilde u$ and $h:= \tilde u(nm_0)$. Then
\[F \in \mathcal{A}_{\delta-2}^{(0,k-1)}(M) = \LtoH{\delta-2}{k-1} \cap \CtoH{0}{\delta-2}{k-1},  \quad h \in H^{k+1}(\partial M)\]

\vv

We utilize the spherical symmetry of $(M,g_{sc})$ to reduce the equation $\Delta_{g_{sc}} \tilde u = F$ to differential equations on the coefficients of $\tilde u$ with respect to its spherical harmonics decomposition. Decompose $\tilde u, F,$ and $h$ as follows 

\begin{equation} \label{spherical coeff}  \tilde u (r,x) = \sum_{\ell=0}^{\infty} \sum_{m=-\ell}^{\ell} a_{m\ell}(r) Y_{m\ell}(x), \quad  F (r,x) = \sum_{\ell=0}^{\infty} \sum_{m=-\ell}^{\ell} b_{m\ell}(r) Y_{m\ell}(x), \quad  h(x) = \sum_{\ell=0}^{\infty} \sum_{m=-\ell}^{\ell} c_{m\ell} Y_{m\ell}(x)   \end{equation}

for $r \in [nm_0, \infty)$ and $x \in \partial M$. Here and below the spherical harmonics $Y_{m\ell}(x)$ are viewed as functions over the unit sphere $S^2$. These same functions will also be thought of over round spheres of any other radius via the natural push-forward map. We will assume that they are normalized with respect to the round metric $\gamma_{\mathbb{S}^2}$ on the unit sphere. 

\vv
We first rewrite the norms of the relevant Banach spaces in terms of the coefficients with respect to the spherical harmonics decomposition. 
For a nonnegative integer $s$ and $f \in H^s(\partial M)$ with spherical harmonic coefficients $f_{m\ell}$, the norm 

\begin{equation}
\norm{f}_{H^s(\partial M)} := \left( \sum_{\ell=0}^{\infty} \sum_{m=-\ell}^{\ell} \left[ 1+\ell(\ell+1)\right]^s |f_{m\ell}|^2 \right)^{1/2} 
\end{equation}
is equivalent to the standard norm on $H^s(\partial M)$. For a nonnegative integer $t$ and real number $\tau$, we will rewrite the norm on $\HtoH{t}{\tau}{s}(M)$ and $\CtoH{t}{\tau}{s}$. 

\vv

We begin with the norm on $\HtoH{t}{\tau}{s}(M)$. Given a function $v \in \HtoH{t}{\tau}{s}$, recall that

\begin{align}\label{norm-v}
\normmH{v}{t}{s}{\tau}^2 &= \sum_{t'=0}^t \int_{nm_0}^{\infty} r^{-2\delta-1+2t'} \norm{\partial_r^{t'} v(r)}_{H^s(S^2)}^2 dr 
\end{align}

Let $v_{m\ell} = v_{m\ell}(r)$ be the spherical harmonic coefficients for $v$. The first term in the sum becomes 

\begin{equation}
\int_{nm_0}^{\infty} r^{-2\delta-1} \norm{v(r)}_{H^s(S^2)}^2 dr = \sum_{\ell=0}^{\infty} \sum_{m=-\ell}^{\ell} \left[1+\ell(\ell+1)\right]^s  \int_{nm_0}^{\infty} r^{-2\delta-1} |v_{m\ell}(r)|^2 dr
\end{equation}

where we invoked the monotone convergence theorem to switch the order of the integral and the infinite sum. Now for each $1\leq t'\leq t$ and $r\in [nm_0, \infty)$, we have that 

\begin{equation}
\int_{S^2} Y_{m\ell} \,\, \partial_r^{t'} v(r) d\sigma_{S^2}  =  \partial_r^{t'} \int_{S^2} Y_{m\ell} \,\, v(r) d\sigma_{S^2}
\end{equation} 
since $\partial_r^{t'-1}$ lives in  $\HtoL{1}{loc}$. In light of the fact that $v_{m\ell}(r) = \int_{S^2} Y_{m\ell} \,\, v(r) d\sigma_{S^2}$, it follows that $v_{m\ell}$ are differentiable $t$ times in $r$ and $v_{m\ell}^{(t')}$ are the spherical harmonic coefficients of $\partial_r^{t'} v$. 

\vv

We can then rewrite the norm in equation \eqref{norm-v} as follows: 

\begin{equation} 
\normmm{v}{t}{s}{\tau}^2 = \sum_{\ell=0}^{\infty} \sum_{m=-\ell}^{\ell} \left[ 1+\ell(\ell+1)\right]^{s} \sum_{t'=0}^t \int_{nm_0}^{\infty} r^{-2\tau -1 +2t'} \left(v_{m\ell}^{(t')}(r)\right)^2 dr
\end{equation}

where we have repeatedly invoked the monotone convergence theorem to switch the order of the integral and the infinite sum. 
\vv

Now consider a function $w \in \CtoH{t}{\tau}{s}$. Recall that 

\begin{align}\label{norm-w}
\normmC{w}{t}{s}{\tau}^2 = \sum_{t'=0}^t \sup\left( r^{-2\tau +2t'} \norm{\partial_r^{t'} w(r)}_{H^s(S^2)}^2\right)
\end{align}

Letting $w_{m\ell} = w_{m\ell}(r)$ be the spherical harmonic coefficients for $w$, we have 

\begin{align}\label{norm-w}
\normmC{w}{t}{s}{\tau}^2 =  \sum_{\ell=0}^{\infty} \sum_{m=-\ell}^{\ell} [1+\ell(\ell+1)]^s\sum_{t'=0}^t \sup\left( r^{-2\tau +2t'} (w^{(t')}_{m\ell}(r))^2 \right)
\end{align}

It is then convenient to define the following norms for functions on $[nm_0, \infty)$: for a function $f_1 \in H^t_{loc}([nm_0,\infty)$ and $f_2 \in C^t([nm_0,\infty))$, define 

\begin{equation} \label{norm-sph}
\norm{f_1}_{H,t,\tau}^2 := \sum_{t'=0}^t \int_{nm_0}^{\infty} r^{-2\tau-1+2t'} \left(f_1^{(t')}(r)\right)^2 dr, \quad \norm{f_2}^2_{\mathcal{C},t,\tau} := \sum_{t'=0}^t \sup\left( r^{-2\tau +2t'} ({f_2}^{(t')}(r))^2 \right)
\end{equation}

We denote by $H^t_{\tau}([nm_0,\infty))$ and $C^t_{\tau}([nm_0,\infty))$ all functions $f_1\in H^t_{loc}([nm_0,\infty)$ and $f_2 \in C^t([nm_0,\infty))$  in which $\normmmH{f_1}{t}{\tau}<\infty$ and $\normmmC{f_2}{t}{\tau}<\infty$ respectively. 

\vv

\vv

Using the above notation, we can then write equation \eqref{norm-v} and \eqref{norm-w} as follows: 

\begin{equation}
\normmH{v}{t}{s}{\tau}^2 = \sum_{\ell=0}^{\infty} \sum_{m=-\ell}^{\ell} \left[ 1+\ell(\ell+1)\right]^{s} \normmmH{v_{m\ell}}{t}{\tau}^2, \quad \normmC{w}{t}{s}{\tau}^2 = \sum_{\ell=0}^{\infty} \sum_{m=-\ell}^{\ell} \left[ 1+\ell(\ell+1)\right]^{s} \normmmC{w_{m\ell}}{t}{\tau}^2
\end{equation}

\vvv

We now return to the statement of the lemma. Recall from equation \eqref{spherical coeff} that 

\begin{itemize} 
\item $\tilde u \in \mathcal{A}^{(2,k+1)}_{\delta}(M)$ with coefficients $a_{m\ell} \in H^2_{\delta}([nm_0,\infty)) \cap C^2_{\delta}([nm_0,\infty))$. 
\item $F:= \Delta_{g_{sc}} \tilde u \in \mathcal{A}^{(0,k-1)}_{\delta-2}(M)$ with coefficients $b_{m\ell} \in L^2_{\delta-2}([nm_0,\infty)) \cap C_{\delta-2}([nm_0,\infty))$. 
\item $h:= \tilde u(nm_0) \in H^{k+1}(S^2)$ with coefficients $c_{m\ell}$. 
\end{itemize}

To prove the lemma, it then suffices to show that there exist a constant $C>0$ independent of $m$ and $\ell$ such that
\begin{equation} \tag{ \textbf{H-Est}}
\begin{aligned} \label{estimate-aml H}
\normmmH{a''_{m\ell}}{0}{\delta-2}^2 + \left[ 1+\ell(\ell+1)\right] \normmmH{a'_{m\ell}}{0}{\delta-1}^2 + \left[ 1+\ell(\ell+1)\right]^2 \normmmH{a_{m\ell}}{0}{\delta}^2 \\ \leq C ( \normmmH{b_{m\ell}}{0}{\delta-2}^2 + \left[ 1+\ell(\ell+1)\right]^{3/2}|c_{m\ell}|^2) 
\end{aligned}
\end{equation}

\begin{equation}  \tag{ \textbf{C-Est}}
\begin{aligned} \label{estimate-aml C}
\normmmC{a''_{m\ell}}{0}{\delta-2}^2 + \left[ 1+\ell(\ell+1)\right] \normmmC{a'_{m\ell}}{0}{\delta-1}^2 + \left[ 1+\ell(\ell+1)\right]^2 \normmmC{a_{m\ell}}{0}{\delta}^2 \\ \leq C ( \normmmC{b_{m\ell}}{0}{\delta-2}^2 + \left[ 1+\ell(\ell+1)\right]^{2}|c_{m\ell}|^2) 
\end{aligned}
\end{equation}

for each $m$ and $\ell$. Indeed, if we multiply \textbf{H-Est} and \textbf{C-Est} by $[1+\ell(\ell+1)]^{k-1}$ and sum over $m$ and $\ell$, we get the two desired estimates in the statement of the lemma. We will demonstrate this for \textbf{H-Est}:  after multiplying \textbf{H-Est} by $[1+\ell(\ell+1)]^{k-1}$, we sum over $m$ and $\ell$ to get the following three estimates 
\begin{equation}
\begin{aligned}
&\normmH{\tilde u}{2}{k-1}{\delta}^2 \nonumber \\
& =  \sum_{\ell=0}^{\infty} \sum_{m=-\ell}^{\ell} \left[ 1+\ell(\ell+1)\right]^{k-1} \normmmH{a_{m\ell}}{2}{\delta}^2\\
&=\sum_{\ell=0}^{\infty} \sum_{m=-\ell}^{\ell} [1+\ell(\ell+1)]^{k-1} \left( \begin{split} \normmmH{a''_{m\ell}}{0}{\delta-2}^2 + \normmmH{a'_{m\ell}}{0}{\delta-1}^2 + \normmmH{a_{m\ell}}{0}{\delta}^2 \end{split} \right) \\
&\leq C \left( \sum_{\ell=0}^{\infty} \sum_{m=-\ell}^{\ell} [1+\ell(\ell+1)]^{k-1} \normmmH{b_{m\ell}}{0}{\delta-2}^2+ \sum_{\ell=0}^{\infty} \sum_{m=-\ell}^{\ell} [1+\ell(\ell+1)]^{k+1/2} |c_{m\ell}|^2 \right) \\
&= C\left( \norm{F}_{\AH{0}{k-1}{\delta-2}}^2 + \norm{h}_{H^{k+1/2}(S^2)}^2 \right) 
\end{aligned}
\end{equation}

\begin{equation}
\begin{aligned}
&\normmH{\tilde u}{1}{k}{\delta}^2 \nonumber \\
& =  \sum_{\ell=0}^{\infty} \sum_{m=-\ell}^{\ell} \left[ 1+\ell(\ell+1)\right]^{k} \normmmH{a_{m\ell}}{1}{\delta}^2\\
&=\sum_{\ell=0}^{\infty} \sum_{m=-\ell}^{\ell} [1+\ell(\ell+1)]^{k} \left( \begin{split}  \normmmH{a'_{m\ell}}{0}{\delta-1}^2 + \normmmH{a_{m\ell}}{0}{\delta}^2 \end{split} \right) \\
&\leq C \left( \sum_{\ell=0}^{\infty} \sum_{m=-\ell}^{\ell} [1+\ell(\ell+1)]^{k-1} \normmmH{b_{m\ell}}{0}{\delta-2}^2 + \sum_{\ell=0}^{\infty} \sum_{m=-\ell}^{\ell} [1+\ell(\ell+1)]^{k+1/2} |c_{m\ell}|^2 \right) \\
&= C\left( \norm{F}_{\AH{0}{k-1}{\delta-2}}^2 + \norm{h}_{H^{k+1/2}(S^2)}^2 \right) 
\end{aligned}
\end{equation}

\begin{equation}
\begin{aligned}
&\normmH{\tilde u}{0}{k+1}{\delta}^2  \nonumber \\
& =  \sum_{\ell=0}^{\infty} \sum_{m=-\ell}^{\ell} \left[ 1+\ell(\ell+1)\right]^{k+1} \normmmH{a_{m\ell}}{0}{\delta}^2 \\
&\leq C \left( \sum_{\ell=0}^{\infty} \sum_{m=-\ell}^{\ell} [1+\ell(\ell+1)]^{k-1} (\normmmH{b_{m\ell}}{0}{\delta-2}^2 + \sum_{\ell=0}^{\infty} \sum_{m=-\ell}^{\ell} [1+\ell(\ell+1)]^{k+1/2} |c_{m\ell}|^2 \right) \\
&= C\left( \norm{F}_{\AH{0}{k-1}{\delta-2}}^2 + \norm{h}_{H^{k+1/2}(S^2)}^2 \right) 
\end{aligned}
\end{equation}

It then follows that 

\begin{align} 
\norm{\tilde u}_{\AH{2}{k+1}{\delta}}&= \max \left\{ \normmH{u}{2}{k-1}{\delta}, \normmH{u}{1}{k}{\delta}, \normmH{u}{0}{k+1}{\delta}\right\}\\
&\leq  C\left( \norm{F}_{\AH{0}{k-1}{\delta-2}}^2 + \norm{h}_{H^{k+1/2}(S^2)}^2 \right) 
\end{align}
as needed. 

\vvv

The rest of the proof is then devoted to prove estimates \ref{estimate-aml H} and \ref{estimate-aml C}. 

\vv

We first introduce a piece of notation. Given 2 quantities $\alpha, \beta$, we will write $\alpha \lesssim \beta$ if there exists a constant $C>0$ depending only on $n$, $m_0$ and $\delta$, such that $\alpha \leq C \beta$. In particular, the constant will not depend on $\tilde u$, $F$, $h$, $m$, and $\ell$. In this notation, the estimates that we will be proving are

\begin{equation} \tag{ \textbf{H-Est}}
\begin{aligned}
\normmmH{a''_{m\ell}}{0}{\delta-2}^2 + \left[ 1+\ell(\ell+1)\right] \normmmH{a'_{m\ell}(r)}{0}{\delta-1}^2 + \left[ 1+\ell(\ell+1)\right]^2 \normmmH{a_{m\ell}(r)}{0}{\delta}^2 \\ \lesssim \normmmH{b_{m\ell}}{0}{\delta-2}^2 + \left[ 1+\ell(\ell+1)\right]^{3/2}|c_{m\ell}|^2
\end{aligned}
\end{equation}

\begin{equation}  \tag{ \textbf{C-Est}}
\begin{aligned}
\normmmC{a''_{m\ell}}{0}{\delta-2}^2 + \left[ 1+\ell(\ell+1)\right] \normmmC{a'_{m\ell}(r)}{0}{\delta-1}^2 + \left[ 1+\ell(\ell+1)\right]^2 \normmmC{a_{m\ell}(r)}{0}{\delta}^2 \\ \lesssim \normmmC{b_{m\ell}}{0}{\delta-2}^2 + \left[ 1+\ell(\ell+1)\right]^{2}|c_{m\ell}|^2
\end{aligned}
\end{equation}
for every $m$ and $\ell$. 

\vv

The relation between $\tilde u$, $F$ and $h$, namely $F = \Delta_{g_{sc}} \tilde u$ and $h = \tilde u(nm_0)$, imply the following differential equation on the coefficients $a_{m\ell}$, $b_{m\ell}$, and $c_{m\ell}$: 

\begin{equation}
\begin{cases} \label{ode-aml}
r(r-2m_0) a''_{m\ell} (r) + 2(r-m_0) a'_{m\ell}(r) - \ell(\ell+1)a_{m\ell}(r) = r(r-2m_0)b_{m\ell}(r), & r\in [nm_0, \infty)\\
a_{m\ell}(nm_0) = c_{m\ell} &
\end{cases}
\end{equation}

We first consider the case $\ell=0$. we integrate once the above differential equation to get

\begin{equation} \label{a00}
r(r-2m_0)a_{00}'(r) = \int_{nm_0}^{r} s(s-2m_0) b_{00}(s) ds + n(n-2)m_0^2 a'_{00}(nm_0)
\end{equation}
which immediately gives the following estimate on $\normmmC{a'_{00}}{0}{\delta-1}$: 

\begin{equation}\label{a00-3}
\normmmC{a_{00}'}{0}{\delta-1} \lesssim \normmmC{b_{00}}{0}{\delta-2} + |a_{00}'(nm_0)|
\end{equation}

To estimate $\normmmH{a'_{00}}{0}{\delta-1}$, we use equation \eqref{a00} to get 

\begin{align}
\int_{nm_0}^{\infty}r^{-2\delta+1} (a'_{00}(r))^2 dr &\lesssim \int_{nm_0}^{\infty} \frac{r^{-2\delta+1}}{r^2(r-2m_0)^2} \left( \int_{nm_0}^{r} s(s-2m_0) b_{00}(s) ds \right)^2 dr + |a'_{00}(nm_0)|^2 \\
&\lesssim \int_{nm_0}^{\infty} r^{-2\delta+3} (b_{00}(r))^2 dr + |a'_{00}(nm_0)|^2
\end{align}

where Hardy's inequality was used in the last line. We then conclude the following estimate on $\normmmH{a'_{00}}{0}{\delta-1}$: 

\begin{equation} \label{a00-3H}
\normmmH{a'_{00}}{0}{\delta-1} \lesssim \normmmH{b_{00}}{0}{\delta-2} + |a'_{00}(nm_0)|
\end{equation}

\vv

We divide equation \eqref{a00} by $r(r-2m_0)$ and integrate to get 

\begin{equation}\label{a00-2}
a_{00}(r) = c_{00} + \int_{nm_0}^r \frac{1}{s(s-2m_0)} \int_{nm_0}^s s'(s'-2m_0)b_{00}(s')ds'ds + a_{00}'(nm_0)\ln\left( \frac{n(r-2m_0)}{(n-2)r}\right) \frac{n(n-2)m_0}{2}
\end{equation}

which, in the same manner as for $a'_{00}$, gives the following estimates: 

\begin{equation}\label{a00-4}
\normmmC{a_{00}}{0}{\delta} \lesssim |c_{00}| + \normmmC{b_{00}}{0}{\delta-2} + |a_{00}'(nm_0)|
\end{equation}

\begin{equation}\label{a00-4H}
\normmmH{a_{00}}{0}{\delta} \lesssim |c_{00}| + \normmmH{b_{00}}{0}{\delta-2} + |a_{00}'(nm_0)|
\end{equation}

\vv

Using the fact that $a_{00}$ vanishes at infinity, we take the limit as $r$ goes to infinity in equation \eqref{a00-2} to get the following expression for $a_{00}'(nm_0)$ in terms of $c_{00}$ and $b_{00}$: 

\begin{equation}
a_{00}'(nm_0) = \frac{2}{n(n-2) \ln\left( \frac{n}{n-2}\right)} \left( -c_{00} - \int_{nm_0}^{\infty} \frac{1}{r(r-2m_0)} \int_{nm_0}^r s(s-2m_0)b_{00}(s)dsdr\right)
\end{equation}

which allows us to estimate $a_{00}'(nm_0)$ to get

\begin{equation} \label{a00-6}
|a_{00}'(nm_0)| \lesssim |c_{00}| + \normmmH{b_{00}}{0}{\delta-2}
\end{equation}
\begin{equation} \label{a00-6 C}
|a_{00}'(nm_0)| \lesssim |c_{00}| + \normmmC{b_{00}}{0}{\delta-2}
\end{equation}

\vv

We now get an estimate for $a_{00}''(r)$. Using the ODE in \eqref{ode-aml} we isolate for $a_{00}''(r)$ to get: 

\begin{equation} 
a_{00}''(r) =  b_{00}(r) - \frac{2(r-m_0)}{r(r-2m_0)} a_{00}'(r)
\end{equation}

It then follows that 

\begin{equation}\label{a00-5}
\normmmC{a_{00}''}{0}{\delta-2} \lesssim \normmmC{b_{00}}{0}{\delta-2} + \normmmC{a_{00}'}{0}{\delta-1}
\end{equation}
\begin{equation}\label{a00-5H}
\normmmH{a_{00}''}{0}{\delta-2} \lesssim \normmmH{b_{00}}{0}{\delta-2} + \normmmH{a_{00}'}{0}{\delta-1}
\end{equation}

Combining the estimates for $a_{00}, a_{00}'$ and $a_{00}''$ in equations \eqref{a00-4}, \eqref{a00-4H}, \eqref{a00-3}, \eqref{a00-3H}, \eqref{a00-5} and \eqref{a00-5H} together with the estimates for $|a_{00}'(nm_0)|$ in equations \eqref{a00-6} and \eqref{a00-6 C}, we finally deduce the desired estimates:

\begin{align}
\normmmH{a_{00}}{2}{\delta} \lesssim |c_{00}| + \normmmH{b_{00}}{0}{\delta-2}
\end{align}
\begin{align}
\normmmC{a_{00}}{2}{\delta} \lesssim |c_{00}| + \normmmC{b_{00}}{0}{\delta-2}
\end{align}
\vvv

We now deal with the case $\ell\geq1$.

\subsection*{Proving H-Est for $\ell\geq 1$}

We multiply both sides of the differential equation in \eqref{ode-aml} by $r^{-2\delta-1}a_{m\ell}(r)$ and integrate by parts to obtain: 

\begin{align} 
&  \int_{nm_0}^{\infty} r^{-2\delta-1} r(r-2m_0)a_{m\ell}'^2(r) dr +\int_{nm_0}^{\infty} \left[ \ell(\ell+1)+(2\delta+1) \left( \delta- \frac{m_0(2\delta+1)}{r}\right) \right]r^{-2\delta-1}a_{m\ell}^2(r) dr \nonumber\\
& =-\int_{nm_0}^{\infty} r^{-2\delta-1} a_{m\ell}(r)b_{m\ell}(r) dr - (nm_0)^{-2\delta} m_0(n-2) c_{m\ell} a'_{m\ell}(nm_0) +\frac{-2\delta-1}{2} (nm_0)^{-2\delta-1} m_0(n-2) c_{m\ell}^2 \label{intparts-aml}
\end{align}

We observe that for $\ell\geq1$,

\begin{equation}
 \ell(\ell+1)+(2\delta+1) \left( \delta- \frac{m_0(2\delta+1)}{r}\right) > \ell(\ell+1) -\frac{1}{2}
\end{equation}

and, hence, we have that 

\begin{align}
\normmmH{a'_{m\ell}}{0}{\delta-1}^2 + [1+\ell(\ell+1)]\normmmH{a_{m\ell}}{0}{\delta}^2 \lesssim  & \int_{nm_0}^{\infty} r^{-2\delta-1} r(r-2m_0)a_{m\ell}'^2(r) dr \nonumber \\
&+\int_{nm_0}^{\infty} \left[ \ell(\ell+1)+(2\delta+1) \left( \delta- \frac{m_0(2\delta+1)}{r}\right) \right]r^{-2\delta-1}a_{m\ell}^2(r) dr 
\end{align}

Deriving an upper bound for the expression in right hand side of equation \eqref{intparts-aml} will then lead to an upper bound for $\normmmH{a'_{m\ell}}{0}{\delta-1}^2 + [1+\ell(\ell+1)]\normmmH{a_{m\ell}}{0}{\delta}^2$. 

\vv

We obtain an estimate for $a'_{m\ell}(nm_0)$ by multiplying equation \eqref{ode-aml} by $a'_{m\ell}(r)$ and integrating by parts to get 

\begin{equation}
\int_{nm_0}^{\infty} (r-m_0) a_{m\ell}'^2(r) dr + \ell(\ell+1)\frac{1}{2} c_{m\ell}^2 - \frac{1}{2} nm_0^2(n-2) {a}_{m\ell}'^2(nm_0) = \int_{nm_0}^{\infty}r(r-2m_0) a_{m\ell}'(r) b_{m\ell}(r) dr 
\end{equation}

where we used the fact that $a_{m\ell}'^2(r) r(r-2m_0) = o(1)$ and $a_{m\ell}^2(r) = o(1)$. By estimating the right side of the above equation as follows 

\begin{equation}
\left| \int_{nm_0}^{\infty} r(r-2m_0) a_{m\ell}'(r) b_{m\ell}(r) dr \right| \lesssim [\ell(\ell+1)]^{1/2} \normmmH{a_{m\ell}'}{0}{\delta-1}^2 + [\ell(\ell+1)]^{-1/2}\normmmH{b_{m\ell}}{0}{\delta-2}^2,
\end{equation}
we deduce that


\begin{equation}
|a_{m\ell}'(nm_0)| \lesssim [\ell(\ell+1)]^{1/4}\normmmH{a'_{m\ell}}{0}{\delta-1} +[\ell(\ell+1)]^{-1/4}\normmmH{b_{m\ell}}{0}{\delta-2} + [\ell(\ell+1)]^{1/2} |c_{m\ell}|
\end{equation}

\vvv

 
 We can now estimate the right hand side of equation \eqref{intparts-aml} to get 

\begin{align}
&-\int_{nm_0}^{\infty} r^{-2\delta-1} a_{m\ell}(r)b_{m\ell}(r) dr - (nm_0)^{-2\delta} m_0(n-2) c_{m\ell} a'_{m\ell}(nm_0) +\frac{-2\delta-1}{2} (nm_0)^{-2\delta-1} m_0(n-2) c_{m\ell}^2 \nonumber \\
&\lesssim \normmmH{a_{m\ell}}{0}{\delta} \normmmH{b_{m\ell}}{0}{\delta-2} + |c_{m\ell}| \left( [\ell(\ell+1)]^{1/4}\normmmH{a'_{m\ell}}{0}{\delta-1} +[\ell(\ell+1)]^{-1/4}\normmmH{b_{m\ell}}{0}{\delta-2} + [\ell(\ell+1)]^{1/2} |c_{m\ell}|\right)\nonumber \\ &\qquad  + c_{m\ell}^2 
\end{align}

We estimate each term appearing in the right hand side of the above equation. Let $\epsilon >0$ that will be chosen to be small later on. Then we have  

\begin{equation}
 \normmmH{a_{m\ell}}{0}{\delta} \normmmH{b_{m\ell}}{0}{\delta-2} \leq \epsilon \ell(\ell+1)  \normmmH{a_{m\ell}}{0}{\delta}^2 + \frac{D(\epsilon)}{\ell(\ell+1)} \normmmH{b_{m\ell}}{0}{\delta-2}^2
\end{equation}

\begin{equation}
c_{m\ell}  \normmmH{b_{m\ell}}{0}{\delta-2} \leq \frac{[\ell(\ell+1)]^{3/4}}{2} c_{m\ell}^2 + \frac{1}{2[\ell(\ell+1)]^{3/4}} \normmmH{b_{m\ell}}{0}{\delta-2}^2 
\end{equation}


\begin{equation}
c_{m\ell} \normmmH{a'_{m\ell}}{0}{\delta-1} \leq \epsilon [\ell(\ell+1)]^{-1/4} \normmmH{a'_{m\ell}}{0}{\delta-1}^2 + D(\epsilon) [\ell(\ell+1)]^{1/4} c_{m\ell}^2 
\end{equation}

where $D=D(\epsilon)$ is a constant depending on $\epsilon$. 

\vv

Combining the above, we get 

\begin{align} 
&\normmmH{a'_{m\ell}}{0}{\delta-1}^2 + [1+\ell(\ell+1)]\normmmH{a_{m\ell}}{0}{\delta}^2  \\
&\leq C\epsilon \left( \ell(\ell+1) \normmmH{a_{m\ell}}{0}{\delta}^2 + \normmmH{a'_{m\ell}}{0}{\delta-1}^2 \right) + C D(\epsilon) \left( [\ell(\ell+1)]^{1/2} c_{m\ell}^2 + \frac{1}{\ell(\ell+1)}\normmmH{b_{m\ell}}{0}{\delta-2}^2 \right) 
\end{align}
for some constant $C>0$ that only depends on $n$, $m_0$, and $\delta$. Choosing $\epsilon$ to be small, we can absorb the expression multiplied to $C\epsilon$ in the above equation to the left hand side to finally deduce 

\begin{align}
\normmmH{a'_{m\ell}}{0}{\delta-1}^2 + [1+\ell(\ell+1)]\normmmH{a_{m\ell}}{0}{\delta}^2 \lesssim [1+\ell(\ell+1)]^{-1} \normmmH{b_{m\ell}}{0}{\delta-2}^2 + [1+\ell(\ell+1)]^{1/2} c_{m\ell}^2 \label{est2}
\end{align}

which is the desired estimate for $\normmmH{a'_{m\ell}}{0}{\delta-1}$ and $\normmmH{a_{m\ell}}{0}{\delta}$. 

\vv

What is left is to estimate $\normmmH{a''_{m\ell}}{0}{\delta-2}$. We use the ODE in \eqref{ode-aml} to get 

\begin{equation}
r^{-2\delta-1} r^2(r-2m_0)^2 \left(a''_{m\ell}(r) \right)^2 = r^{-2\delta-1} \bigg[ b_{m\ell}(r) + \ell(\ell+1)a_{m\ell}(r) - 2(r-m_0) a'_{m\ell}(r) \bigg]^2
\end{equation}

We can then estimate

\begin{align}
\normmmH{a''_{m\ell}}{0}{\delta-2}^2 &\lesssim \normmmH{b_{m\ell}}{0}{\delta-2}^2 + [1+\ell(\ell+1)]^2 \normmmH{a_{m\ell}}{0}{\delta}^2 + \normmmH{a'_{m\ell}}{0}{\delta-1}^2 + \ell(\ell+1) \normmmH{b_{m\ell}}{0}{\delta-2} \normmmH{a_{m\ell}}{0}{\delta}\nonumber \\  &\qquad + \normmmH{a'_{m\ell}}{0}{\delta-1} \normmmH{b_{m\ell}}{0}{\delta-2} + \ell(\ell+1) \normmmH{a_{m\ell}}{0}{\delta} \normmmH{a'_{m\ell}}{0}{\delta-1} \\
&\lesssim [1+\ell(\ell+1)] \normmmH{a'_{m\ell}}{0}{\delta-1}^2 + [1+\ell(\ell+1)]^2 \normmmH{a_{m\ell}}{0}{\delta}^2 + \normmmH{b_{m\ell}}{0}{\delta-2}^2\\
&\lesssim [1+\ell(\ell+1)]^{3/2} c_{m\ell}^2 + \normmmH{b_{m\ell}}{0}{\delta-2}^2
\end{align}
where we used equation \eqref{est2} in the last line. Combining the above equation with equation \eqref{est2}, we finally get the desired estimate {\bf H-Est}: 
\begin{equation}
\normmmH{a''_{m\ell}}{0}{\delta-2}^2 + [1+\ell(\ell+1)]\normmmH{a'_{m\ell}}{0}{\delta-1}^2 + [1+\ell(\ell+1)]^2 \normmmH{a_{m\ell}}{0}{\delta}^2 \lesssim \normmmH{b_{m\ell}}{0}{\delta-2}^2 + [1+\ell(\ell+1)]^{3/2} c_{m\ell}^2
\end{equation}

\vv

\subsection*{Proving C-Est for $\ell\geq 1$} 

For $r\in [nm_0,\infty)$ define

\[ z:=\frac{r}{m_0} - 1, \quad R:= n-1, \quad h_{m\ell}(z) := a_{m\ell}(r), \quad f_{m\ell}(z) := r(r-2m_0) b_{m\ell}(r)  \]

Note that $R>1$ since $n>2$. The desired estimate in {\bf C-Est} in terms of $h$ is then 
 
\begin{equation}
\normmmC{h''_{m\ell}}{0}{\delta-2}^2 + [1+\ell(\ell+1)]\normmmC{h'_{m\ell}}{0}{\delta-1}^2 + [1+\ell(\ell+1)]^2 \normmmC{h_{m\ell}}{0}{\delta}^2 \lesssim \normmmC{f_{m\ell}}{0}{\delta}^2 + [1+\ell(\ell+1)]^2 c_{m\ell}^2  \tag{\bf C-Est$'$}
\end{equation}

The IVP for $a_{m\ell}$ in \eqref{ode-aml} becomes

\begin{equation}
\begin{cases} \label{ode-hml}
(z^2-1) h''_{m\ell} (z) + 2z h'_{m\ell}(z) - \ell(\ell+1)h_{m\ell}(z) = f_{m\ell}(z), & z\in [R, \infty)\\
h_{m\ell}(R) = c_{m\ell} &\\
h_{m\ell} \in H^2_{\delta}([R,\infty)) \cap C^2_{\delta}([R,\infty))
\end{cases}
\end{equation}

The above ODE is the Legendre differential equation; the Legendre functions of the first and second kind, $P_{\ell}$ and $Q_{\ell}$, are two linearly independent solutions to the homogeneous equation in \eqref{ode-hml} (i.e. with $f_{m\ell} = 0$) satisfying the following asymptotics as $z\to \infty$ (see \cite{olver} chapter 5 section 12): 

\begin{equation}
P_{\ell}(z) = O(z^{\ell}), \quad Q_{\ell}(z) = O(z^{-\ell-1})
\end{equation}
We will frequently use some properties of those functions discussed and proved in section \ref{legendre appendix} in the Appendix.

\vv

We normalize $P_{\ell}$ and $Q_{\ell}$ so that 

\begin{equation}
\lim_{z\to \infty} z^{-\ell} P_{\ell}(z) = 1, \quad \lim_{z\to \infty} z^{\ell+1} Q_{\ell}(z) = 1
\end{equation}

Using the method of Frobenius, we can expand $P_{\ell}$ and $Q_{\ell}$ as a sum of powers of $z$ on $[R,\infty)$, which we present in proposition \ref{frobenius} in the Appendix. We rewrite the proposition here for convenience. 

\begin{prop}
$P_{\ell}$ and $Q_{\ell}$ admit an expansion of the following form. For $z>1$, 
\begin{equation}
P_{\ell}(z) = \sum_{k=0}^\ell a_k z^{\ell-k}, \quad Q_{\ell}(z) = \sum_{k=0}^{\infty} b_k z^{-\ell-1-k}
\end{equation}

where the coefficients $a_k$ and $b_k$ are defined recursively as follows:

\[ a_0 = b_0 = 1, \quad a_1 = b_1 = 0\]
\[\text{for $k\geq 2$}, \quad a_{k} = \frac{(\ell-k+2)(\ell-k+1)}{k^2-k(2\ell+1)}a_{k-2}, \quad b_{k} = \frac{(\ell+k-1)(\ell+k)}{k(2\ell+k+1)}b_{k-2}  \]
\end{prop}

We observe immediately from the above that $Q_{\ell}(z)$ is positive and $z^{\ell+1}Q_{\ell}(z)$ is decreasing on $[R,\infty)$.

\vv

Using the variation of parameters method (see \cite{advancedmathmethods}), we can explicitly write the solution to \eqref{ode-hml}: 
\begin{equation} \label{hml}
h_{m\ell}(z) = A Q_{\ell}(z) + P_{\ell}(z) \int_{z}^\infty Q_{\ell}(t) f_{m\ell}(t) [ W(t) (t^2-1) ]^{-1} dt + Q_{\ell}(z) \int_{R}^z P_{\ell}(t) f_{m\ell}(t) [ W(t) (t^2-1) ]^{-1} dt 
\end{equation}

where $W(t) := P_{\ell}(t) Q_{\ell}'(t) - P_{\ell}'(t) Q_{\ell}(t)$ is the Wronskian and $A$ is defined by 
\begin{equation} \label{A}
A = \frac{1}{Q_{\ell}(R)} \left( c_{m\ell} - P_{\ell}(R) \int_{R}^\infty Q_{\ell}(t) f_{m\ell}(t) [ W(t) (t^2-1) ]^{-1} dt\right) 
\end{equation}
Note that $W(t)(t^2-1)$ is a constant by Lagrange's identity (see \cite{lagrange's-identity} pg 354). We compute that constant to be $2\ell+1$ by taking the limit as $z$ goes to $\infty$ in the expansion of $P_{\ell}$ and $Q_{\ell}$. 

\vv

We summarize here some estimates on the Legendre functions uniform in $z$ and $\ell$ that we prove in the Appendix (check proposition \ref{prop-estPQ}): 
There exists a constant $C = C(R)$ such that for any $\ell\geq 1$ and $z\in [R,\infty)$, the following holds 

\begin{equation} \label{estPQ}
z^{-\ell}|P_{\ell}(z)| \leq  C \left(\frac{2z}{z+\sqrt{z^2-1}} \right)^{-\ell},\qquad z^{\ell+1} |Q_{\ell}(z)| \leq C\left(\frac{2z}{z+\sqrt{z^2-1}} \right)^{\ell}
\end{equation}
\begin{equation}\label{estPQ'}
z^{-(\ell-1)}|P'_{\ell}(z)| \leq C \ell \left(\frac{2z}{z+\sqrt{z^2-1}} \right)^{-\ell}, \qquad z^{\ell+2} |Q'_{\ell}(z)| \leq C \ell \left(\frac{2z}{z+\sqrt{z^2-1}} \right)^{\ell}
\end{equation}

We note that the function $\frac{2z}{z+\sqrt{z^2-1}}$ is decreasing on $[R,\infty)$ and is bounded below and above by $1$ and $2$ respectively. Using the expression for $h_{m\ell}$ in equation \eqref{hml} and the uniform bounds of $P_{\ell}$ and $Q_{\ell}$ in equation \eqref{estPQ}, we obtain

\begin{align}
 z^{-\delta}|h_{m\ell}(z)| &\leq z^{-\delta}|A| |Q_{\ell}(z)| +z^{-\delta} \frac{1}{2\ell+1}|P_{\ell}(z)| \int_{z}^{\infty} |Q_{\ell}(t)| |f_{m\ell}(t)| dt + z^{-\delta}\frac{1}{2\ell+1}|Q_{\ell}(z)| \int_{R}^z |P_{\ell}(t)| |f_{m\ell}(t)| dt \\
&\leq z^{-\delta} |A| Q_{\ell}(R) R^{\ell+1} z^{-\ell-1} \nonumber\\ &\qquad+ \frac{C^2}{2\ell+1} \left(\sup_{t\geq R} t^{-\delta} |f_{m\ell}(t)| \right) z^{\ell-\delta} \left(\frac{2z}{z+\sqrt{z^2-1}} \right)^{-\ell} \int_{z}^{\infty} \left(\frac{2t}{t+\sqrt{t^2-1}} \right)^{\ell} t^{-\ell-1}t^{\delta} dt \nonumber\\ &\qquad  +\frac{C^2}{2\ell+1} \left(\sup_{t\geq R} t^{-\delta} |f_{m\ell}(t)| \right) z^{-\ell-1-\delta} \left(\frac{2z}{z+\sqrt{z^2-1}} \right)^{\ell}\int_{R}^{z} \left(\frac{2t}{t+\sqrt{t^2-1}} \right)^{-\ell}t^{\ell}t^{\delta} dt\\
&\leq |A|Q_{\ell}(R) R^{-\delta} + \frac{C^2}{2\ell+1} \left(\sup_{t\geq R} t^{-\delta} |f_{m\ell}(t)| \right) \left(\frac{1}{\ell-\delta}+ \frac{1}{\ell+1+\delta}\left(1-(\frac{R}{z})^{\ell+\delta+1}\right) \right)
\end{align} 

It immediately follows that

\begin{equation} \label{est-hml}
\normmmC{h_{m\ell}}{0}{\delta} \lesssim |A|Q_{\ell}(R) + \ell^{-2} \normmmC{f_{m\ell}}{0}{\delta}
\end{equation}

To derive an estimate for $|A|Q_{\ell}(R)$, we use equation \eqref{A} to get

\begin{align}
|A|Q_{\ell}(R) &\leq |c_{m\ell}| + \frac{C^2}{2\ell+1} \left( \sup_{t\geq R} t^{-\delta} |f_{m\ell}(t)|\right) R^{\ell} \left(\frac{2R}{R+\sqrt{R^2-1}} \right)^{-\ell} \int_{R}^{\infty} \left(\frac{2t}{t+\sqrt{t^2-1}} \right)^{\ell} t^{-\ell-1}t^{\delta} dt  \label{A1}\\
&\leq  |c_{m\ell}| + \frac{C^2R^{\delta}}{(2\ell+1)(\ell-\delta)} \left( \sup_{t\geq R} t^{-\delta} |f_{m\ell}(t)|\right) \label{A2}
\end{align}
The above estimate for $|A|Q_{\ell}(R)$ together with equation \eqref{est-hml} implies 

\begin{equation} \label{est-hml2}
\normmmC{h_{m\ell}}{0}{\delta} \lesssim |c_{m\ell}| + \ell^{-2} \normmmC{f_{m\ell}}{0}{\delta}
\end{equation}

To achieve an estimate for $\normmmC{h_{m\ell}'}{0}{\delta-1}$, we take the derivative of equation \eqref{hml} to get

\begin{equation}
h_{m\ell}'(z) = A Q_{\ell}'(z) + \frac{1}{2\ell+1}P_{\ell}'(z) \int_{z}^\infty Q_{\ell}(t) f_{m\ell}(t) dt + \frac{1}{2\ell+1}Q_{\ell}'(z) \int_{R}^z P_{\ell}(t) f_{m\ell}(t) dt 
\end{equation}
In a similar manner, we will apply the uniform bounds on $P_{\ell}$ and $Q_{\ell}$ in equations \eqref{estPQ} and \eqref{estPQ'} to obtain the desired estimate for $\normmmC{h_{m\ell}'}{0}{\delta-1}$. Using the above equation for $h_{m\ell}'$ as well as equations \eqref{estPQ} and \eqref{estPQ'}, we get

\begin{align}
 z^{-\delta+1}|h_{m\ell}'(z)| &\leq z^{-\delta+1}|A| |Q_{\ell}'(z)| + z^{-\delta+1}\frac{1}{2\ell+1}|P_{\ell}'(z)| \int_{z}^{\infty} |Q_{\ell}(t)| |f_{m\ell}(t)| dt + z^{-\delta+1}\frac{1}{2\ell+1}|Q_{\ell}'(z)| \int_{R}^z |P_{\ell}(t)| |f_{m\ell}(t)| dt \\
&\leq  z^{-\delta+1}|A| |Q_{\ell}'(z)| + \frac{C^2 \ell}{2\ell+1} \left(\sup_{t\geq R} t^{-\delta} |f_{m\ell}(t)| \right) z^{\ell-\delta} \left(\frac{2z}{z+\sqrt{z^2-1}} \right)^{-\ell}\int_{z}^{\infty} \left(\frac{2t}{t+\sqrt{t^2-1}} \right)^{\ell}t^{-\ell-1}t^{\delta} dt \nonumber\\ &\qquad  +\frac{C^2\ell}{2\ell+1} \left(\sup_{t\geq R} t^{-\delta} |f_{m\ell}(t)| \right) z^{-\ell-1-\delta} \left(\frac{2z}{z+\sqrt{z^2-1}} \right)^{\ell} \int_{R}^{z} \left(\frac{2t}{t+\sqrt{t^2-1}} \right)^{-\ell} t^{\ell}t^{\delta} dt\\
&\leq  z^{-\delta+1} |A| |Q_{\ell}'(z)| + \frac{C^2\ell}{2\ell+1} \left(\sup_{t\geq R} t^{-\delta} |f_{m\ell}(t)| \right) \left(\frac{1}{\ell-\delta}+ \frac{1}{\ell+1+\delta}\left(1-(\frac{R}{z})^{\ell+\delta+1}\right) \right) \label{est for h'}
\end{align} 

We estimate the term $ z^{-\delta+1} |A| |Q_{\ell}'(z)|$. First, we observe from the expansion of $Q_{\ell}(z)$ in proposition \ref{frobenius} that $z^{\ell+2}Q_{\ell}'(z)$ is negative and increasing on $[R,\infty)$, which in particular implies that 

\begin{align}
z^{-\delta+1} |A||Q_{\ell}'(z)| &\leq z^{-\delta+1} |A| |Q_{\ell}'(R)| R^{\ell+2} z^{-\ell-2}\\
&\leq z^{-\delta-1} \left( \frac{R}{z} \right)^{\ell} R^2 |A| |Q_{\ell}'(R)|\\
&\leq R^{-\delta+1} |A| |Q_{\ell}'(R)| \label{A for Q'}
\end{align}

We then use the recursive relation for $Q_{\ell}$ in equation \eqref{recursive} to deduce that 

\begin{align}
|Q_{\ell}'(R)| &= \frac{\ell}{R^2-1} (-R Q_{\ell}(R) + Q_{\ell-1}(R)) \\
&< \frac{\ell}{R^2-1}  Q_{\ell-1}(R)
\end{align}
which, in light of equations \eqref{A for Q'} and \eqref{A2}, implies that 

\begin{align}
z^{-\delta-1} |A| |Q_{\ell}'(z)| &\leq \frac{\ell R^{-\delta+1}}{R^2-1} |A| Q_{\ell-1}(R)\\
 &\leq \frac{\ell R^{-\delta+1}}{R^2-1} \left( |c_{m\ell}| + \frac{C^2R^{\delta}}{(2\ell-1)(\ell-1-\delta)} \left( \sup_{t\geq R} t^{-\delta} |f_{m\ell}(t)|\right) \right)
\end{align}
The above together with equation \eqref{est for h'} finally lead to the desired estimate for $\normmmC{h_{m\ell}'}{0}{\delta-1}$: 

\begin{equation} \label{est-hml3}
\normmmC{h_{m\ell}'}{0}{\delta-1} \lesssim \ell |c_{m\ell}| + \ell^{-1} \normmmC{f_{m\ell}}{0}{\delta}
\end{equation}

What is left is estimating $\normmmC{h_{m\ell}''}{0}{\delta-2}$. Using the ODE satisfied by $h_{m\ell}$ in \eqref{ode-hml}, we have 
\begin{equation}
z^{2-\delta} h''_{m\ell} (z) =  - 2\frac{z^{3-\delta}}{z^2-1} h'_{m\ell}(z) + \frac{\ell(\ell+1)z^{2-\delta}}{z^2-1}h_{m\ell}(z) +\frac{z^{2-\delta}}{z^2-1} f_{m\ell}(z)
\end{equation}

which then, using equations \eqref{est-hml2} and \eqref{est-hml3} implies 

\begin{align}
\normmmC{h_{m\ell}''}{0}{\delta-2} &\lesssim \normmmC{h_{m\ell}'}{0}{\delta-1}+\ell(\ell+1) \normmmC{h_{m\ell}}{0}{\delta} + \normmmC{f_{m\ell}}{0}{\delta}\\
&\lesssim \ell^2 |c_{m\ell}| + \normmmC{f_{m\ell}}{0}{\delta}
\end{align}

The above equation together with equations \eqref{est-hml2} and \eqref{est-hml3} finally imply the desired estimate: 

\begin{equation}
\normmmC{h''_{m\ell}}{0}{\delta-2}^2 + [1+\ell(\ell+1)]\normmmC{h'_{m\ell}}{0}{\delta-1}^2 + [1+\ell(\ell+1)]^2 \normmmC{h_{m\ell}}{0}{\delta}^2 \lesssim \normmmC{f_{m\ell}}{0}{\delta}^2 + [1+\ell(\ell+1)]^2 c_{m\ell}^2  \tag{\bf C-Est$'$}
\end{equation}

\vvv

This concludes the proof of the lemma. 

\end{proof}

\section{Reduction of the problem} \label{reduction chapter}

In this section, we reduce the static Einstein vacuum equations into a simpler system involving ODEs, the Laplace equation on $M$, and first order partial differential equations on $\partial M$. 

\vv

Let $g$ be a metric on $M$ of the form $dr^2+g(r)$, where $g(r)$ is a metric on $S^2$ for each $r\in[nm_0,\infty)$. The level sets of the function $r$ defines a foliation with leaves denoted by $S_r$. We define the unit vector field $n:=\dd{r}$ that is normal to the foliation. We denote by $\slashed \nabla$ and $\cancel{div}$ the covariant derivative and divergence with respect to the induced metric $g(r)$ on $S_r$, and $\slashed d$ the exterior derivative on $S_r$. 

\vv

 We then define the second fundamental form $K$ as the $(0,2)$ symmetric tensor field on $M$ that is tangential to the leaves $S_r$ of the foliation and satisfies:

\begin{equation}
K(X,Y) := g(\nabla_nX, Y)
\end{equation}
for vector fields $X,Y$ on $M$ that are tangential to $S_r$. We will decompose $K$ into the sum of its traceless and trace parts: 
\[K = \hat K + \frac{1}{2} trK \,\, g\]

\noindent Note that $\text{Hess}_{g}(r) = K$ and $\Delta_g r = trK$. 

\vv

The following equations on the leaves $S_r$ describe the evolution of the geometry on $M$ in terms of $K$ and $Ric$ (see \cite{christodoulou} and \cite{klainerman}). Given coordinates $(r, \theta^1, \theta^2)$ on $M$, 

\begin{equation} \label{geomeq1}
\partial_r trK + \frac{1}{2} (trK)^2 + |\hat K|^2=-R_{00}
\end{equation}

\begin{equation}\label{geomeq2}
\nabla_r \hat K_{ij} + trK \hat K_{ij} = -\left[ R_{ij} + \frac{1}{2} g_{ij}(R_{00}-R)\right]
\end{equation}

\begin{equation} \label{geomeq3}
R_{S_r} - \frac{1}{2} (trK)^2 + |\hat K|^2 = R-2R_{00}
\end{equation}

\begin{equation}\label{geomeq4}
\slashed \nabla^j \hat K_{ji} - \frac{1}{2} \slashed \nabla_i trK = R_{0i}
\end{equation}

\begin{equation}\label{geomeq5}
\partial_r g_{ij} = 2\hat K^{ij} + trK \, g_{ij}
\end{equation}

where $i,j = 1,2$, $R_{S_r}$ is the scalar curvature of $(S_r, g(r))$, $R$ is the scalar curvature on $(M,g)$, and $\slashed \nabla$ is the connection on $(S_r, g(r))$. Moreover, $R_{00} := Ric(n,n)$ and $R_{0i} := Ric(n, \dd{\theta^i})$ for $i=1,2$. 

\vv

We note that the left side of equation \eqref{geomeq2} can be simplified as follows
\begin{align}
\nabla_r \hat K_{ij} + trK \hat K_{ij} &= \partial_r \hat K_{ij} - 2\Gamma_{0j}^l \hat K_{il} +trK \hat K_{ij} \\
&= \partial_r \hat K_{ij}-2\hat K^l_j \hat K_{il}
\end{align}
Equation \eqref{geomeq2} can then be written as follows: 

\begin{equation}\label{geomeq22}
(\mathcal{L}_{\dd{r}} \hat K)_{ij} -2\hat K^l_j \hat K_{il}= -\left[ R_{ij} + \frac{1}{2} g_{ij}(R_{00}-R)\right]
\end{equation}

\vv

The above equations determine all the components of the Ricci curvature of $g$. More specifically, if the right hand sides of equations \eqref{geomeq1} to \eqref{geomeq4} are known on all the leaves, then the Ricci curvature can be fully recovered. In fact, if we in addition know that $(g,u)$ solves the conformal static vacuum Einstein equations for some function $u$ on $M$, then, due to the contracted Bianchi identities, equations \eqref{geomeq3} and \eqref{geomeq4} need only to be imposed on the boundary for the Ricci curvature to be fully recovered. The next proposition will prove this fact and will demonstrate the desired reduction of our problem.

\begin{redthm} \label{prop-reduction}
Let $(\gammao, \frac{1}{2} \trko)$ be Bartnik data. Let $g = dr^2 +g(r)$ and $u$ be a metric and function on $M$ respectively, where $g(r)$ is a metric on $S^2$ for every $r \in [nm_0,\infty)$. The pair $(\mathfrak g, f) = (e^{-2u} g, e^u)$ solves the static Einstein vacuum equations with Bartnik data $(\gammao, \frac{1}{2} \trko)$ if and only if $(g,u)$ satisfies 

\begin{align} \label{redeq1}
\Delta_g u &= 0, \quad \text{on $M$}\\ \label{redeq2}
\partial_r trK + \frac{1}{2}trK^2 + |\hat K|^2 + 2(\partial_ru)^2 &= 0, \quad \text{on $M$}\\ \label{redeq3}
\nabla_r \hat K + trK \hat K + \left[2 \slashed du \otimes \slashed du + g(r) \left((\partial_r u)^2 - |\nabla u|^2 \right)\right] &= 0, \quad \text{on $M$}\\ \label{redeq4}
2|\slashed \nabla u|^2 - 2(\partial_r u)^2 - |\hat K|^2 - R_{\partial M} + \frac{1}{2}{trK}^2 &= 0, \quad \text{on $\partial M$}\\ \label{redeq5}
2(\partial_r u) \slashed du - \cancel{div} (\hat K) + \frac{1}{2} \slashed dtrK  &=0, \quad \text{on $\partial M$}\\ \label{redeq6}
\left. e^{-2u}g \right|_{\partial M} &= \gammao, \quad \text{on $\partial M$}\\ \label{redeq7}
\left. e^u \left( trK \right|_{\partial M} - 2 \partial_r u \right) &= \trko, \quad \text{on $\partial M$}
\end{align}
\end{redthm}
\begin{proof}
The ``only if " direction is clear from equations \eqref{geomeq1} to \eqref{geomeq4}. We prove the ``if "  direction. 
\

\vv
Suppose $(g,u)$ satisfy equations \eqref{redeq1} to \eqref{redeq7}. It suffices to show that $Ric = 2du \otimes du$. We first decompose the Ricci curvature of $g$ with respect to the foliation. Let $\Pi$ be the $(1,1)$ projection tensor field defined by 
\begin{equation}
\Pi^{\mu}_{\nu} = \delta^{\mu}_{\nu} - n^{\mu}n_{\nu}
\end{equation}
We then define the function $Q$, the $1$-form $P$ tangential to the foliation, and the $(0,2)$ symmetric tensor field $S$ tangential to the foliation as follows: 
\begin{equation}
Q := Ric(n,n), \quad P_{\mu} := \Pi^{\mu'}_{\mu}n^{\nu} R_{\mu'\nu}, \quad  S_{\mu\nu} := \Pi^{\mu'}_{\mu} \Pi^{\nu'}_{\nu}R_{\mu'\nu'}
\end{equation}
The Ricci curvature of $g$ can then be written in the following way:
\begin{equation} \label{ricci-decomposition}
Ric = Q \nb \otimes \nb + P \otimes \nb + \nb \otimes P + S
\end{equation}
where $\nb$ is the 1-form achieved by lowering the index for $n$. We will omit the underbar when we write $\nb$ in components. 
\vv

Define the function $H$ on $M$ and the 1-form $A$ tangent to the foliation in the following way: 
\begin{equation} \label{defH} H := R - 2|\nabla u|^2\end{equation}
\begin{equation} \label{defA} A := P - 2n(u) \slashed du \end{equation}

\

We now compare equations \eqref{geomeq1} - \eqref{geomeq4} with equations \eqref{redeq2} - \eqref{redeq5}. From Equation \eqref{geomeq1} and \eqref{redeq2}, we deduce on $M$ that
\begin{equation} Q = 2n(u)^2 \end{equation}
From equation \eqref{geomeq2} and \eqref{redeq3}, we deduce on $M$ that
\begin{equation} S = 2\slashed du \otimes \slashed du + \frac{1}{2} \gamma H \end{equation}
 We also have by definition of $A$: 
\begin{equation} P = 2n(u) \slashed d u + A \end{equation}

From equation \eqref{geomeq3} and \eqref{redeq4}, it follows that on $\partial M$, 
\begin{equation} R - 2Q = 2|\nabla u|^2 - 4 n(u)^2 \end{equation}
which gives us: 
\begin{equation} \left. H \right|_{\partial M} = 0\end{equation}

From equation \eqref{geomeq4} and \eqref{redeq5}, we get: 
$$\left. A \right|_{\partial M} = 0$$

To prove the statement, we just need to show that $H,A= 0$. 

\vv

We first prove the following lemma.

\begin{lem} \label{AandHlemma}
Let $dr^2+g(r)$ be a metric on $M$ where $g(r)$ is a metric on $S^2$ for every $r \in [nm_0,\infty)$. Suppose that the Ricci decomposition relative to the foliation defined by $r$, as written in equation \eqref{ricci-decomposition}, is

\begin{equation}
Q = 2n(u)^2, \qquad P = 2n(u) \slashed d u + A, \qquad S = 2\slashed du \otimes \slashed du + \frac{1}{2} g(r) H
\end{equation}
where $u$ is a harmonic function on $M$, $H$ is a function, and $A$ is a 1-form tangent to the foliation. 

Then $A$ and $H$ satisfy 
\begin{equation} \label{AandHeq} (\nabla_n A)_k + A_i K^i_k + trK A_k = 0
\end{equation}

\begin{equation}
\nabla_nH + H trK = 2div (A)
\end{equation}
\end{lem}
\begin{proof}



Recall the second Bianchi identity: 
\begin{equation} \frac{1}{2} \nabla_{\nu} R = \nabla^{\mu}R_{\nu \mu}\end{equation}

where $\mu, \nu = 0,1,2$. 

We can write the Ricci curvature as follows: 

\begin{equation}
Ric = Q \nb \otimes \nb + P \otimes \nb + \nb \otimes P + S
\end{equation}

\vv

We compute the divergence of the tensor $2du\otimes du$ to be:
\begin{align} 
 \nabla^{\mu} (2du\otimes du)_{\mu\nu}&= 2\nabla^{\mu} du_{\mu}du_{\nu} + 2 du_\mu \nabla^{\mu} du_{\nu} \\
 &= \Delta u \,  du_{\nu} + 2\mathrm{Hess}(u) (\nabla u, \partial_{\nu})  \\
 &=\nabla_{\nu} |\nabla u|^2
 \end{align}

\vv

where $\Delta u = 0$ was used in the last line. Using the Bianchi identities, we get

\begin{align}
\frac{1}{2} \nabla_{\nu} R &= \nabla^{\mu} \bigg(2du \otimes du + A \otimes \nb + \nb \otimes A + \frac{1}{2}H\gamma \bigg)_{\mu\nu}\\
&=\nabla^{\mu} (2du\otimes du)_{\mu\nu} + \nabla^{\mu}A_{\mu} n_{\nu} + A_{\mu} \nabla^{\mu} n_{\nu} + \nabla^{\mu} n_{\mu} A_{\nu} + n_{\mu} \nabla^{\mu} A_{\nu}\\
& \quad + \frac{1}{2} \nabla_{\nu'}H \Pi^{\nu'}_{\nu} + \frac{1}{2} H \nabla^{\mu}(g_{\mu \nu} - n_{\mu}n_{\nu}) \nonumber\\
&=\nabla_{\nu}|\nabla u|^2 + \nabla^{\mu}A_{\mu} n_{\nu} + A_{\mu} K^{\mu}_{\nu} + trK A_{\nu} + \nabla_n A_{\nu} \\
& \quad + \frac{1}{2} \nabla_{\nu'}H \Pi^{\nu'}_{\nu} - \frac{1}{2} H trK n_{\nu} \nonumber
\end{align}

where $\Pi^{\nu'}_{\mu} = \delta^{\nu'}_{\mu} - n^{\nu'} n_{\mu}$. We also used the fact that $\gamma_{\mu\nu} = g_{\mu' \nu'} \Pi^{\mu'}_{\mu} \Pi^{\nu'}_{\nu} =g_{\mu \nu} - n_{\mu}n_{\nu}$.

\vv

Using fermi coordinates $(r,\theta^1, \theta^2)$ and letting $\nu=i=1,2$ , we get
\begin{equation}
(\nabla_n A)_i + A_j K^j_i + trK A_i= 0
\end{equation}

Letting $\nu = 0$, we get
\begin{equation}
\nabla_nH + H trK = 2 \nabla^{\mu}A_{\mu}
\end{equation}
as desired. 

\end{proof}

\vv

We then have that $A$ satisfies,
\begin{equation}
\begin{cases} (\nabla_n A)_k + A_i K^i_k + trK A_k = 0, & \text{on $M$} \\

A = 0, & \text{on $\partial M$}
\end{cases}
\end{equation}
By the existence and uniqueness theory of ODEs, it follows that $A=0$.

\vv

Since $div (A) = 0$, we get that $H$ satisfies

\begin{equation}
\begin{cases} \nabla_n H +H trK= 0, & \text{on $M$} \\

H = 0, & \text{on $\partial M$}
\end{cases}
\end{equation}

By invoking again the existence and uniqueness theorem of ODEs, we deduce that $H = 0$. 
\end{proof}

\section{ Proof of The Main Theorem} \label{proof section}

The reduction theorem in section \ref{reduction chapter} suggests that we study the map 

\begin{multline*}
 \Psi: \mathcal{M}^{k+1}(\partial M) \times H^k(\partial M)  \times \mathcal{M}^k_{\delta}(M)  \times {\mathcal{A}^{(2,k+1)}_{\delta}} \\ 
 \to \mathcal{A}^{(0,k-1)}_{\delta-2}(M) \times  L^2_{\delta-2} \bigg( [nm_0,\infty); H^k(S^2) \bigg) \times  L^2_{\delta-2} \bigg( [nm_0,\infty) ; \mathcal{H}^k (S^2) \bigg)\\
 \times H^{k-1}(\partial M) \times \Omega^{k-1}(\partial M) \times \mathcal{H}^{k}(\partial M) \times  H^{k}(\partial M) 
 \end{multline*}

\begin{align} \label{Gtilde}
& \Psi(\gammao, \frac{1}{2} \trko , g, u ) := \begin{pmatrix} \Delta_{g} u \lv \\ 
\partial_r trK + \frac{1}{2}trK^2 + |\hat K|^2 + 2(\partial_ru)^2  \lv \\
\nabla_r \hat K + trK \hat K  + \left[2 \slashed d u \otimes \slashed d u + g(r) \left((\partial_r u)^2 - |\nabla u|^2 \right)\right]  \lv \\
2|\slashed \nabla u|^2 - 2(\partial_r u)^2 - |\hat K|^2 - R_{\partial M} + \frac{1}{2}{trK}^2  \lv \\
2(\partial_r u) \slashed du - \cancel{div} (\hat K) + \frac{1}{2} \slashed dtrK     \lv   \\
\left. e^{-2u} g \right|_{\partial M} - \gammao \lv \\
\left. trK \right|_{\partial M} - e^{-u}\left( trK_{\mathfrak{B}} + 2 e^{u}\partial_r u\right) \end{pmatrix} 
\end{align}
where $trK$ and $\hat K$ are with respect to the metric $ g$, and $R_{\partial M}$ is with respect to the metric $e^{2u}\gammao$. Furthermore, norms $\lvert \cdot \rvert$ used in the second, third and fourth line are with respect to the metric $ g$. 

\vv

We wish to show that there exists a map taking Bartnik data $(\gammao, \frac{1}{2} \trko)$ close to Schwarzschild data to a pair $(g,u)$ satisfying $\Psi(\gammao, \frac{1}{2} \trko, g, u) = 0$, showing that $(\mathfrak g, f) = (e^{-2u} g, e^u)$ solves the static Einstein vacuum equations with Bartnik data $(\gammao, \frac{1}{2} \trko)$. This can be achieved by first attempting to show that the linearization of $\Psi$ with respect to $(g,u)$ at $( \gammasc, \frac{1}{2} \trksc, g_{sc}, u_{sc} )$ is an isomorphism, or merely surjective, and then invoking the implicit function theorem. However, the linearization of the contracted Codazzi equation, in the fifth line of the definition of $\Psi$, leads to obstructions to surjectivity stemming from the divergence operator acting on symmetric traceless tensors on $S^2$. More specifically, we are faced with the cokernel of the divergence operator: a $6$-dimensional space of obstructions equal to the space of conformal Killing vector fields on $S^2$. 

\vv

This difficulty does not preclude the possibility of finding solutions given arbitrary Bartnik data close to Schwarzschild data. A similar situation arises when one attempts to show the existence of metrics on the sphere with prescribed scalar curvature (see \cite{kazdan}). The operator of study will not satisfy the conditions for the inverse function theorem, yet existence holds as shown in \cite{an-huang}. In our case, we circumvent this difficulty by introducing an artificial object, in the form of a vector field $X$, to the meaning of a solution to our problem, proving its existence using the implicit function theorem, and then finally showing that this vector field $X$ vanishes, yielding a solution to the original problem. 

\vv

\subsection{Definition of the Artificial Vector Field $X$}\label{artificial-section}

 As explained in the introduction of section \ref{proof section}, the contracted Codazzi equations give rise to obstructions that are in correspondence with the space of conformal Killing vector fields on $S^2$. We will overcome these seeming obstructions by introducing an artificial vector field $X$ to the definition of a solution; this means that the solution will consist of a metric $g$, a function $u$, {\it and} a vector field $X$. This needs to be done in a way so that, firstly, the corresponding modified problem is solvable, and secondly, the artificial vector field, in fact, vanishes for a solution to the modified problem, yielding a solution to the original problem. To achieve this, the artificial vector field $X$ needs to be carefully defined, which will require a certain way of uniquely extending conformal Killing fields from $S^2$ to the ambient manifold $M$. This procedure will be outlined in this section. 
 
 Notably, Huang and An have also introduced an artificial vector field $X$ in \cite{an-huang} and \cite{an-huang2} for analogous purposes; specifically, they define $X$ to be a vector field that vanishes on the boundary and asymptotically approaches a Killing vector field at infinity. In contrast, we will define $X$ to be a vector field that satisfies $\left(\mathcal{L}_{\dd{r}} X \right)^T = 0$ on the boundary and asymptotically approaches a conformal killing vector field on $(M,g_{sc})$ that restricts to a conformal Killing vector field on $(\partial M, \gamma_{\mathbb{S}^2})$.

\vvv

Given a metric $g$ and a vector field $X$ on $M$, we denote by $\conflie{g}X$ and $\Delta_{g,conf}X$ the conformal Lie derivative of $g$ with respect to $X$ and the conformal laplacian of $X$ defined by 

\begin{equation}
\conflie{g}X := \widehat{\mathcal{L}_X g}, \quad \Delta_{g,conf} X:= \text{div}_g \left(\conflie{g}X  \right)
\end{equation}

\noindent where $\widehat{\mathcal{L}_X g}$ is the traceless part of $\mathcal{L}_X g$. It follows that $X$ is conformal Killing on $(M,g)$ if and only if $\conflie{g} X = 0$.

\vv



\begin{defn}
Given a conformal Killing vector field $X_{CK}$ on $(\partial M, \gamma_{\mathbb{S}^2})$, we denote by $\overline{X_{CK}}$ the unique vector field on $(M,g_{sc})$ extending $X_{CK}$ on $\partial M$ and satisfying the evolution equation  
\begin{equation}
\mathcal{L}_{\dd{r}} \overline{X_{CK}} = 0, \quad \text{in $M$}
\end{equation}
Also, we will use $`` \cancel{div}_{\gamma_{\mathbb{S}^2}} (X_{CK})"$ to denote both the divergence of $X_{CK}$ on $(\partial M, \gamma_{\mathbb{S}^2})$ and the same function extended to a function on $M$ independent of $r$. It should be clear from context which one we are referring to. 
\end{defn}


\vv

\begin{defn}
Define the space $\mathcal{X}_{\infty}$ as the space of conformal killing vector fields $X_{\infty}$ on $(M,g_{sc})$ of the form 
\begin{equation} \label{xinf}
X_{\infty} = f(r) \bigg( \cancel{div}_{\gamma_{\mathbb{S}^2}} (X_{CK}) \bigg) \dd{r} + h(r) \overline{X_{CK}}
\end{equation}
where $f = f(r)$ and $h= h(r)$ are smooth functions on $M$ such that $f= 0 $ and $h= 1$ on $\partial M$ and $X_{CK} $ is a conformal Killing vector field on $(\partial M, \gamma_{\mathbb{S}^2})$. \end{defn}

\vv

\noindent In the case that $X_{CK}$ is Killing, equation \eqref{xinf} becomes 

\begin{equation}\label{xinfkilling}
X_{\infty} = h(r) \overline{X_{CK}}
\end{equation}

\begin{lem} \label{xinf-lemma}
Let $X_{CK}$ be a nontrivial conformal Killing vector field on $(\partial M, \gamma_{\mathbb{S}^2})$. 
\begin{enumerate}[label=\textbf{(\alph*)}]
\item Suppose $X_{CK}$ is Killing on $(\partial M, \gamma_{\mathbb{S}^2})$. Then $h\equiv 1$ is the unique smooth function $h=h(r)$ on $M$  in which $h =1$ on $\partial M$ and $X_{\infty}$, defined by equation \eqref{xinfkilling}, is conformal Killing on $(M,g_{sc})$. In fact, $X_{\infty}$ would also be Killing. 

\item Suppose $X_{CK}$ is not Killing on $(\partial M, \gamma_{\mathbb{S}^2})$. There exists unique smooth functions $f= f(r)$ and $h=h(r)$ on $M$ such that $f\equiv 0$, $h\equiv 1$ on $\partial M$ and the vector field $X_{\infty}$ defined by \eqref{xinf}
is conformal Killing on $(M,g_{sc})$. Furthermore, $f = O(r^2)$ and $h=O(r^2)$. 
\end{enumerate}
In particular, $\mathcal{X}_{\infty}$ is a $6$ dimensional vector space of conformal Killing vector fields on $(M,g_{sc})$. \\
\end{lem}
\begin{proof}
We will repeatedly use the following identity of the Lie derivative: for any vector fields $X,Y, Z$ and any $(0,2)$ tensor field $T$ on $M$, 

\begin{equation}
\Big[\mathcal{L}_{X} T \Big] (Y,Z) = X \bigg( T(Y,Z) \bigg) - T([X,Y], Z) - T(Y,[X,Z])
\end{equation}

We first prove (a). Suppose $X_{CK}$ is Killing on $(\partial M, \gamma_{\mathbb{S}^2})$. It is clear that $h(r) \overline{X_{CK}}$ is Killing on $(M,g_{sc})$ for $h \equiv 1$ on $M$ as it is a rotation vector on the spherically symmetric Schwarzschild manifold. Now suppose $h=h(r)$ is a smooth function such that $h\equiv 1$ on $\partial M$ and  $h(r) \overline{X_{CK}}$ is conformal Killing on $M$. In particular, we have that 

\begin{align}
0 &= \Big[ \mathcal{L}_{h(r) \overline{X_{CK}}} g_{sc} \Big] \left( \dd{r}, \overline{X_{CK}} \right)\\
&= g_{sc} \left( \left[\dd{r}, h(r) \overline{X_{CK}}\right], \overline{X_{CK}}\right) + g_{sc} \left( \left[Y, h(r) \overline{X_{CK}}\right], \dd{r} \right) \\
&= h'(r) g_{sc}(\overline{X_{CK}}, \overline{X_{CK}}) 
\end{align}
In view of the fact that $X_{CK} \neq 0$, it follows that $h' \equiv 0$ and so $h\equiv 1$ on $M$ as needed. 

\vv

We now prove (b). Suppose $X_{CK}$ is not Killing on $(\partial M, \gamma_{\mathbb{S}^2})$. Let $f = f(r)$ and $h=h(r)$ be smooth functions on $M$. Recall that the metric $g_{sc}$ can be written as 
\[g_{sc} = dr^2 + r(r-2m_0) \gamma_{\mathbb{S}^2}\]
Only for the proof of this lemma, we will denote the function $\cancel{div}_{\gamma_{\mathbb{S}^2}}(X_{CK}) $ by $B_{X_{CK}}$ for simplicity of the notation. Recall that $B_{X_{CK}}$ is understood as a function on $M$ or a function on $\partial M$ depending on the context, and that $B_{X_{CK}}$ as a function on $M$ is constant in $r$. 

\vv

Since $X_{CK}$ is not Killing, we have that $B_{X_{CK}} $ is nonzero. Moreover, after fixing a spherical coordinate system on $\partial M$, the vector field $X_{CK}$ can be written as a linear combination of the vector fields $ \slashed \nabla_{\gamma_{\mathbb{S}^2}} Y_{m}^{\ell = 1}$, for $m=-1,0,1$, where $Y_{m}^{\ell=1}$ are the $\ell=1$ spherical harmonics on $\partial M$. In particular, it holds that
\begin{align}
\slashed \nabla_{\gamma_{\mathbb{S}^2}} B_{X_{CK}} = -2 X_{CK}
\end{align}

This implies that

\begin{align}
\nabla_{g_{sc}} \left( B_{X_{CK}} \right)  &= \frac{1}{r(r-2m_0)} \slashed \nabla_{\gamma_{\mathbb{S}^2}} \left( B_{X_{CK}} \right) \\
&= -\frac{2}{r(r-2m_0)} \overline{X_{CK}}
\end{align}

\vv

For $X_{\infty} := f(r) B_{X_{CK}} \dd{r} + h(r) \overline{X_{CK}}$ and arbitrary vector fields $Y,Z$ tangent to the foliation, we compute

\begin{align}
\Big[ \mathcal{L}_{X_{\infty}} g_{sc} \Big] \left(\dd{r}, \dd{r} \right) &= 2 g_{sc}\left(\left[\dd{r}, X_{\infty}\right], \dd{r} \right)\\
& = 2f'(r) B_{X_{CK}}
\end{align}

\begin{align}
\Big[ \mathcal{L}_{X_{\infty}} g_{sc} \Big] \left(\dd{r}, Y \right) &= g_{sc}\left(\left[\dd{r}, X_{\infty}\right], Y\right) + g_{sc}\left(\left[Y, X_{\infty}\right], \dd{r}\right)\\
& = h'(r) g_{sc}( \overline{X_{CK}}, Y) + f(r) Y \left( B_{X_{CK}} \right)\\ 
&=h'(r) g_{sc}( \overline{X_{CK}}, Y) + f(r) g_{sc} \left(  \nabla B_{X_{CK}} , Y \right)\\
&= \left( h'(r) - \frac{2}{r(r-2m_0)} f(r)\right) g_{sc}(\overline{X_{CK}}, Y)
\end{align}

\begin{align}
\Big[ \mathcal{L}_{X_{\infty}} g_{sc} \Big] \left(Y, Z \right) &= \Big[ \mathcal{L}_{f(r)B_{X_{CK}} \dd{r}} g_{sc}\Big] \left(Y, Z \right) + \Big[\mathcal{L}_{h(r)\overline{X_{CK}}} g_{sc}\Big] \left(Y, Z \right)\\
& = f(r)\,\, B_{X_{CK}}\,\, trK_{sc} \, \, g_{sc}(Y,Z) + r(r-2m_0)h(r) \, \Big[\mathcal{L}_{X_{CK}} \gamma_{\mathbb{S}^2}\Big] (Y,Z)\\
& = f(r)\,\, B_{X_{CK}}\,\, trK_{sc} \, \,g_{sc}(Y,Z) + r(r-2m_0)h(r) \, B_{X_{CK}}  \, \gamma_{\mathbb{S}^2}(Y,Z)\\
&= \left[ f(r) trK_{sc} +h(r) \right] B_{X_{CK}} g_{sc}(Y,Z)
\end{align} 

In the above calculation, we used the fact that $ X_{CK}$ is conformal Killing on each leaf and hence $\mathcal{L}_{ X_{CK}} \gamma_{\mathbb{S}^2} = B_{X_{CK}} \, \gamma_{\mathbb{S}^2}$.

\vv

It follows that $\mathcal{L}_{X_{\infty}}g_{sc}$ is conformal to $g_{sc}$,  $f \equiv 0$ and $h\equiv 1$ on $\partial M$  if and only if  the pair $(f,h)$ satisfy the following on $[n m_0, \infty)$.

\begin{equation}
\begin{cases}2f'(r) = f(r) trK_{sc} + h(r),\\
h'(r) = \frac{2}{r(r-2m_0)} f(r), \\
f (n m_0)= 0, \\
h (n m_0)= 1\
\end{cases}
\end{equation}

This decouples to the following initial value problems for $f$ and $h$ on $[n m_0, \infty)$. 

\begin{equation} \label{asymp-f}
\begin{cases}f''(r) - \frac{(r-m_0)}{r(r-2m_0)} f'(r) + \frac{2m_0^2}{r^2(r-2m_0)^2} f(r) = 0 \\
f(n m_0) = 0 \\
f'(n m_0) = \frac{1}{2} \\
\end{cases}
\end{equation}

\begin{equation} \label{asymp-h}
\begin{cases}h'(r) = \frac{2}{r(r-2m_0)} f(r) \\
h(n m_0) = 1 \\
\end{cases}
\end{equation}

Invoking the existence and uniqueness theorem for ODEs, it follows that there exists unique smooth functions $f$ and $g$ satisfying the above initial value problems.

\vv

 We have then proven that there exists unique smooth functions $f= f(r)$ and $h=h(r)$ on $M$ such that $f\equiv 0$, $h\equiv 1$ on $\partial M$ and the vector field $X_{\infty}$ is conformal Killing on $(M,g_{sc})$.

\vv

We utilize Fuchsian theory to establish that both $f$ and $h$ are $O(r^2)$. We first observe that the ODE for $f$ in equation \eqref{asymp-f} has a regular singular point at infinity. We can then express $f$ as a Frobenius series as follows (see \cite{advancedmathmethods}):

\begin{equation}
f(r) = (r-m_0)^{\alpha} \sum_{n=0}^{\infty} a_n (r-m_0)^{-n} 
\end{equation} 

where $\alpha \in \R$ is to be determined and $a_0 \neq 0$. We substitute this expression of $f$ into the ODE in \eqref{asymp-f} to get 
\begin{equation}
\begin{aligned} 
&\sum_{n=0}^{\infty} a_n (\alpha-n)(\alpha-n-1) (r-m_0)^{\alpha-n-2} - \sum_{n=0}^{\infty} a_n (\alpha-n) (r-m_0)^{\alpha-n-2} \left( 1+\frac{m_0^2}{r(r-2m_0)}\right)\\
 &+ 2m_0^2 \sum_{n=0}^{\infty} a_n (r-m_0)^{\alpha-n-4} \frac{(r-m_0)^4}{r^2(r-2m_0)^2} 
 \end{aligned} \quad  = \quad  0
\end{equation}

Upon examining the highest power of $r-m_0$, we deduce that $\alpha$ must satisfy the equation 
\begin{equation}
\alpha(\alpha-1) - \alpha = 0
\end{equation}

implying that $\alpha$ can only be $0$ or $2$. It follows that $f=O(r^2)$. The fact that $h=O(r^{2})$ follows immediately from equation \eqref{asymp-h}. 

\end{proof}


\begin{remark}
Note that any vector field $X_{\infty}$ in $\mathcal{X}_{\infty}$ satisfies 
\[\left(\mathcal{L}_{\dd{r}} X_{\infty} \right)^T = 0, \qquad \text{on $\partial M$}\]
Hence, the lemma proves an existence and uniqueness result for an overdetermined problem: for any conformal Killing vector field $X_{CK}$ on $(\partial M, \gamma_{\mathbb{S}^2})$, there exists a unique conformal Killing vector field $X_{\infty}$ on $(M,g_{sc})$ satisfying the following boundary conditions on $\partial M$: 
\begin{equation}
\left. X_{\infty} \right|_{\partial M} = X_{CK}, \quad \left(\mathcal{L}_{\dd{r}} X_{\infty} \right)^T = 0
\end{equation}
Furthermore, $X_{\infty}$ will of the form as in equation \eqref{xinf} for some functions $f=f(r)$ and $h=h(r)$. 
\end{remark}

\begin{defn} We define $\widehat{\mathcal{X}}^{2}_{\delta}(M)$ to be all vector fields $X \in \mathcal{X}^{2}_{\delta}(M)$ (see definition \ref{def of Xkdelta}) such that $X|_{\partial M}$ is tangent to $\partial M$ and $ \left( \mathcal{L}_{\dd{r}} X \right)^T = 0$ on $\partial M$. \\
The artificial vector field $X$ will be chosen to live in the space $\widehat{\mathcal{X}}^{2}_{\delta}(M) \oplus \mathcal{X}_{\infty}(M)$. The reasons for this choice will be clear in the next sections.  
\end{defn}

\vv
\subsection{Definition and Existence of the Modified Solution}

In this section, we will define the modified problem and its solutions, which we call ``the modified solutions", and prove their existence. Here and onwards, we fix a number $\delta \in (-1,-\frac{1}{2}]$ and an integer $k \geq 5$.

\vv

Let $\eta$ be a smooth cut off function on $[n m_0 , \infty)$ satisfying $\eta(r) = 1$ for $r \geq nm_0 +2$ and $\eta(r) = 0$ for $r\leq nm_0 +1$. Given $X \in \widehat{\mathcal{X}}^{2}_{\delta}(M) \oplus \mathcal{X}_{\infty}(M)$, define the function $F(X)$ on $M$ by 

\begin{equation} 
F(X) := e^{-r^4\eta(r) |X|^2}
\end{equation}
where $| \cdot |$ is taken with respect to $g_{sc}$. 
\vv

Given $ g = dr^2 +g(r) \in \mathcal{M}^k_{\delta} (M)$ and $X_{\infty} \in \mathcal{X}_{\infty}$, define $\omega( g, X_{\infty})$ to be the 1-form on $\partial M$ achieved by lowering the index of $X_{\infty}|_{\partial M}$ with respect to $ g(n m_0)$.  Note that $X_{\infty}|_{\partial M}$ is a conformal Killing field on $(\partial M, \gamma_{\mathbb{S}^2})$ by definition of the space $\mathcal{X}_{\infty}$.

\vv

\vv

Define $ \Phi$ by:

\begin{multline*}
 \Phi: \mathcal{M}^{k+1}(\partial M) \times H^k(\partial M)  \times \mathcal{M}^k_{\delta}(M)  \times {\mathcal{A}^{(2,k+1)}_{\delta}}(M)  \times \bigg( \widehat{\mathcal{X}}^{2}_{\delta}(M) \oplus \mathcal{X}_{\infty}(M) \bigg) \\ 
 \to  \mathcal{A}^{(0,k-1)}_{\delta-2}(M) \times  L^2_{\delta-2} \bigg( [nm_0,\infty); H^k(S^2) \bigg) \times  L^2_{\delta-2} \bigg( [nm_0,\infty) ; \mathcal{H}^k (S^2) \bigg) \times \mathcal{X}^{0}_{\delta-2}(M)\\
 \times H^{k-1}(\partial M) \times \Omega^{k-1}(\partial M) \times \mathcal{H}^{k}(\partial M) \times  H^{k}(\partial M) 
 \end{multline*}

\begin{align} \label{Gtilde}
& \Phi(\gammao, \frac{1}{2} \trko , g, u, X ) := \begin{pmatrix} \Delta_{g} u &  \text{on $M$} \lv \\ 
\partial_r trK + \frac{1}{2}trK^2 + |\hat K|^2 + 2(\partial_ru)^2 &  \text{on $M$} \lv \\
\nabla_r \hat K + trK \hat K  + \left[2 \slashed d u \otimes \slashed d u + g(r) \left((\partial_r u)^2 - |\nabla u|^2 \right)\right] &  \text{on $M$} \lv \\
\Delta_{g, conf} (F(X) X) & \text{on $M$} \lv \\
2|\slashed \nabla u|^2 - 2(\partial_r u)^2 - |\hat K|^2 - R_{\partial M} + \frac{1}{2}{trK}^2 & \text{on $\partial M$} \lv \\
2(\partial_r u) \slashed du - \cancel{div} (\hat K) + \frac{1}{2} \slashed dtrK + \omega( g, X_{\infty})  &  \text{on $\partial M$}  \lv   \\
\left. e^{-2u} g \right|_{\partial M} - \gammao &  \text{on $ \partial M$} \lv \\
\left. trK \right|_{\partial M} - e^{-u}\left( trK_{\mathfrak{B}} + 2 e^{u}\partial_r u\right) & \text{on $\partial M$}  \end{pmatrix} 
\end{align}

where $X_{\infty}$ is the projection of $X$ onto $\mathcal{X}_{\infty}(M)$, $trK$ and $\hat K$ are with respect to the metric $ g$, and $R_{\partial M}$ is with respect to the metric $e^{2u}\gammao$. Furthermore, norms $\lvert \cdot \rvert$ used in the second, third and fifth equation are with respect to the metric $ g$.

\vv


\vv

\begin{defn}
Given Bartnik data $(\gammao, \frac{1}{2} \trko)$, we say that a 3-tuple $(g,u, X)$ is a modified solution if $ \Phi(\gammao, \frac{1}{2} \trko, g, u, X) = 0$.
\end{defn}

\begin{remark} \label{X=0 iff} In view of proposition \ref{prop-reduction}, a modified solution is a solution to the original problem if and only if $X = 0$. 
\end{remark}
\vv



\vv

The main tool to obtain the existence of the modified problem is the implicit function theorem on Banach manifolds (see \cite{implicit}), which is stated here for convenience. 

\begin{thm}
Let $U \subset E$, $V \subset F$ be open subsets of Banach spaces $E$ and $F$, and let $\Psi: U \times V \to G$ be a $C^r$ map to a Banach space $G$, with $r\geq 1$. For some $x_0 \in U$, $y_0 \in V$, assume the partial derivatives in the second argument $D_2 \Psi(x_0,y_0): F \to G$ is an isomorphism. Then there are neighbourhoods $U_0$ of $x_0$ and $W_0$ of $\Psi(x_0,y_0)$ and a unique $C^r$ map $H: U_0 \times W_0 \to V$ such that for all $(x,w) \in U_0 \times W_0$, $\Psi(x, H(x,w)) = w$. 
\end{thm}

 \vvv

The map $\Phi$ is indeed $\mathcal{C}^1$ near $( \gammasc, \frac{1}{2} \trksc, g_{sc}, u_{sc}, 0 )$. To see this, we first note the following: 
\begin{itemize}
\item The map $u \mapsto du \otimes du$ is $\mathcal{C}^1$ from $\AHC{2}{k+1}{\delta}(M)$ to $\HtoHH{1}{2\delta-2}{k}$. 
\item The map $g \mapsto trK$ is $\mathcal{C}^1$ from $\mathcal{M}^k_{\delta}(M)$ to $\HtoH{1}{\delta-1}{k}$. 
\item The map $g \mapsto \hat K$ is $\mathcal{C}^1$ from $\mathcal{M}^k_{\delta}(M)$ to $\HtoHH{1}{\delta-1}{k}$. 
\end{itemize}
This immediately shows that each line, excluding the fourth line, in the definition of $\Phi$ is $\mathcal{C}^1$. It remains to show that the map $(g,X) \mapsto \Delta_{g, conf} (F(X) X) $ is $\mathcal{C}^1$ from $\mathcal{M}^k_{\delta}(M) \times \bigg( \widehat{\mathcal{X}}^{2}_{\delta}(M) \oplus \mathcal{X}_{\infty}(M) \bigg)$ to $\mathcal{X}^0_{\delta-2}$ near $(g_{sc}, 0)$. This follows directly from the smoothness of $F(X)$ and the following identity of the conformal laplacian (see \cite{york}): 

\begin{equation}
\Delta_{g, conf} (F(X) X)_{\mu} = \Delta_g (F(X) X)_{\mu} + \frac{1}{3} \nabla_{\mu} \Big(\text{div}_g (F(X) X)\Big) + R_{\mu \nu} F(X) X^{\nu}
\end{equation}

\vv


We can then differentiate $\Phi$ at $( \gammasc, \frac{1}{2} \trksc, g_{sc}, u_{sc}, 0 )$ and study its derivative.

\vv
Let $ D\Phi_{sc}$ denote the derivative of $ \Phi$ with respect to the last three components evaluated at $( \gammasc, \frac{1}{2} \trksc, g_{sc}, u_{sc}, 0 )$ where
\begin{multline*}
D \Phi_{sc}: T_{g_{sc}}\mathcal{M}^k_{\delta} \times {\mathcal{A}^{(2,k+1)}_{\delta}} \times  \bigg( \widehat{\mathcal{X}}^{2}_{\delta}(M) \oplus \mathcal{X}_{\infty}(M) \bigg)\\
  \to  \mathcal{A}^{(0,k-1)}_{\delta-2}(M) \times  L^2_{\delta-2} \bigg( [nm_0,\infty); H^k(S^2) \bigg) \times  L^2_{\delta-2} \bigg( [nm_0,\infty) ; \mathcal{H}^k (S^2) \bigg)\\
 \times H^{k-1}(\partial M) \times \Omega^{k-1}(\partial M) \times \mathcal{H}^{k}(\partial M) \times  H^{k}(\partial M) 
 \end{multline*}

\begin{prop} \label{Dphi-iso}
 $D \Phi_{sc}$ is an isomorphism.
\end{prop}
\begin{proof} 
The proof of this will be the content of section \eqref{proofDPhi}. 
\end{proof}

\vv

We can now conclude the existence theorem for the extended problem.
\begin{thm}
There exists a neighbourhood $\mathcal{U}$ of $( \gammasc, \frac{1}{2} \trksc)$ in $\mathcal{M}^{k+1}(\partial M) \times H^k(\partial M) $ and a unique $\mathcal{C}^1$ map ${\bf H}: (\gammao, \frac{1}{2} \trko) \mapsto (g, u, X)$ on $\mathcal{U}$ into $ \mathcal{M}^k_{\delta}(M) \times \mathcal{A}^{(2,k+1)}_{\delta}(M) \times  \widehat{\mathcal{X}}^{2}_{\delta}(M) \oplus \mathcal{X}_{\infty}(M) $ satisfying 
\begin{equation}  \Phi(\gammao, \frac{1}{2} \trko, {\bf H}(\gammao, \frac{1}{2} \trko)) = 0, \qquad \text{for all $(\gammao, \frac{1}{2} \trko) \in \mathcal{U}$}  \end{equation}

\end{thm}
\begin{proof}
Follows from proposition \eqref{Dphi-iso} and the implicit function theorem on Banach manifolds. 
\end{proof}

\vv

\subsection{The Vanishing of $X$ for Modified solutions $(g,u,X)$}

The next step is to show that if $(g,u,X)$ is a modified solution, then $X=0$, yielding a solution $(g,u)$ to the conformal static vacuum Einstein equations.  
\vv

Let $(\gammao, \frac{1}{2} \trko) \in \mathcal{U}$ be Bartnik data and let $(g,u, X):= {\bf H}(\gammao, \frac{1}{2} \trko)$ be the corresponding modified solution. 

\vv

We first find the Ricci curvature of the metric $g$. 

\begin{prop}
The Ricci decomposition of $ g = dr^2 + g(r)$ relative to the foliation defined by $r$ is given by 
\begin{equation} \label{ricci decomp}
Ric = Q \nb \otimes \nb + P \otimes \nb + \nb \otimes P + S
\end{equation}


\begin{equation} \label{ricci decomp2}
Q = 2n(u)^2, \qquad P = 2n(u) \slashed d u + A, \qquad S = 2\slashed du \otimes \slashed du + \frac{1}{2} g(r) H
\end{equation}
where $H$ and $A$ are the unique function on $M$ and 1-form on $M$ tangent to the foliation satisfying:

\begin{equation} \label{eqforA}
\begin{cases} \nabla_n A_k + A_i K^i_k + trK A_k = 0, & \text{on $M$} \\

A = \omega(g, X_{\infty}), & \text{on $\partial M$}
\end{cases}
\end{equation}

\begin{equation} \label{eqforH}
\begin{cases} \nabla_nH + H trK = 2div(A), & \text{on $M$} \\

H = 0, & \text{on $\partial M$} 
\end{cases}
\end{equation}

\end{prop}
\begin{proof}
Lemma \eqref{AandHlemma} shows that $A$ and $H$ satisfy the desired transport equations on $M$. The boundary condition for $A$ and $H$ follow by comparing equations 

\[ 2|\slashed \nabla u|^2 - 2(\partial_r u)^2 - |\hat K|^2 - R_{\partial M} + \frac{1}{2}{trK}^2 = 0, \quad \text{and} \quad 2(\partial_r u) \slashed du - \cancel{div} (\hat K) + \frac{1}{2} \slashed dtrK  + \omega(g,X_{\infty})=0 \]

with equations \eqref{geomeq3} and \eqref{geomeq4}. 
\end{proof}
 
 \vv

The relation between the Ricci curvature of $g$ and $u$ as described in the above proposition leads to the following regularity result. 

\begin{prop} \label{regprop} The following holds for any modified solution $(g,u,X)$. 

\begin{itemize}
\item The Ricci curvature of $g$ is $\mathcal{C}^1$ away from the boundary. Furthermore, there exists a universal constant $C>0$ such that for $R>nm_0$, 

\begin{equation} \label{regRic} 
\sup_{r>R} r^2( |Ric| + r|\nabla Ric|) \leq C\norm{\gamma_{\infty} - \gamma_{\mathbb{S}^2}}_{\mathcal{H}^k(\partial M)} + o(R^{\delta})
\end{equation}
as $R$ goes to $\infty$.

\item The vector field $F(X) X$ lies in $\mathcal{X}^3_{\delta}(M)$.
\end{itemize}
\end{prop}
\begin{proof}
We first find explicit expressions for $A$ and $H$. Letting $(r,\theta^1, \theta^2)$ be fermi coordinates, we compute for $i=1,2$, 
\begin{align}
\nabla_n A_i &= \partial_r A_i - \Gamma_{0i}^j A_j\\
&= \partial_r A_i - K^j_i A_j
\end{align}

 Equation \eqref{eqforA} for $A$ then becomes 
\begin{equation} 
\begin{cases} \partial_r A_i + trK A_i = 0, & \text{on $M$} \\

A = \omega(g, X_{\infty}), & \text{on $\partial M$}
\end{cases}
\end{equation}
which gives 
\begin{equation}
A_i(r) = \frac{1}{L(r)} \omega_i
\end{equation}

where $L(r) := \exp\left( \ds \int_{nm_0}^r trK(s) ds\right)$. 

\vv

We then solve equation\eqref{eqforH} for $H$ to obtain 

\begin{equation} 
H(r) = \frac{1}{L(r)} \int_{nm_0}^r L(s) 2 div(A(s)) ds
\end{equation}

From proposition \eqref{prop-spaces}, we have 

\begin{equation} 
\left|trK - \frac{2}{r}\right| \in \HtoH{1}{\delta-1}{k}, \qquad \hat K \in \HtoHH{1}{\delta-1}{k}
\end{equation}

It then follows that $A \in \HtoHH{2}{loc}{k}$ and $div( A) \in \HtoHH{1}{loc}{k-1}$, which in turn implies that $H \in \HtoH{2}{loc}{k-1}$.

\vv

Furthermore, since $k\geq 5$, the Sobolev embedding described in proposition \eqref{prop-spaces} imply that $trK$ and $\hat K$ are continuous in $r$, have 3 continuous angular derivatives, and satisfy 
\begin{equation} \label{reg-trk}
trK = \frac{2}{r} + o(r^{\delta-1}), \quad |\hat K| = o(r^{\delta-1})
\end{equation}

Using the asymptotics of $trK$ described above, we derive the asymptotics of $L$ to be

\begin{equation}
L(r)^{-1} = \mathcal{O}(r^{-2}), \quad \partial_r (L(r)^{-1}) = \mathcal{O}(r^{-3}), \quad |\slashed \nabla (L(r)^{-1})| = \mathcal{O}(r^{-3})
\end{equation}
as $r$ goes to $\infty$. 

\vv

It follows that $A$ and $H$ have continuous first derivatives and satisfy 

\begin{equation} \label{reg-A}
|A(r)| = \mathcal{O}(r^{-3}),\quad  |\partial_r A(r)| = \mathcal{O}(r^{-4}), \quad |\slashed \nabla A(r)| = \mathcal{O}(r^{-4})
\end{equation}

\begin{equation}
H(r) = \mathcal{O}(r^{-2}), \quad \partial_r H(r) = \mathcal{O}(r^{-3}), \quad |\slashed \nabla H(r)| = \mathcal{O}(r^{-3})
\end{equation}

\vv

By virtue of the fact that $(g,u,X)$ is a modified solution, we have 

\begin{equation} \label{reg1}
\partial_{r}^2 u + trK \partial_r u + \slashed \Delta_{g(r)} u = 0
\end{equation}

The above equation together with the fact that $u \in \mathcal{A}^{(2,k+1)}_{\delta}(M)$ then implies $\partial_r^2 u  \in \HtoH{1}{\delta-2}{k-2}$. This in particular implies that $\partial_r u, |\slashed \nabla u| \in \HtoH{2}{\delta-1}{k-2}$. Using the Sobolev embeddings again and the fact that $k\geq 5$, we deduce that $\partial_r u, |\slashed \nabla u|$ have continuous first derivatives and satisfy 

\begin{equation} \label{reg-u} 
\partial_r u= o(r^{\delta-1}), \quad \partial_r^2 u = o(r^{\delta-2}), \quad |\slashed \nabla \partial_r u| = o(r^{\delta-2})
\end{equation}

\vv

Having achieved the asymptotics for $u$, we can now derive an explicit expression for the leading order term for $H$ and $\nabla H$. Using the Gauss equation, we get 

\begin{align}
r^2 H(r) &= r^2 R{g(r)} -2 + 2r^2 (\partial_r u)^2 + r^2 |\hat K|^2 - \frac{r^2}{2} \left( trK - \frac{2}{r} \right)^2 + 2r \left(trK - \frac{2}{r}\right)\\
&= r^2 R_{g(r)} -2 + o(r^{\delta})
\end{align}

where we have used equations \eqref{reg-u} and \eqref{reg-trk}. 
\vv

It then follows that 

\begin{equation} \label{reg-H}
 H(r) =r^{-2} ( R_{\gamma_{\infty}} -2) + o(r^{\delta-2}), \quad |\nabla H(r)| = r^{-3} \Big( |R_{\gamma_{\infty}} -2|^2 + |\slashed d R_{\gamma_{\infty}}|^2_{\gamma_{\infty}} \Big)^{1/2}  + o(r^{\delta-3})
\end{equation}
where $ | \cdot |_{\gamma_{\infty}}$ is the $\gamma_{\infty}$-norm.

\vv

In view of the expression of the $Ric$ in terms of $u$, $A$, and $H$ in equations \eqref{ricci decomp} and \eqref{ricci decomp2}, we deduce the desired regularity of $Ric$, namely that it is $\mathcal{C}^1$ away from the boundary. We are now in a position to prove equation \eqref{regRic}. Using again equations \eqref{ricci decomp} and \eqref{ricci decomp2}, we estimate $|Ric|$ and $|\nabla Ric|$: for some universal constant $C>0$, we have 

\begin{equation}
|Ric| \leq C (|\nabla u|^2 + |A| + |H|), \quad |\nabla Ric| \leq C ( |\nabla^2 u|^2 + |\nabla A| + |\nabla H|)
\end{equation}

Using the asymptotics of $u$, $A$, and $H$ laid out in equations \eqref{reg-u}, \eqref{reg-A} and \eqref{reg-H}, we get that for $R>nm_0$, 

\begin{align}
\sup_{r>R} r^2( |Ric| + r|\nabla Ric|) &\leq C \sup_{S^2} \left( |R_{\gamma_{\infty}}-2| + |\slashed d R_{\gamma_{\infty}}|_{\gamma_{\infty}} \right) + o(R^{\delta})\\
& \leq C \norm{\gamma_{\infty} - \gamma_{\mathbb{S}^2}}_{\mathcal{C}^3(S^2)} + o(R^{\delta})\\
&\leq C \norm{\gamma_{\infty} - \gamma_{\mathbb{S}^2}}_{\mathcal{H}^k(S^2)} + o(R^{\delta})
\end{align}
as $R$ goes to $\infty$. In the last line, we have used again the Sobolev embeddings and the fact that $k\geq 5$. We have also allowed the constant $C$ to change from line to line while staying universal, i.e. independent of $g$ and $u$. 

\vvv

We turn our attention to the second statement of the proposition. It suffices to show that $X_0$ admits 3 derivatives. By virtue of the fact that $(g,u,X)$ is a modified solution, we have 

\begin{equation}
\partial_r trK + \frac{1}{2} trK^2 + |\hat K|^2 = -2(\partial_r u)^2
\end{equation}

\begin{equation}
\nabla_r \hat K + trK \hat K = -2\slashed du \otimes \slashed du - g(r) ((\partial_r u)^2 - |\nabla u|^2)
\end{equation}

Thanks to equation \eqref{reg-u}, we have that $|\nabla u| \in \HtoH{2}{\delta-1}{k-2}$ and so admits 2 radial derivatives. Hence, the above equations directly impy that $trK$ and $\hat K$ admit 3 radial derivative and, in fact, live in $\HtoH{3}{\delta-1}{k}$ and $\HtoHH{3}{\delta-1}{k}$ respectively. Due to the evolution equation 

\begin{equation}
\partial_r g(r) = trK \, g(r) + 2\hat K,
\end{equation}
we deduce that $g(r) \in \HtoHH{4}{loc}{k}$ and so admits 4 radial derivatives. By the Sobolev embedding and the fact that $k\geq 5$, this implies that $g(r), \partial_r g(r), \partial_r^2 g(r)$ are continuous in $r$ and are $\mathcal{C}^3$ on the sphere. We conclude that the metric $g = dr^2+g(r)$ is of class $\mathcal{C}^3$. Since $F(X) X$ satisfy the elliptic equation $\Delta_{g,conf} F(X) X = 0$ with respect to a $\mathcal{C}^3$ metric, standard localized interior estimates show that $F(X) X$ lives in $\mathcal{X}^3_{\delta}(M)$ (see for example \cite{choquet} appendix II). 


\end{proof}

\begin{remark}
Stronger regularity results can be proven for the modified solution $(g,u, X)$. Specifically, $g,u,$ and $X$ are in fact smooth away from the boundary. Nonetheless, the above regularity result is sufficient for our purpose. We will use it to show the nonexistence of nontrivial conformal Killing fields on $(M,g)$ (see lemma \eqref{confkill=0}). 
\end{remark}

\vv

\vv

We now show that $X = 0$. Letting $\bar X := F(X) X$, we note that $\Delta_{g, conf} \bar X = 0$ on $M$ in light of the fourth line in the definition of $\Phi$ and the fact that $(g,u,X)$ is a modified solution. We decompose $X$ as follows, 

\begin{equation}
X = X_0 + X_{\infty}
\end{equation}

where $X_0 \in \widehat{\mathcal{X}}^2_{\delta}(M)$ and $X_{\infty} \in \mathcal{X}_{\infty}$. We make the following observations: 

\begin{itemize}
\item If $\limsup_{r\to \infty} r^4|X|^2 =\infty$, then $F(X)$ decays exponentially in $r$. Since $|X| = |X_{\infty} + X_0|= O(r^2)$ by lemma \ref{xinf-lemma}, we deduce that $|\bar X|$ decays exponentially in $r$. 
\item If $\limsup_{r\to \infty} r^4 |X|^2 < \infty$, then $F(X) = O(1)$ and so $|\bar X| = O(r^{-2})$.
\end{itemize}

The above implies that $|\bar X| = O(r^-2)$. This allows us to perform the following integration-by-parts computation:

\begin{align} 
0 &= \int_{M} \bar X^{\mu} \Delta_{g,conf} \bar X_{\mu} dV_g\\
&= -\frac{1}{2} \int_{M} \left| \conflie{g}\bar X \right|^2 dV_g - \int_{\partial M} \conflie{g} \bar X (\bar X, \dd{r}) d\sigma_{g(nm_0)} \label{intbypartsX}
\end{align}

The above calculation is valid since $ \left| \conflie{g}\bar X \right|^2$ is integrable and the boundary integral at infinity vanishes. 

\vv

We now compute $\conflie{g} \bar X (\bar X, \dd{r})$ on $\partial M$ in order to evaluate the boundary integral in equation \eqref{intbypartsX}. Recall that we decomposed $X$ as follows, 

\begin{equation}
X = X_0 + X_{\infty}
\end{equation}

where $X_0 \in \widehat{\mathcal{X}}^2_{\delta}(M)$ and $X_{\infty} \in \mathcal{X}_{\infty}$. We also have that $X_0$ and $X_{\infty}$ satisfy the following on $\partial M$. 

\begin{equation}
X_{\infty} = X_{CK}, \quad \left( \mathcal{L}_{\dd{r}} X_0\right)^{T} = 0
\end{equation}
\begin{equation}
  \mathcal{L}_{\dd{r}} X_{\infty} = f'(n m_0) \cancel{div}_{\gamma_{\mathbb{S}^2}}(X_{CK}) \dd{r} + h'(nm_0) X_{CK} = \frac{1}{2}\cancel{div}_{\gamma_{\mathbb{S}^2}}(X_{CK}) \dd{r}
\end{equation}
\begin{equation}
g(X_0, \dd{r}) = g(X_{\infty}, \dd{r}) = 0
\end{equation}

Using the above, we compute that the integrand of the boundary integral appearing in equation \eqref{intbypartsX} satisfies the following on $\partial M$. 

\begin{align}
\conflie{g} \bar X(\bar X, \dd{r})  &= \conflie{g} X( X, \dd{r}) \\
&= \mathcal{L}_X g (X, \dd{r}) - \frac{2}{3} \text{div} X \,  g(X, \dd{r})\\
&= g(\mathcal{L}_{\dd{r}} X, X)\\
&= g(\mathcal{L}_{\dd{r}} X_0, X)+ g(\mathcal{L}_{\dd{r}} X_{\infty}, X)\\
&=0
\end{align}

where the fact that $\bar X = X$ near $\partial M$ was used in the first equality.  
It then follows that 

\begin{equation}
0 = \frac{1}{2} \int_{M} \left| \conflie{g} \bar X \right|^2 dV_g 
\end{equation}

implying that $\bar X$ is conformal Killing on $(M,g)$. However, the next lemma shows that if $\gamma_{\infty}$ is close enough to $\gamma_{\mathbb{S}^2}$, then there does not exist a non trivial conformal Killing field on $(M,g)$ that vanishes at $\infty$. 

\begin{lem} \label{confkill=0}
Let $\delta <0$. There exists $\epsilon>0$ such that the following holds. \\
Let $g \in M^k_{\delta} (M)$ satisfy the statement of proposition \eqref{regprop} (i.e. $Ric$ is $\mathcal{C}^1$ away from $\partial M$ and equation \eqref{regRic} holds). Suppose also that $\norm{\gamma_{\infty} - \gamma_{\mathbb{S}^2}}_{\mathcal{H}^k(S^2)} < \epsilon$. If $Z \in \mathcal{X}^{3}_{\delta}(M)$ is a conformal Killing vector field on $(M,g)$, then $Z = 0$. 
\end{lem}
\begin{proof}
Let  $Z \in \mathcal{X}^{3}_{\delta}(M)$ is a conformal Killing vector field on $(M,g)$. A direct computation shows that we can express the third covariant derivative of $Z$ as follows.

\begin{equation} \label{eqforZ}
\nabla^3 Z = A\cdot \nabla Z + B \cdot Z
\end{equation}

where $A$ and $B$ are linear expressions in $Riem$ and $\nabla Riem$, where $Riem$ is the Riemann curvature tensor. We will move the proof to the appendix to avoid digressing from the main discussion (refer to \ref{appendix CK}). Since the dimension of $M$ is $3$, ${Riem}$ can be written in terms of only $\text{Ric}$ and $g$, and so $A$ and $B$ can be thought of as linear expressions in $Ric$ and $\nabla Ric$. 

\vv

\vv

An application of a Hardy-type inequality shows that there exists an $R_0>nm_0$ depending only on $g$ and a positive constant $C$ depending only on $\delta$ such that for any $R\geq R_0$ and any vector field $Z \in \mathcal{X}^3_{\delta}(M)$, 
\begin{equation} \label{hardy}
\int_{[R,\infty)\times S^2} r^{-2\delta-3} (|Z|^2+ r^2|\nabla Z|^2) dV_g \leq C \int_{[R,\infty)\times S^2} r^{-2(\delta-3)-3} |\nabla^3 Z|^2 dV_g 
\end{equation}
A proof of this inequality is provided in section \ref{hardy appendix} in the Appendix (see corollary \ref{main hardy prop}).


\vv

On the other hand, given $R>nm_0$, equation \eqref{eqforZ} implies 
\begin{align}
 \int_{[R,\infty)\times S^2} r^{-2(\delta-3)-3} |\nabla^3 Z|^2 dV &\leq C \int_{[R,\infty)\times S^2} r^{-2(\delta-3)-3} (|\nabla Ric|^2 |Z|^2 + |Ric|^2 |\nabla Z|^2 ) dV \\
 &\leq C \left( \sup_{r> R} r^4(|Ric|^2 + r^2|\nabla Ric|^2) \right) \int_{[R,\infty)\times S^2} r^{-2\delta-3} (|Z|^2 + r^2|\nabla Z|^2)  dV \\
 &\leq C\left(\norm{\gamma_{\infty} - \gamma_{\mathbb{S}^2}}_{\mathcal{H}^k(S^2)}^2 + R^{2\delta}\right) \int_{[R,\infty)\times S^2} r^{-2\delta-3} (|Z|^2 + r^2|\nabla Z|^2)  dV  \\
 &\leq C(\epsilon+ R^{2\delta}) \int_{[R,\infty)\times S^2} r^{-2\delta-3} (|Z|^2 + r^2|\nabla Z|^2)  dV
\end{align}
where the constant $C$ is allowed to change from line third to fourth line while staying universal, {\it i.e.} independent of $Z$, $R$, and $g$. 

\vv
We have then proven that 
\begin{equation}
\int_{[R,\infty)\times S^2} r^{-2\delta-3} (|Z|^2+ r^2|\nabla Z|^2) dV  \leq C(\epsilon+R^{2\delta}) \int_{[R,\infty)\times S^2} r^{-2\delta-3} (|Z|^2+r^2 |\nabla Z|^2) dV 
\end{equation}

Choosing $\epsilon$ small enough and $R$ large enough so that $C(\epsilon+ R^{2\delta})<1$ implies that $Z=0$ on $[R,\infty)\times S^2$. 

\vv

Since $Z$ satisfies the elliptic equation  $\Delta_{g,conf} Z = 0$ and vanishes on an open set, standard arguments then imply that $Z=0$ on $M$ (see for example \cite{boostproblem}).
\end{proof}

\begin{remark}
The nonexistence of nontrivial conformal Killing vector fields vanishing at infinity on asymptotically flat manifolds has already been established in \cite{boostproblem} and \cite{york}. The above lemma extends this nonexistence result to a broader class of metrics, including some that are not asymptotically flat. 
\end{remark}
\vv

After possibly shrinking the neighbourhood ${ \mathcal{ U}}$ of $(\gammasc, \frac{1}{2} \trksc)$ and using the continuity of $\mathcal{H}$, we can assume that for $(\gammao, \frac{1}{2} \trko) \in \mathcal{U}$, the metric $g$ of the modified solution $\mathcal{H}(\gammao, \frac{1}{2} \trko) = (g,u,X)$ satisfies $\norm{\gamma_{\infty} - \gamma_{\mathbb{S}^2}}_{\mathcal{H}^k(S^2)} < \epsilon$. Since $\bar X$ is conformal Killing on $(M,g)$ and lives in $\mathcal{X}^3_{\delta}(M)$, the above lemma then implies that $\bar X = 0$, and hence $X = 0$. We then finally conclude, by remark \ref{X=0 iff}, that $(g,u)$ is a solution to the conformal static vacuum Einstein equations with Bartnik data $(\gammao, \frac{1}{2} \trko)$. This concludes the proof of the main theorem.


\vvv

\subsection{Proof that $D\Phi_{sc}$ is an Isomorphism \label{proofDPhi} }
In this section, we will prove proposition \ref{Dphi-iso}. 

\vv


We first remind the reader of the values of some key parameters for the Schwarzschild solution $(\mathfrak{g_{sc}}, f_{sc})$ and the conformal Schwarzschild solution 
$g_{sc}= f^2_{sc}\mathfrak{g_{sc}}$, $u_{sc}= \ln(f_{sc})$

\begin{multicols}{2}
\begin{itemize}
\item $\mathfrak{g_{sc}} = \left( 1-\frac{2m_0}{r}\right)^{-1} dr^2 + r^2 \gamma_{\mathbb{S}^2}$
\item $trK_{\mathfrak{g_{sc}}} = \frac{2 \sqrt{1-\frac{2m_0}{r}}}{r} $
\item $\gamma_{\mathfrak{g_{sc}}} = (nm_0)^2 \gamma_{\mathbb{S}^2}$ 
\item $f_{sc} = \sqrt{1-\frac{2m_0}{r}}$
\end{itemize}
\end{multicols}

\vv

\begin{multicols}{2}
\begin{itemize}
\item $g_{sc} = dr^2 + \left( 1-\frac{2m_0}{r} \right) r^2 \gamma_{\mathbb{S}^2} $
\item $\gamma_{sc} = n(n-2){m_0}^2 \gamma_{\mathbb{S}^2}$ 
\item $u_{sc} = \ln{\sqrt{1-\frac{2m_0}{r} }}$
\item $  trK_{sc} = \frac{2(r-m_0)}{r(r-2m_0)}$
\item $\hat K_{sc} = 0$
\item ${R_{\partial M}}_{sc} = \frac{2}{n(n-2){m_0}^2}$
\end{itemize}
\end{multicols}

\vv

Let $\tilde g \in T_{g_{sc}} \mathcal{M}^k_{\delta}$, $\tilde u \in \mathcal{A}^{(2,k+1)}_{\delta}(M)$, and $\tilde X \in \widehat{\mathcal{X}}^{2}_{\delta}(M)$. For small $t$, let $g(t)$, $u(t)$, and $X(t)$ be smooth 1-parameter families satisfying

\begin{multicols}{2}
\begin{itemize}
\item $g(0) = g_{sc}$
\item $u(0) = u_{sc}$
\item $X(0) = 0$
\item $g'(0) = \tilde g$
\item $u'(0)= \tilde u$
\item $X'(0) = \tilde X$
\end{itemize}
\end{multicols}

Define the following 

\begin{multicols}{1}
\begin{itemize}
\item $\widetilde{trK} := \ddt trK(t)$
\item $\widetilde{\hat K} := \ddt \hat K(t)$
\item $\tilde \gamma := \ddt \tilde g(t)(nm_0)$
\item $\tilde \omega := \ddt \omega(g(t), X_{\infty}(t) )$
\end{itemize}
\end{multicols}

where $X_{\infty}(t)$ is the projection of $X(t)$ into the space $\mathcal{X}_{\infty}$. By definition of $\omega$, we have that $\tilde \omega$ is a conformal Killing field on $(S^2, g_{sc}(nm_0))$. 

\vv

We compute $D\Phi_{sc}$ to be 

\begin{multline*}
D \Phi_{sc}: T_{g_{sc}}\mathcal{M}^k_{\delta} \times \mathcal{A}^{(2,k+1)}_{\delta}(M) \times \bigg( \widehat{\mathcal{X}}^{2}_{\delta}(M) \oplus \mathcal{X}_{\infty}(M)\bigg)\\
  \to  \mathcal{A}^{(0,k-1)}_{\delta-2}(M) \times  L^2_{\delta-2} \bigg( [nm_0,\infty); H^k(S^2) \bigg) \times  L^2_{\delta-2} \bigg( [nm_0,\infty) ; \mathcal{H}^k (S^2) \bigg)\\
\times \mathcal{X}^{0}_{\delta-2}(M) \times H^{k-1}(\partial M) \times \Omega^{k-1}(\partial M) \times \mathcal{H}^{k}(\partial M) \times  H^{k}(\partial M) 
 \end{multline*}
 
 \begin{align} 
 D \Phi_{sc} (\tilde g, \tilde u, \tilde X) &=  \left. \frac{d}{dt} \right|_{t=0} \Phi(\gammasc, \frac{1}{2} \trksc, g(t), u(t), X(t) ) \nonumber  \\ 
 &=  \begin{pmatrix} \Delta_{g_{sc}} \tilde u +  (\partial_r u_{sc}) ( \widetilde{trK}) \lv \\ 
\partial_r \widetilde{trK} + trK_{sc} \widetilde{trK}  + 4(\partial_r u_{sc})( \partial_r \tilde u) \lv \\
\mathcal{L}_{\dd{r}} {\widetilde{\hat K}}   \lv \\
\Delta_{g_{sc}, conf} \tilde X  \lv \\ 
 - 4(\partial_r u_{sc}) (\partial_r \tilde u) + trK_{sc} \left. \widetilde{trK} \right|_{\partial M}  +  \frac{4}{n(n-2)m_0^2} \tilde u +2 \slashed \Delta_{\gamma_{sc}} \tilde u \lv \\
2(\partial_r u_{sc}) \slashed d\tilde u - \cancel{div} (\widetilde{\hat K}) + \tilde \omega  \lv \\
 \frac{n}{n-2} \tilde \gamma - 2 n^2m_0^2 \, \tilde u \,  g_{S^2} \lv \\
\left. \widetilde{trK} \right|_{\partial M} + \frac{2}{nm_0} \tilde u - 2\partial_r \tilde u \end{pmatrix}
 \end{align}
 
Note that for the third line in $D\Phi_{sc}$, we used the fact that 

\begin{equation}
\nabla_r \widetilde{\hat K} + trK \widetilde{\hat K} = \mathcal{L}_{\dd{r}} \widetilde{\hat K}
\end{equation}

since $\hat K_{sc} = 0$. \\

\vv

Let $(A,B,C,D,E,F,G,H)$ be an arbitrary element in the codomain of $D \Phi_{sc}$. We wish to show that there exists a unique $(\tilde g, \tilde u, \tilde X)$ in the domain satisfying 

\begin{equation} \label{1st}
\Delta_{g_{sc}} \tilde u +  (\partial_r u_{sc}) ( \widetilde{trK}) = A, \quad \text{on $M$}
\end{equation}
 
 \begin{equation} \label{2nd}
\partial_r \widetilde{trK} + trK_{sc} \widetilde{trK}  + 4(\partial_r u_{sc})( \partial_r \tilde u) = B, \quad \text{on $M$}
\end{equation}

\begin{equation} \label{3rd}
\mathcal{L}_{\dd{r}} {\widetilde{\hat K}} = C, \quad \text{on $M$}
\end{equation}

\begin{equation} \label{4th}
\Delta_{g_{sc}, conf} \tilde X = D, \quad \text{on $M$}
\end{equation}

\begin{equation} \label{5th}
 - 4(\partial_r u_{sc}) (\partial_r \tilde u) + trK_{sc} \left. \widetilde{trK} \right|_{\partial M}+\frac{4}{n(n-2)m_0^2} \tilde u + 2\slashed \Delta_{\gamma_{sc}} \tilde u = E, \quad \text{on $\partial M$}
\end{equation}

\begin{equation}\label{6th}
2(\partial_r u_{sc}) \slashed d\tilde u - \cancel{div} (\widetilde{\hat K}) + \tilde \omega = F, \quad \text{on $\partial M$}
\end{equation}

\begin{equation} \label{7th}
 \frac{n}{n-2} \tilde \gamma - 2 n^2m_0^2 \, \tilde u \,  g_{S^2}=G, \quad \text{on $\partial M$}
\end{equation}

\begin{equation} \label{8th}
\left. \widetilde{trK} \right|_{\partial M} + \frac{2}{nm_0} \tilde u - 2\partial_r \tilde u =H, \quad \text{on $\partial M$}
\end{equation}

The above equations can be decoupled to give a non-local elliptic system on $\tilde u$.

\begin{lem}\label{decouple}   Let $\tilde u \in \mathcal{A}^{(2,k+1)}_{\delta}(M)$ and $\tilde g \in T_{g_{sc}} \mathcal{M}^k_{\delta}$ satisfy equations \eqref{2nd} and \eqref{8th}. 

Then $\tilde u$ and $\tilde g$ satisfy equations \eqref{1st} and \eqref{5th} if and only if $\tilde u$ satisfies

 \begin{equation}
 \label{eq1}
\Delta_{g_{sc}} \tilde u-\frac{4m_0^2}{[r(r-2m_0)]^2} \tilde u = -\frac{n(n-2)m_0}{2}\frac{4m_0^2}{[r(r-2m_0)]^2} \left(\frac{4-n}{n(n-2)m_0} \tilde u|_{\partial M}+ \partial_r\tilde u|_{\partial M} \right)+\psi, \quad \text{on $M$}
\end{equation}

 \begin{align} \label{eq2}
\frac{4}{nm_0} \partial_r \tilde u + 2\slashed \Delta \tilde u+ \frac{4}{n^2(n-2)m_0^2} \tilde u =\Gamma , \quad \text{on $\partial M$}
\end{align}

where $\psi$ and $\Gamma$ are defined as follows

\begin{equation} \label{psi}
\psi := A -\frac{m_0}{r^2(r-2m_0)^2} \int_{nm_0}^r s(s-2m_0) B \, ds - \frac{m_0^3n(n-2)}{r^2(r-2m_0)^2} H 
\end{equation}
\begin{equation} \label{Gamma}
\Gamma := E - \frac{2(n-1)}{n(n-2)m_0} H
\end{equation}
\end{lem}

\begin{remark}
In light of the spaces that $A$, $B$, $E$ and $H$ live in, it follows that $\psi \in \mathcal{A}^{(0,k-1)}_{\delta-2}(M)$ and $\Gamma \in H^{k-1}(S^2)$. Note that the boundary value problem that $\tilde u$ satisfies in equations \eqref{eq1} and \eqref{eq2} does not depend on $\tilde g$ or $\tilde X$; hence, we have indeed decoupled the system. 
\end{remark}
\begin{proof}
We directly deduce that equations \eqref{eq2} and \eqref{5th} are equivalent by using equation \eqref{8th}. 

\vv

We rewrite equation \eqref{2nd} as follows: 
\begin{equation} \label{DG2a}
\partial_r \left(  \exp\left(\int_{nm_0}^r trK_{sc}(s) ds\right) \widetilde{trK}\right) = -\exp\left(\int_{nm_0}^r trK_{sc}(s) ds\right)\Big(4(\partial_r u_{sc}) (\partial_r \tilde u) - B\Big).
\end{equation}

A direct computation gives:
\begin{equation} 
\exp\left(\int_{nm_0}^r trK_{sc}(s) ds\right) = \frac{1}{n(n-2)m_0^2}r(r-2m_0).
\end{equation}

We integrate equation \eqref{DG2a} to get an expression for $\widetilde{trK}$ in terms of $\tilde u$: for $r\in [nm_0,\infty)$ and $p \in S^2$, we have 
 \begin{align}  
\widetilde{trK} (r,p) &= \frac{n(n-2)m_0^2}{r(r-2m_0)}\left. \widetilde{trK} \right|_{\partial M}(p) - 4\frac{n(n-2)m_0^2}{r(r-2m_0)} \int_{nm_0}^r \frac{1}{n(n-2)m_0^2}s(s-2m_0) (\partial_r u_{sc})(s) (\partial_r \tilde u)(s,p) ds \nonumber\\ 
&\qquad  + \frac{n(n-2)m_0^2}{r(r-2m_0)} \int_{nm_0}^r \frac{1}{n(n-2)m_0^2}s(s-2m_0) B(s,p)  ds \\
&= \frac{n(n-2)m_0^2}{r(r-2m_0)}\left. \widetilde{trK} \right|_{\partial M}(p) - 4\frac{m_0^2}{r(r-2m_0)} \int_{nm_0}^r \frac{1}{m_0^2}s(s-2m_0) \left(\frac{m_0}{s(s-2m_0)}\right) (\partial_r \tilde u)(s,p) ds \nonumber\\
&\qquad  + \frac{1}{r(r-2m_0)} \int_{nm_0}^r s(s-2m_0) B(s,p)  ds \\
&= \frac{n(n-2)m_0^2}{r(r-2m_0)}\left. \widetilde{trK} \right|_{\partial M}(p) - \frac{4m_0}{r(r-2m_0)} \int_{nm_0}^r \partial_r \tilde u(s,p) \, ds + \frac{1}{r(r-2m_0)} \int_{nm_0}^r s(s-2m_0) B(s,p)  ds\\
&= \frac{n(n-2)m_0^2}{r(r-2m_0)}\left. \widetilde{trK} \right|_{\partial M}(p) - \frac{4m_0}{r(r-2m_0)} \Big(\tilde u(r,p) - \tilde u|_{\partial M}(p)\Big) + \frac{1}{r(r-2m_0)} \int_{nm_0}^r s(s-2m_0) B(s,p)  ds
\end{align}

We plug this into equation \eqref{1st} to derive 

\begin{align}
A &= \Delta_{sc} \tilde u + \frac{m_0}{r(r-2m_0)} \left( \frac{n(n-2)m_0^2}{r(r-2m_0)}\left. \widetilde{trK} \right|_{\partial M} - \frac{4m_0}{r(r-2m_0)} \Big(\tilde u - \tilde u|_{\partial M}\Big)\right) \nonumber\\
&\qquad + \frac{m_0}{r^2(r-2m_0)^2} \int_{nm_0}^r s(s-2m_0) B(s,p)  ds \\
&= \Delta_{sc} \tilde u - \frac{4m_0^2}{r^2(r-2m_0)^2} \tilde u + \frac{4m_0^2}{r^2(r-2m_0)^2} \left( \tilde u|_{\partial M} + \frac{n(n-2)m_0}{4} \left.\widetilde{trK}\right|_{\partial M} \right)\nonumber \\
&\qquad  + \frac{m_0}{r^2(r-2m_0)^2} \int_{nm_0}^r s(s-2m_0) B(s,p)  ds
\end{align}

Using equation \eqref{8th}, it then follows that $\tilde u$ satisfies:

\begin{align}
\Delta_{g_{sc}} \tilde u-\frac{4m_0^2}{r^2(r-2m_0)^2} \tilde u &= A-\frac{n(n-2)m_0}{2}\frac{4m_0^2}{r^2(r-2m_0)^2} \left( \frac{4-n}{n(n-2)m_0} \tilde u|_{\partial M}+ \partial_r\tilde u|_{\partial M} \right) \nonumber\\
&\qquad - \frac{n(n-2)m_0^3}{r^2(r-2m_0)^2} H - \frac{m_0}{r^2(r-2m_0)^2} \int_{nm_0}^r s(s-2m_0) B(s,p)  ds\\
&= -\frac{n(n-2)m_0}{2}\frac{4m_0^2}{r^2(r-2m_0)^2} \left( \frac{4-n}{n(n-2)m_0} \tilde u|_{\partial M}+ \partial_r\tilde u|_{\partial M} \right) + \psi
\end{align}

This proves that equation \eqref{eq1} is equivalent to equation \eqref{1st}. 

\end{proof}



\vv

The rest of the proof will proceed in the following steps. 

\begin{enumerate}[label=\textbf{Step \arabic*: }]
\item We will show that for every $\psi \in \mathcal{A}^{(0,k-1)}_{\delta-2}(M)$ and $\Gamma \in H^{k-1}(S^2)$, there exists a unique solution $\tilde u \in {\mathcal{A}^{(2,k+1)}_{\delta}}$ solving equations \eqref{eq1} and \eqref{eq2}. 


\item We will show if $\tilde u \in {\mathcal{A}^{(2,k+1)}_{\delta}}$ satisfies \eqref{eq1} and \eqref{eq2} with $\psi$ and $\Gamma$ given by equations \eqref{psi} and \eqref{Gamma}, then there exists a unique $\tilde g \in T_{g_{sc}}\mathcal{M}^k_{\delta}(M)$ and a unique conformal Killing field $\tilde \omega$ satisfying equations \eqref{1st} to \eqref{3rd} and \eqref{5th} to \eqref{8th}. 

\item We will show that there exists a unique vector field $\tilde Y \in \widehat{\mathcal{X}}^{2}_{\delta}(M)$ satisfying $\Delta_{g_{sc}, conf} \tilde Y =D$. 
\end{enumerate}

 The above 3 steps will then imply that there exists a unique $(\tilde g, \tilde u, \tilde X)$ in the domain of $D\Phi_{sc}$ solving equations \eqref{1st} to \eqref{8th}. In particular, $\tilde u$ and $\tilde g$ are achieved from steps 1 and 2 respectively and $\tilde X:= \tilde Y + \tilde \omega$, where $\tilde Y$ and $\tilde \omega$ are achieved from steps 2 and 3 respectively. 

\vvv

\vvv

{ \Large \bf Step 1: Solving for $\tilde u$} \\

To study the boundary value problem in \eqref{eq1} and \eqref{eq2}, we will investigate the properties of the corresponding non-local elliptic operator $\mathcal{P}_{sc}$, which maps $\mathcal{A}^{(2,k+1)}_{\delta}(M)$ into $\mathcal{A}^{(0,k-1)}_{\delta-2}(M)\times H^{k-1}(\partial M)$ and is defined by 

\begin{equation} \label{Psc}
\mathcal{P}_{sc}( \tilde u) := \begin{pmatrix} \Delta_{g_{sc}} \tilde u-\frac{4m_0^2}{[r(r-2m_0)]^2} \tilde u  +\frac{n(n-2)m_0}{2}\frac{4m_0^2}{[r(r-2m_0)]^2} \left(\frac{4-n}{n(n-2)m_0} \tilde u|_{\partial M}+ \partial_r\tilde u|_{\partial M} \right) \lv  \\
\frac{2}{nm_0} \partial_r \tilde u + \slashed \Delta \tilde u+ \frac{2}{n^2(n-2)m_0^2} \tilde u \end{pmatrix} 
\end{equation}

In fact, this operator will turn out to be Fredholm of index 0 as shown in the following proposition. 

\begin{remark}
In \cite{an-huang} and \cite{an-huang2}, the authors study the static Einstein vacuum equations in a gauge different from the one used in this paper. Specifically, they study an operator analogous to the operator $\Phi$ considered here. They achieve that the linearization of their operator is an isomorphism, so as to invoke the implicit function theorem, by first establishing that it is Fredholm of index $0$ and then showing that its kernel is trivial. Our approach here is similar except that our gauge allows us to decouple the equations; this decoupling reduces the task of proving that $D\Phi_{sc}$ is an isomorphism to proving that a much simpler operator, $\mathcal{P}_{sc}$ acting on the linearization of the lapse function $\tilde u$, is an isomorphism. Specifically, we will establish that $\mathcal{P}_{sc}$ is Fredholm of index $0$ and has a trivial kernel. The remaining parameters, $\tilde g$ and $\tilde X$, are governed by straightforward ODEs with $\tilde u$ appearing in the forcing term, and the fact that $D\Phi_{sc}$ is an isomorphism will follow readily (see \textbf{Step 2}). 
\end{remark}

\begin{prop}
Fix $\delta \in (-1,-\frac{1}{2}]$ and $k\in \Z_{\geq0}$. Let $T: \mathcal{A}^{(2,k+1)}_{\delta}(M) \to H^{k-1}(\partial M)$ and $S: \mathcal{A}^{(2,k+1)}_{\delta}(M) \to H^{k-1}(\partial M)$  be operators defined by 
\begin{equation}
T(\tilde u) := \partial_r \tilde u  + \mu \tilde u
\end{equation}
\begin{equation} 
S(\tilde u) := \slashed \Delta_{\gamma_{sc}} \tilde u + \beta_1 \partial_r \tilde u + \beta_2 \tilde u
\end{equation}
where $ \mu, \beta_1, \beta_2 \in H^{k}(\partial M)$. \\

Let $\mathcal{P}$ be the nonlocal elliptic differential operator defined by 
 $$\mathcal{P}: \mathcal{A}^{(2,k+1)}_{\delta}(M)\to \mathcal{A}^{(0,k-1)}_{\delta-2}(M) \times H^{k-1}(\partial M) $$
\begin{equation}
\mathcal{P}(\tilde u) := \begin{pmatrix} \Delta_{g_{sc}} \tilde u - V_1\tilde u - V_2 \overline{T(\tilde u)} \lv \\  S(\tilde u) \end{pmatrix}
\end{equation}

where $V_1, V_2 \in \HtoH{1}{-3}{k}$ and $\overline{T(\tilde u)}$ is the function on $M$ defined by $(r,p) \mapsto T(\tilde u)(p)$ for $(r,p) \in M$. Then $\mathcal{P}$ is Fredholm of index $0$. 
\end{prop}
\begin{proof}

\vv

Decompose the operator $\mathcal{P} = \mathcal{P}_1+ \mathcal{P}_2$, where the operators $\mathcal{P}_1, \mathcal{P}_2:  \mathcal{A}^{(2,k+1)}_{\delta}(M)\to \mathcal{A}^{(0,k-1)}_{\delta-2}(M)\times H^{k-1}(\partial M)$  are defined by
\[ \mathcal{P}_1 (\tilde  u) := \begin{pmatrix} \Delta_{g_{sc}} \tilde u \lv \\ \slashed \Delta_{\gamma_{sc}} \tilde u \end{pmatrix}, \qquad  \mathcal{P}_2 (\tilde  u) := \begin{pmatrix} -V_1 \tilde u - V_2 \overline{T(\tilde u)} \lv \\ \beta_1 \partial_r \tilde u + \beta_2 \tilde u \end{pmatrix} \]

In light of proposition \eqref{prop-spaces}, we observe that $\mathcal{P}_2$ is a compact operator. Indeed, for any $\tilde u \in {\mathcal{A}^{(2,k+1)}_{\delta}}$, we have that $-V_1 \tilde u - V_2 \overline{T(\tilde u)}$ lives in $\HtoH{1}{-3}{k}$ which compactly embeds in $\mathcal{A}^{(0,k-1)}_{\delta-2}(M)$, and $\beta_1 \partial_r \tilde u + \beta_2 \tilde u$ lives in $H^{k}(\partial M)$ which compactly embeds in $H^{k-1}(\partial M)$. To show that $\mathcal{P}$ is Fredholm of index 0, it then suffices to show that $\mathcal{P}_1$ is Fredholm of index 0 (see \cite{fredholm})

\vvv



\vv

By theorem \ref{elliptic thm}, the operator $\mathcal{Q}$ defined by: 

\[\mathcal{Q} (\tilde  u) := \begin{pmatrix} \Delta_{g_{sc}} \tilde u \lv \\ \tilde u \end{pmatrix} \]

is an isomorphism from ${\mathcal{A}^{(2,k+1)}_{\delta}}$ to $\mathcal{A}^{(0,k-1)}_{\delta-2}(M)\times H^{k+1}(\partial M)$. 

\vv

We also recall the following standard result on the laplacian on compact manifolds: The operator $\slashed \Delta_{\gamma_{sc}}: H^{k+1}(\partial M) \to H^{k-1}(\partial M)$ is Fredholm of index $0$. In fact, the kernel is the one-dimensional space of constant functions on $\partial M$ and the cokernel is the same since $\slashed \Delta_{\gamma_{sc}}$ is self-adjoint.

\vv

We observe that $\mathcal{P}_1 = (\mathcal{ID}, \slashed \Delta_{\gamma_{sc}}) \circ \mathcal{Q}$, where 
$$(\mathcal{ID}, \slashed \Delta_{\gamma_{sc}}): \mathcal{A}^{(0,k-1)}_{\delta-2}(M) \times H^{k+1}(\partial M) \to \mathcal{A}^{(0,k-1)}_{\delta-2}(M) \times H^{k-1}(\partial M)$$
 is defined by 
$$(\mathcal{ID}, \slashed \Delta_{\gamma_{sc}}) (\tilde v, f) := (\tilde v, \slashed \Delta_{\gamma_{sc}} f)$$
Since $(\mathcal{ID}, \slashed \Delta_{\gamma_{sc}})$ is Fredholm of index 0 and $\mathcal{Q}$ is an isomorphism, it then follows that $\mathcal{P}_1$ is Fredholm of index 0 as needed. 

\end{proof}

\vv

By the above proposition, showing that the nonlocal operator $\mathcal{P}_{sc}$ defined in \eqref{Psc} has trivial kernel is sufficient to prove that the system in \eqref{eq1} and \eqref{eq2} is uniquely solvable for every $\psi \in \mathcal{A}^{(0,k-1)}_{\delta-2}(M)$ and $\Gamma \in H^{k-1}(S^2)$. This will be the content of the next proposition.

\begin{prop} \label{prop-nonlocalpde-unique}
Fix $\delta \in (-1, -\frac{1}{2}]$. Let $u \in {\mathcal{A}^{(2,k+1)}_{\delta}}$ satisfy
\begin{equation} \label{nonlocalpde-homo}
\begin{cases}
\Delta_{g_{sc}} \tilde u-\frac{4m_0^2}{[r(r-2m_0)]^2} \tilde u = -\frac{n(n-2)m_0}{2}\frac{4m_0^2}{[r(r-2m_0)]^2} \left(\frac{4-n}{n(n-2)m_0} \tilde u|_{\partial M}+ \partial_r\tilde u|_{\partial M} \right), & \text{on $M$}\\
\frac{4}{nm_0} \partial_r \tilde u + 2\slashed \Delta \tilde u+ \frac{4}{n^2(n-2)m_0^2} \tilde u  = 0, & \text{on $\partial M$} 
\end{cases}
\end{equation}
Then $\tilde u=0$.
\end{prop}
\begin{proof}
Similarly to what was done in the proof of lemma \eqref{main-est}, we utilize the spherical symmetry of the conformal Schwarzschild metric to reduce the system in \eqref{eq1} and \eqref{eq2} to differential equations on the coefficients of $\tilde u$ with respect to its spherical harmonics decomposition: 

\vv

\begin{equation}\label{u sph harm}
\ds \tilde u(r, x) = \sum_{\ell=0}^{\infty} \sum_{m=-\ell}^{\ell} \tilde a_{m\ell}(r) Y_{m\ell}(x) 
\end{equation}
for $r \in [nm_0,\infty)$ and $x\in \partial M$. 

\vv



We define the functions $a_{m\ell} (r) := \tilde a_{m\ell}(nm_0 r)$ on $[1,\infty)$. Using the discussion in lemma \eqref{main-est}, the condition $\norm{\tilde u}_{\mathcal{A}^{(2,k+1)}_{\delta}}< \infty$ in particular implies $\norm{a_{m\ell}}_{1,\delta} < \infty$ for every $m$ and $\ell$, where 

\begin{equation}
\norm{a_{m\ell}}^2_{1,\delta} = \int_{1}^{\infty} r^{-2\delta+1} (a_{m\ell}'(r))^2 dr + \int_{1}^{\infty} r^{-2\delta-1}(a_{m\ell}(r))^2 dr
\end{equation}

\vv

The system in \eqref{eq1} and \eqref{eq2} as well as the condition $\norm{a_{m\ell}}_{1,\delta}<\infty$ implies the following non-local differential equations on all the coefficients $\tilde a_{m\ell}(r)$: 

\begin{align} \label{ode-333}
\begin{cases}  \ds
r(r-2m_0) \tilde a_{m\ell}''(r) + 2(r-m_0) \tilde a_{m\ell}'(r) - \left( \frac{4m_0^2}{r(r-2m_0)}+\ell(\ell+1)\right) \tilde a_{m\ell}(r) \\ 
\hspace{2cm} = -\frac{2n(n-2)m_0^3}{r(r-2m_0)} \bigg(  \frac{4-n}{n(n-2)m_0} \tilde a_{m\ell}(nm_0)+ \tilde a_{m\ell}'(nm_0) \bigg), {\quad} r\in [n m_0,+\infty)  \\
\frac{2}{nm_0} \tilde a_{m\ell}'(nm_0) -\frac{1}{n(n-2)m_0^2} (\ell(\ell+1) -\frac{2}{n}) \tilde a_{m\ell}(nm_0) = 0,\\
\norm{ a_{m\ell}}_{1,\delta} < \infty
\end{cases}
\end{align}

The functions $a_{m\ell}$, using \eqref{ode-333}, 
satisfy the following similar 
non-local differential equations: 

\begin{align} \label{ode-33}
\begin{cases} 
r(r-\frac{2}{n}) a_{m\ell}''(r) + 2(r-\frac{1}{n}) a_{m\ell}'(r) - \left( \frac{4}{n^2r(r-\frac{2}{n})}+\ell(\ell+1)\right) a_{m\ell}(r) \\ 
\hspace{6cm} = -\frac{2(n-2)}{n^2r(r-\frac{2}{n})} \bigg(  \frac{4-n}{n-2} a_{m\ell}(1)+ a_{m\ell}'(1) \bigg), \\
\frac{2}{n^2} a_{m\ell}'(1) -\frac{1}{n(n-2)} (\ell(\ell+1) -\frac{2}{n}) a_{m\ell}(1) = 0,\\
\norm{a_{m\ell}}_{1,\delta} < \infty
\end{cases}
\end{align}

Note that the mass parameter $m_0$ does not appear in the non-local differential equation for 
$a_{m\ell}$. From here onwards we will study the system \eqref{ode-33} instead of \eqref{ode-333}. 
\vv

\vv

We consider \eqref{ode-33} for any nonnegative integer $\ell$ and seek to derive that $a_{m\ell}(r)=0$ is the only solution. 
We consider the $\ell=0$ case separately. We replace the last condition in \eqref{ode-33} with $a_{m\ell}'(1) = C$ and find the explicit (unique) solution to this shooting problem to be

\begin{equation} \label{sol-a00}
a_{00}(r) = -C\, \frac{-2+n+6 r-3 nr-2 nr^2+n^2 r^2}{r(nr-2)},
\end{equation}

It then can easily be verified that $\lim_{r\to \infty}a_{00}(r) = 0$ if and only if $C=0$, implying that $a_{00} = 0$ is the only solution to \eqref{ode-33}.

\vv


\vv

We conclude the only solutions to \eqref{ode-33} in the $\ell=0$ case 
is the zero solution. 

\vvv

We now deal with the $\ell\geq 1$ case. We will write $a_{\ell}$ instead of $a_{m \ell}$ for simplicity. 

Define the following constants: 
\begin{itemize}
 \item \[ C_{\ell} = -\frac{2(n-2)}{n^2} \left( \frac{4-n}{n-2} a_{\ell}(1) + a_{\ell}'(1) \right)\] 
 \item \[ \alpha_{\ell} = \frac{n^2}{n(2-\ell(\ell+1)) - 6}\] 
 \item \[ \beta_{\ell} = \frac{n^2(n \ell(\ell+1) - 2)}{(n(2-\ell(\ell+1)) -6)(2(n-2))} \] 
 \end{itemize}

Note that $\alpha_{\ell}$ and $\beta_{\ell}$ are well defined for $\ell\geq 1$ and are both negative. 

Then the function $a_{\ell}$ solves the following initial value problem: 

\begin{equation} \label{ode-3}
\begin{cases}
r(r-\frac{2}{n}) a_{\ell}''(r) + 2(r-\frac{1}{n}) a_{\ell}'(r) - \left( \frac{4}{n^2r(r-\frac{2}{n})}+l(l+1)\right) a_{\ell}(r) =  \frac{1}{r(r-\frac{2}{n})} C_{\ell}  \\
a_{\ell}'(1) = \beta_{\ell} C_{\ell}\\
a_{\ell}(1) = \alpha_{\ell} C_{\ell}
\end{cases}
\end{equation}

If $C_{\ell} = 0$, then $a_{\ell}=0$ by the existence and uniqueness theorem from ODE theory. Suppose now that $ C_{\ell}\neq 0$. By considering $\frac{a_{\ell}}{C_{\ell}}$ instead of $a_{\ell}$, we can assume without loss of generality that $C_{\ell}=1$. 

We rewrite the system in the following way: 

\begin{equation} \label{ode-3}
\begin{cases}
\frac{d}{dr} \left[ r(r-\frac{2}{n}) a_{\ell}'(r)  \right] = \left( \frac{4}{n^2r(r-\frac{2}{n})}+\ell(\ell+1)\right) a_{\ell}(r) +  \frac{1}{r(r-\frac{2}{n})}   \\
a_{\ell}'(1) = \beta_{\ell} \\
a_{\ell}(1) = \alpha_{\ell}
\end{cases}
\end{equation}

We will show that $a_{\ell}$ blows up at infinity contradicting that $\norm{a_{\ell}}_{1,\delta}<\infty$. This will then imply that $C_{\ell} = 0$ and hence $a_{\ell}=0$. First, we prove a technical lemma.\\

\begin{lem} \label{tech}
Let $h_1,h_2$ be smooth functions on $[1,\infty)$ such that $h_1$ is positive and $\lim_{r\to \infty}h_1(r) = C_1$ for some $C_1>0$. Let $g$ be a function on $[1,\infty)$ satisfying the following ODE: 
\begin{equation} \label{generalode}
\frac{d}{dr} \left[ r(r-\frac{2}{n}) g'(r)  \right] = h_1(r)g(r) + h_2(r)
\end{equation}
Then the following is true. 
\begin{itemize}
\item Suppose that $h_2(r) \geq 0$ ($\leq 0$) on $[1,\infty)$ and that both $g(r_*)$ and $g'(r_*)$ are positive (negative) for some $r_* \in [1,\infty)$. Then $g$ and $g'$ are positive (negative) on $(r_*,\infty)$. 
\item Suppose that $g$ and $g'$ are positive (negative) on $(r_*, \infty)$ for some $r_*\in [1,\infty)$ and that $h_2(r) = O(r^{-2})$. Then $\lim_{r\to \infty} g(r) = \infty$ ($-\infty$).
\end{itemize}
\vv

\end{lem}
\begin{proof}
Suppose that $h_2$ is nonnegative everywhere. Let $r_* \in [1,\infty)$ be such that $g'(r_*) >0$ and $g(r_*) >0$. We will prove that $g'(r)>0$ on $(r_*, \infty)$ by using a simple bootstrap method.  By continuity of $g'(r)$, we know that $g'(r)>0$ on $[r_*, r_*+\delta)$ for some $\delta>0$. Then the set $B:= \{ r \in (r_*, \infty) \mid g'(s) >0 \text{ for $s \in [r_*, r)$}\}$ is nonempty. Suppose that $R:= \sup B < \infty$. By continuity, we have that $g'(R) = 0$ and $g'(r) > 0$ for $r\in (r_*, R)$. Since $g(r_*)>0$ and $g$ is increasing on $(r_*, R)$, we have that $g(R) >0$. By letting $r=R$ in equation \eqref{generalode}, it follows that $\frac{d}{dr} [ r(r-2/n)g']$ is positive at $R$ and, in turn, on a neighbourhood of $R$. This implies that $r(r-2/n)g'(r)$ is increasing on a neighbourhood of $R$, which implies that $0< r(r-2/n)g'(r) < R(R-2/n)g'(R)$ for $r<R$ and close to $R$. As this contradicts that $g'(R) = 0$, we conclude that $\sup B = \infty$ and hence $g'$ and $g$ are positive on $(r_*, \infty)$. 

\vv

Suppose now that $g$ and $g'$ are positive on $(r_*, \infty)$ for some $r_* \in [1,\infty)$ and that $h_2(r) =O(r^{-2})$. In virtue of the positivity of $g'$ and $g$ as well as the monotone convergence theorem, it follows that $\lim_{r\to \infty} g(r)$ either is a positive number or is $\infty$. Suppose that $\lim_{r \to \infty} g(r) = A >0$, which in particular implies that $g'$ is integrable.

\vv

By integrating equation \eqref{generalode}, we get that

\begin{equation}
r(r-2/n)g'(r) = \frac{n-2}{n} g'(1) + \int_{1}^r [h_1(s) g(s) + h_2(s)] ds
\end{equation}

Using the fact that $h_1$ and $g$ are $O(1)$ and that $h_2(r) = O(r^{-2})$, it follows that $\sup_{r\geq1} r g'(r) < \infty$. 

\vv

Furthermore, we have that

\begin{align}
rg'(r) - g'(1) &= \int_1^r [sg'(s)]' ds\\
&= \int_1^r g'(s) ds + \int_{1}^r sg''(s) ds \\
&= \int_1^r \frac{h_1(s) g(s)}{s-2/n} ds + \int_1^r\frac{h_2(s)}{s-2/n} ds - \int_1^r \frac{s}{s-2/n} g'(s) ds
\end{align}

It follows that there exists a constant $M>0$ such that 

\begin{equation}
\int_1^r \frac{h_1(s) g(s)}{s-2/n} ds \leq M
\end{equation}
for any $r \in [1,\infty)$; since the integrand is positive, the limit as $r$ tends to $\infty$ exists. In particular, this implies that $\frac{g}{r} \in L^1([1,\infty))$, which then contradicts that $A$ is positive. Hence, we see  that $\lim_{r\to \infty} g(r) = \infty$, as needed.

\vv

The case when $h_2$ is nonpositive is identical. 

\vv

\end{proof}

\begin{cor} \label{cortech}
Let $h_1,h_2$ be smooth functions on $[1,\infty)$ such that $h_1$ is positive and $\lim_{r\to \infty}h_1(r) = C_1$ for some $C_1>0$. Let $g$ and $\tilde g$ be functions on $[1,\infty)$ satisfying the following ODE: 
\begin{equation}
\frac{d}{dr} \left[ r(r-\frac{2}{n}) g'(r)  \right] = h_1(r)g(r) + h_2(r), \qquad \frac{d}{dr} \left[ r(r-\frac{2}{n}) \tilde g'(r)  \right] = h_1(r)\tilde g(r) + h_2(r),
\end{equation}
If $g(1) = \tilde g(1)$ and $g'(1) < \tilde g'(1)$, then $g(r) < \tilde g(r)$ for any $r\in (1,\infty)$. 
\end{cor}
\begin{proof}
Define $f:=g - \tilde g$ and observe that $f$ satisfies 
\begin{equation}
\frac{d}{dr} \left[ r(r-\frac{2}{n}) f'(r)  \right] = h_1(r)f(r)
\end{equation}

\noindent Observe also that $f(1) = 0$ and $f'(1) <0$. This in particular implies that there exists an $r_*>1$ close enough to $1$ such that $f(r)$ and $f'(r)$ are negative for any $r \in (1,r_*)$. By invoking lemma \eqref{tech}, we conclude that $f(r)<0$ for any $r\in (1,\infty)$ as needed. 
\end{proof}

\vv


\vv

We now return to our goal of showing that $a_{\ell}$ blows up at $\infty$.

\vv

We first decompose $a_{\ell} = f_{\ell} + g_{\ell}$ where $f_{\ell}$ solves

\begin{equation} \label{flode}
\begin{cases}
\frac{d}{dr} \left[ r(r-\frac{2}{n}) f_{\ell}'(r)  \right] = \left( \frac{4}{n^2r(r-\frac{2}{n})}+\ell(\ell+1)\right) f_{\ell}(r) - \frac{4 \alpha_{\ell}}{n^2 r(r-2/n)}  \\
f_{\ell}'(1) = 0 \\
f_{\ell}(1) = \alpha_{\ell}
\end{cases}
\end{equation}

and $g_{\ell}$ solves 

\begin{equation} 
\begin{cases}
\frac{d}{dr} \left[ r(r-\frac{2}{n}) g_{\ell}'(r)  \right] = \left( \frac{4}{n^2r(r-\frac{2}{n})}+\ell(\ell+1)\right) g_{\ell}(r) + \frac{1+(4/n^2)\alpha_{\ell}}{r(r-2/n)}  \\
g_{\ell}'(1) = \beta_{\ell} \\
g_{\ell}(0) = 0
\end{cases}
\end{equation}

\noindent By letting $r=1$ in the equation for $f_{\ell}$, we observe that $r(r-2/n)f_{\ell}'(r)$ is decreasing near $r=1$; since $f'_{\ell}(1) = 0$, it follows that $f_{\ell}$ and $f_{\ell}'$ are negative near $r=1$. In particular, letting $\tilde f_{\ell}(r) := f_{\ell}(r) - \alpha_{\ell}$, we also have that $\tilde f_{\ell}$ and $\tilde f'_{\ell}$ are negative near $r=1$. Using the system in \eqref{flode}, we deduce that $\tilde f_{\ell}$ satisfies the following: 
\begin{equation}
\begin{cases}
\frac{d}{dr} \left[ r(r-\frac{2}{n}) \tilde f_{\ell}'(r)  \right] = \left( \frac{4}{n^2r(r-\frac{2}{n})}+\ell(\ell+1)\right) \tilde f_{\ell}(r) + \alpha_{\ell} \ell(\ell+1)  \\
\tilde f_{\ell}'(1) = 0 \\
\tilde f_{\ell}(1) = 0
\end{cases}
\end{equation}

 \noindent We invoke lemma \eqref{tech} on $\tilde f_{\ell}$ to deduce that $\tilde f_{\ell}$ and $\tilde f'_{\ell}$ are negative on $(1,\infty)$.
\noindent This in particular implies that $f_{\ell}(r)$ and $ f_{\ell}'(r)$ are negative on $(1,\infty)$. We again invoke lemma \eqref{tech} on $f_{\ell}$ to conclude that $\lim_{r\to \infty} f_{\ell}(r) = - \infty$. 

\vv

\noindent It suffices to show that $g_{\ell}(r)<0$ for all $r\in [1,\infty)$. We first observe from the definition of $\beta_{\ell}$ and $\alpha_{\ell}$ that 
\begin{align}
 \frac{\beta_{\ell}}{1+(4/n^2)\alpha_{\ell}} &=-\frac{n^2}{2(n-2)} \left( 1+ \frac{2(n-2)}{n(\ell(\ell+1)-2) +2} \right)\\
 &< -\frac{n^2}{2(n-2)} \label{beta1}
 \end{align}
 for all $n>2$ and $\ell \in \N$. In particular, we have that $1+(4/n^2)\alpha_{\ell} >0$ for every $ n>2$ and $\ell \in \N$.

\vv


Let $\tilde g_{\ell}$ be the function on $[1,\infty)$ solving 

\begin{equation} \label{g1}
\begin{cases}
\frac{d}{dr} \left[ r(r-\frac{2}{n}) \tilde g_{\ell}'(r)  \right] = \left( \frac{4}{n^2r(r-\frac{2}{n})}+\ell(\ell+1)\right) \tilde g_{\ell}(r) + \frac{1}{r(r-2/n)}  \\
\tilde g_{\ell}'(1) = -\frac{n^2}{2(n-2)} \\
\tilde g_{\ell}(1) = 0
\end{cases}
\end{equation}

In light of corollary \eqref{cortech} along with equation \eqref{beta1}, it follows that the negativity of $\tilde g_{\ell}(r)$ on $(1,\infty)$ implies the negativity of $\frac{g_{\ell}(r)}{1+(4/n^2)\alpha_{\ell}}$ on $(1,\infty)$, which in turn implies that the negativity of $g_{\ell}(r)$ on $(1,\infty)$. It then suffices to show that $\tilde g_{\ell}(r)<0$ on $(1,\infty)$. 

\vv

We will prove that $\tilde g_{\ell}(r)<0$ on $(1,\infty)$ by induction on $\ell$. We first find the solution for $\ell = 0$ to be:

\begin{equation}
\tilde g_0(r) = -\frac{n(r-1)}{2(r-2/n)}
\end{equation}

which is negative everywhere.

\vv

Now suppose that  $\tilde g_{\ell}(r)<0$ on $(1,\infty)$ for some nonnegative integer $\ell$. Define $h_{\ell} := \tilde g_{\ell+1} - \tilde g_{\ell}$, which will solve: 

\begin{equation} 
\begin{cases}
\frac{d}{dr} \left[ r(r-\frac{2}{n}) h_{\ell}'(r)  \right] = \left( \frac{4}{n^2r(r-\frac{2}{n})}+(\ell+1)(\ell+2)\right) h_{\ell}(r) + 2(\ell+1) \tilde g_{\ell}(r)\\
h_{\ell}'(1) = 0 \\
h_{\ell}(1) = 0
\end{cases}
\end{equation}

We directly compute $h_{\ell}''(1) = \frac{n}{n-2} 2(\ell+1)\tilde g_{\ell}(1)= 0$ and $h_{\ell}'''(1) = \frac{n}{n-2} 2(\ell+1)\tilde g_{\ell}'(1)<0$. This then implies that $h_{\ell}$ and $h_{\ell}'$ are negative near $r=1$. Using the fact that $\tilde g_{\ell}(r)<0$ on $(1,\infty)$ and invoking lemma \eqref{tech}, it follows that $h_{\ell}(r)<0$ on $(1,\infty)$, which in turn implies that $\tilde g_{\ell+1}(r)<0$ on $(1,\infty)$ as needed. 

\vv

\vv

We have finally shown that $a_{\ell}$ blows up at infinity for every $\ell \in \N$ contradicting that $\norm{a_{\ell}}_{1,\delta}<\infty$. We conclude that the assumption that $C_{\ell} \neq 0$ was false and hence $a_{\ell} = 0$ for every $\ell \in \N$ and, in turn, $\tilde u = 0$.

\end{proof}

This concludes Step 1, and we move on to Step 2.

\vvv

{\Large \bf Step 2: Solving for $\tilde g$ and $\tilde \omega$} 

\vv

Let $\tilde u \in  {\mathcal{A}^{(2,k+1)}_{\delta}}$ be the function satisfying equations \eqref{eq1} and \eqref{eq2} with $\psi$ and $\Gamma$ given by equations \eqref{psi} and \eqref{Gamma}. We wish to show that there exists a unique $\tilde g \in T_{g_{sc}}\mathcal{M}^k_{\delta}$ and a conformal Killing field $\tilde \omega$ satisfying equations \eqref{1st} to \eqref{3rd} and \eqref{5th} to \eqref{8th}. 

\vv

 Equation \eqref{8th} determines uniquely the initial data for $\widetilde{trK}$, which is given by

\begin{equation} \label{initialtrK}
\left. \widetilde{trK}\right|_{\partial M} = H +2\partial_r \tilde u |_{\partial M} -\frac{2}{nm_0} \tilde u|_{\partial M} 
\end{equation}
and is living in $H^k(\partial M)$. We rewrite the ODE \eqref{2nd} obeyed by $\widetilde{trK}$ here for convenience. 

 \begin{equation*} 
\partial_r \widetilde{trK} + trK_{sc} \widetilde{trK}  + 4(\partial_r u_{sc})( \partial_r \tilde u) = B, \quad \text{on $M$} 
\end{equation*}
which, together with the initial condition in \eqref{initialtrK}, determines uniquely $\widetilde{trK}$ on $M$. We explicitly solve for $\widetilde{trK}$ on $M$ to get: 

\begin{equation}
\widetilde{trK}(r,p)= \frac{n(n-2)m_0^2}{r(r-2m_0)}\left. \widetilde{trK} \right|_{\partial M}(p) - \frac{4m_0}{r(r-2m_0)} \Big(\tilde u(r,p) - \tilde u|_{\partial M}(p)\Big) + \frac{1}{r(r-2m_0)} \int_{nm_0}^r s(s-2m_0) B(s,p)  ds
\end{equation}
for $r\in [nm_0,\infty)$ and $p \in S^2$. 

\vv

We observe that 

\begin{itemize} 
\item $(\partial_r u_{sc})( \partial_r \tilde u) \in  \LtoH{\delta-2}{k}$
\item $B \in \LtoH{\delta-2}{k}$. 
\end{itemize}

and so $\widetilde{trK}$ lies in $\HtoH{1}{\delta-1}{k}$.


\vvv

We turn our attention to $\widetilde{\hat K}$. We first recall the well known fact regarding the divergence operator on symmetric traceless tensors on $S^2$ (see \cite{christodoulou}). 

\begin{prop}
Let $k \geq2$. Let $\gamma$ be a smooth metric on $S^2$ with positive curvature. Denote by $\mathcal{D}^k(S^2)$ the space of traceless symmetric $(0,2)$ tensors on $(S^2, \gamma)$ with components in $H^k(S^2)$. Let ${\Omega^{\perp}}^{k-1}(\partial M)$ be the space of vector fields on $S^2$ with components in $H^{k-1}(M)$ that are $L^2$ orthogonal to conformal Killing vector fields on $(S^2,\gamma)$. Then the divergence operator $\cancel{div}_{\gamma}$ is an isomorphism from $\mathcal{D}^k(S^2)$ to ${\Omega^{\perp}}^{k-1}(S^2)$
\end{prop} 

The above proposition along with equation \eqref{6th} imply that $\frac{2}{n(n-2)m_0} \slashed d\tilde u + \tilde \omega-F$ must be orthogonal to conformal Killing vector fields on $(\partial M, \gamma_{sc})$. Since $\tilde \omega$ is conformal Killing on $(\partial M,\gamma_{sc})$, this requirement determines $\tilde \omega$ uniquely. Indeed, if $Y_1,...,Y_6$ is an $L^2$ orthonormal basis of conformal Killing vector fields, then $\tilde \omega$ must be 

\begin{equation}
\tilde \omega := \sum_{i=1}^6 \left( \int_{S^2} Y_i \cdot (F-\frac{2}{n(n-2)m_0} \slashed d\tilde u ) d\sigma_{\gamma_{sc}}\right) Y_i
\end{equation}

The above proposition together with \eqref{6th} determine uniquely the initial condition for $\widetilde{\hat K}$ to be 

\begin{equation} \label{initialK}
\left. \widetilde{\hat K}\right|_{\partial M} = \cancel{div}_{\gamma_{sc}}^{-1} \left( \frac{2}{n(n-2)m_0} \slashed d\tilde u + \tilde \omega-F\right) 
\end{equation}
living in $\mathcal{H}^k(\partial M)$. We rewrite the ODE \eqref{3rd} obeyed by $\widetilde{\hat K}$ here for convenience. 

 \begin{equation*} 
\mathcal{L}_{\dd{r}} {\widetilde{\hat K}}  = C, \quad \text{on $M$}
\end{equation*}
which, together with the initial condition in \eqref{initialK}, determines $\hat K$ uniquely on $M$. Fixing fermi coordinates, $(r,\theta^1, \theta^2)$, we explicitly solve for $\hat K$ on $M$ to get: 

\begin{equation}
\hat K_{ij}(r,p) = \left. \widetilde{\hat K}_{ij}\right|_{\partial M}  +  \int_{nm_0}^r C_{ij}(s,p) 
\end{equation}
for $i,j=1,2$, $r \in [nm_0,\infty)$ and $p \in S^2$.

Since $C \in \LtoHH{\delta-2}{k}$, it follows that $\hat K$ lies in $\HtoHH{1}{\delta-1}{k}$. 
\vv

Equation \eqref{7th} determines uniquely initial data for $\tilde g$ given by 

\begin{equation} \label{initialg}
\left.  \tilde g \right|_{\partial M} = \frac{n-2}{n} \left(  2 n^2m_0^2 \, \tilde u \,  \gamma_{\mathbb{S}^2}+ G \right)
\end{equation}
living in $\mathcal{H}^k(\partial M)$. The evolution of $\tilde g$ is determined by $\widetilde{trK}$ and $\widetilde{\hat K}$ in the following equation:

 \begin{equation} \label{evolg}
 \mathcal{L}_{\dd{r}} \tilde g = 2 \widetilde{\hat K} + \widetilde{trK} \, {g_{sc}} + trK_{sc} \, \tilde g
 \end{equation}
Equations \eqref{evolg} and \eqref{initialg} determine uniquely $\tilde g$ to be, in fermi coordinates $(r,\theta^1, \theta^2)$, 

\begin{equation}
\tilde g_{ij} = r(r-2m_0) \int_{nm_0}^r \frac{1}{s(s-2m_0)} \left( 2 \widetilde{\hat K}_{ij} + \widetilde{trK} {g_{sc}}_{ij} \right) ds + \frac{r(r-2m_0)}{nm_0^2} G_{ij}
\end{equation}

In light of the fact that $\widetilde{trK} \in \HtoH{1}{\delta-1}{k}$ and $\widetilde{\hat K} \in \HtoHH{1}{\delta-1}{k}$, it follows $\tilde g$ is of the form $\tilde g = r^2(G+h(r))$ where $G \in \mathcal{H}^k(S^2)$ and $h\in \HtoHH{2}{\delta}{k}$; this implies that $\tilde g \in T_{g_{sc}}\mathcal{M}^k_{\delta}$ as needed.

\vv

We have then shown that there exists a unique $\tilde g$ and $\tilde \omega$ satisfying equations \eqref{2nd}, \eqref{3rd} and \eqref{5th} to \eqref{8th}. It follows by lemma \eqref{decouple} that equation \eqref{1st} is satisfied as well. This concludes Step 2, and we move on to Step 3.

\vvv

{\Large \bf Step 3: Solving for $\tilde X$} 

\vv

We wish to show that there exists a unique $\tilde Y \in \widehat{\mathcal{X}}^{2}_{\delta} (M)$ satisfying 

\begin{equation}
\Delta_{g_{sc}, conf} \tilde Y = D
\end{equation}
where $D \in \mathcal{X}^{0}_{\delta-2}(M)$. Similar results have been shown in \cite{maxwell} for the above equation with trivial Dirichlet and Neumann conditions. In our case, vector fields $Y$ in $\widehat{\mathcal{X}}^2(M)$ satisfy the following mixed boundary conditions 
\begin{equation}
g(Y, \dd{r}) = 0, \quad \left( \mathcal{L}_{\dd{r}} Y \right)^T = 0
\end{equation}
 The isomorphism of the operator $\Delta_{g_{sc}, conf}$ in our space follows by minor modifications of the proof in \cite{maxwell}. We add the proof here for the sake of completeness.

\begin{prop}
Let $\delta \in (-1,-\frac{1}{2}]$. The operator $\Delta_{g_{sc},conf}$ is an isomorphism from $\widehat{\mathcal{X}}^{2}_{\delta}(M)$ to $\mathcal{X}^{0}_{\delta-2}(M)$. 
\end{prop}
\begin{proof}

\vv
Recall that if $X \in \widehat{\mathcal{X}}^{2}_{\delta} (M)$, then 
\begin{equation} \label{recallX}
g(X, \dd{r}) = 0, \quad \left( \mathcal{L}_{\dd{r}} X \right)^T = 0
\end{equation} 

on $\partial M$. In particular, we have that 
\begin{align}
\conflie{g_{sc}}X(X,\dd{r}) &= \mathcal{L}_X g (X, \dd{r}) - \frac{2}{3} \text{div} X \, g(X, \dd{r})\\
& = g(\mathcal{L}_{\dd{r}} X, X)\\
&=0
\end{align}
on $\partial M$. 

\vv
Let $X \in \widehat{\mathcal{X}}^{2}_{\delta}(M)$ satisfy $\Delta_{g_{sc},conf} X = 0$. Given $R\geq nm_0$, let $\phi_R$ be a cutoff function on $[nm_0,\infty)$ satisfying $\phi_R(x) = 1$ for $x\leq R$, $\phi_R(x) = 0$ for $x\geq R+1$, and $-2\leq \phi_R'(x) \leq 0$ for any $x\in [nm_0,\infty)$. 

We integrate by parts to get 

\begin{align}
0 &= \int_{M}  \phi_R X^{\mu} \Delta_{g_{sc},conf}  X_{\mu} dV_{g_{sc}}\\
&= -\frac{1}{2} \int_{M}  \left( \conflie{g_{sc}} \phi_R X \right)  \cdot \left( \conflie{g_{sc}} X \right) dV_{g_{sc}}  - \int_{\partial M} \conflie{g_{sc}}  X (X, \dd{r}) d\sigma_{g_{sc}(nm_0)} \\ 
&=-\frac{1}{2} \int_{M} \phi_R \left| \conflie{g_{sc}} X \right|^2 dV_{g_{sc}} -\frac{1}{2} \int_M 2 \phi_R' \, \conflie{g_{sc}}  X (X, \dd{r}) dV_{sc}
\end{align}

This implies that for any $R\geq nm_0$, 

\begin{align}
\int_{B_R \setminus B_{nm_0}} \left| \conflie{g_{sc}} X \right|^2 dV_{g_{sc}} & \leq \int_{M} \phi_R \left| \conflie{g_{sc}} X \right|^2 dV_{g_{sc}}\\
&= -\int_M 2 \phi_R' \, \conflie{g_{sc}}  X (X, \dd{r}) dV_{sc} \\
&\leq 4 \int_{B_{R+1} \setminus B_{R}}| \conflie{g_{sc}}  X (X, \dd{r})| dV_{sc}\\
&\lesssim \int_{B_{R+1} \setminus B_{R}} |\nabla X|^2 dV_{sc}
\end{align}

Since $\delta \in (-1,-\frac{1}{2})$, we have that $|\nabla X|^2 \leq |\nabla X|^2 r^{-2\delta-1}$ and so $|\nabla X|^2$ is integrable on $M$ and $\int_M |\nabla X|^2 dV_{sc} \leq \norm{\nabla X}_{0,\delta-1}^2$. We can then take the limit as $R$ goes to infinity in the above equations to deduce that 

\begin{equation} 
\int_{M} \left| \conflie{g_{sc}} X \right|^2 dV_{g_{sc}} = 0
\end{equation}
implying that $X$ is conformal Killing on $(M,g_{sc})$. Since the equation $\Delta_{g_{sc}, conf} X = 0$ is an elliptic PDE with smooth coefficients, elliptic regularity shows that $X$ is $C^{\infty}$. We can then invoke lemma \ref{confkill=0} to conclude that $X = 0$, which shows that the kernel of $\Delta_{g_{sc}, conf}$ is trivial.

\vvv

Now we show that the kernel of the adjoint is also trivial. It will then follow that $\Delta_{g_{sc},conf}$ is an isomorphism. Recall that 

 \[\Delta_{g_{sc},conf}^* : \left( \mathcal{X}^{0}_{\delta-2} (M)\right)^* \to \left( \mathcal{X}^{2}_{\delta} (M) \right)^*  \]
 
For any number $\tau \in \R$, Riesz's representation theorem allows us to identify $\left( \mathcal{X}^{0}_{\tau} (M) \right)^*$ with $ \mathcal{X}^{0}_{\tau} (M)$ via the map : 
\[J :  \mathcal{X}^{0}_{\tau} (M) \to \left( \mathcal{X}^{0}_{\tau} (M) \right)^*\]
\[ \text{For $Y \in \mathcal{X}^{0}_{\tau} (M)$}, \,\, J(Y): X \in \mathcal{X}^{0}_{\tau} (M) \mapsto \int_M X\cdot Y r^{-2\tau-3} dV_{g_{sc}} \]
where $\cdot$ is with respect to $g_{sc}$. For simplicity of the notation, we will denote both $Y$ and $J(Y)$ by $Y$; it will be clear from context which one we are referring to. 
 
 \vv
 
 We then have that $Y \in \left( \mathcal{X}^{0}_{\delta-2} (M)\right)^*$ is in the kernel of $\Delta_{g_{sc},conf}^*$ if and only if 
 
 \begin{equation}
 \int_{M} Y\cdot  \Delta_{g_{sc},conf} X \,\,r^{-2(\delta-2)-3} dV_{g_{sc}} = 0
 \end{equation}
for every $X \in \widehat{\mathcal{X}}^{2}_{\delta}(M)$, which is equivalent to the above equation holding for every smooth compactly supported vector field $X$ in $\widehat{\mathcal{X}}^{2}_{\delta}(M)$ by a density argument. 

\vv

It follows from elliptic regularity that $Y \in \mathcal{X}^{2}_{\delta-2}(M)$. In fact, $Y$ will be smooth since the metric $g_{sc}$ is smooth. Given an arbitrary smooth compactly supported vector field $X$ in $\widehat{\mathcal{X}}^{2}_{\delta}(M)$, we can then integrate by parts to get
\vv

 \begin{align} 
 \int_M Y\cdot  \Delta_{g_{sc}, conf} X \,\,  r^{-2(\delta-2)-3} dV_{sc} &= \int_M X\cdot  \Delta_{g_{sc}, conf} \bar Y dV_{sc} \\
 & \qquad + \int_{\partial M}  \left( \mathcal{L}_{g_{sc}, conf} \bar Y (X,\dd{r}) - \mathcal{L}_{g_{sc}, conf} X (\bar Y,\dd{r}) \right) d\sigma\\
 &= \int_M X\cdot  \Delta_{g_{sc}, conf} \bar Y dV_{sc} \\
 & \qquad  \int_{\partial M} \left[  X^{i} ( \mathcal{L}_{\dd{r}} \bar Y_{i} - \partial_i \bar Y_0 ) + \frac{5}{3} \partial_r X_0 \bar Y_0 + \frac{2}{3} \cancel{div} (X^T) \bar Y_0 \right] d\sigma  
 \end{align}
 where $\bar Y := r^{-2(\delta-2) -3} Y$, $X_0 := g_{sc}(\dd{r}, X)$ and $\bar Y_0 := g_{sc}(\dd{r}, \bar Y)$. Note that the boundary terms vanish at infinity since $X$ is compactly supported. Since $X$ was an arbitrary vector field in a dense subset of $\widehat{\mathcal{X}}^{2}_{\delta}(M)$, it follows that $\bar Y$ satisfies

\begin{equation}
\Delta_{g_{sc}, conf} \bar Y = 0, \quad g(\dd{r}, \bar Y) = 0, \quad \left( \mathcal{L}_{\dd{r}} \bar Y \right)^T = 0
\end{equation}

Considering that $Y \in \mathcal{X}^{0}_{\delta-2} (M) $, it follows that $\bar Y \in \widehat{\mathcal{X}}^{2}_{-\delta-1} (M) $. Since $\Delta_{{g_{sc}}, conf} \bar Y = 0$ and $-1< -\delta-1< 0$, we have that $Y \in \widehat{\mathcal{X}}^{2}_{\tau} (M)$ for any $\tau\in (-1,0)$ (see \cite{choquet} and \cite{Bartnik1}). We can then apply the same integration-by-parts argument carried out earlier to conclude that $\bar Y$ is conformal Killing on $(M,g_{sc})$ and hence, by lemma \eqref{confkill=0}, vanishes as needed. 
 

\end{proof}

\appendix 
\section{Appendix}

\subsection{A PDE of Finite Type for Conformal Killing Vector Fields}\label{appendix CK}
In this section, we will prove an identity satisfied by conformal Killing vector fields that is used in the proof of lemma \ref{confkill=0}. More specifically, we will prove that any conformal Killing vector field $Z$ on an arbitrary $n$-dimensional Riemannian manifold $(M,g)$ satisfies the following PDE of finite type: 

\begin{equation} \label{eqforZappendix}
\nabla^3 Z = A\cdot \nabla Z + B \cdot Z
\end{equation}

where $A$ and $B$ are linear expressions in $Riem$ and $\nabla Riem$.

\vv

Let $Z$ be a conformal Killing field and let $\psi := \frac{2}{n} \text{div} Z$. The conformal Killing equation is

\begin{equation}
\nabla_i Z_j + \nabla_j Z_i =  \psi g_{ij}
\end{equation}

Eisenhart in \cite{eisenhart} (see pages 231-232 in \cite{eisenhart}) proves the following identities:

\begin{equation} \label{eq1Z}
\nabla_k\nabla_jZ_i = -Z_m R^m_{kij} + \frac{1}{2} (g_{ij} \nabla_k \psi + g_{ik} \nabla_j \psi - g_{jk} \nabla_i \psi )
\end{equation} 

\begin{equation} \label{eq2Z}
g^{il} Z_m \nabla_l R^m_{kij} - Z_m \nabla_k R^m_j - \nabla_k Z_mR^m_j - \nabla_j Z_m R^m_k + \frac{n-2}{2} \nabla_k \nabla_j \psi + \frac{1}{2} g_{jk} \Delta \psi =0
\end{equation}

where $R$ denotes $Riem$ or $Ric$ depending on the number of indices. Taking the trace of equation \eqref{eq2Z}, we get 
\begin{equation} \label{eq3Z}
\Delta \psi = \frac{2}{n-1} \left( Z_m \nabla_i R^{mi} + \nabla_i Z_m R^{mi} \right)
\end{equation}

Using equation \eqref{eq3Z} to eliminate $\Delta \psi$ in equation \eqref{eq2Z}, we get the following expression of $\nabla_k \nabla_j \psi$: 
\begin{equation} \label{eq4Z}
\nabla_k \nabla_j \psi = \frac{2}{n-2} \bigg( -g^{il} Z_m \nabla_l R^m_{kij} + Z_m \nabla_k R^m_j + \nabla_k Z_mR^m_j + \nabla_j Z_m R^m_k -  \frac{1}{n-1} g_{jk} \left( Z_m \nabla_i R^{mi} + \nabla_i Z_m R^{mi} \right)\bigg)
\end{equation}

Taking a derivative of equation \eqref{eq1Z} and using equation \eqref{eq4Z}, we get the desired result. 

\subsection{A Hardy-type Inequality} \label{hardy appendix}

In this section, we will prove a Hardy-type inequality that is used in the proof of lemma \ref{confkill=0}, namely equation \eqref{hardy}. 

\vv

Let $g \in \mathcal{M}^k_{\delta}(M)$ be a Riemannian metric on $M = [nm_0,\infty) \times S^2$ of the form 

\[g = dr^2+ r^2(\gamma_{\infty} + h(r))\] 
where $\gamma_{\infty}$ is a metric on $S^2$ and $h \in \HtoHH{2}{\delta}{k}$. We will prove the following. 

\begin{prop} \label{hardy prop}  There exists an $R_0>nm_0$ depending only on $g$ such that for $R\geq R_0$ and $\tau>0$, the inequality
\begin{equation}
\int_{[R,\infty)\times S^2}  r^{\tau-2} |T|^2   \, dV_g \leq \frac{4}{\tau^2} \int_{[R,\infty)\times S^2} r^{\tau} |\nabla T|^2 dV_g
\end{equation}
holds for all tensor fields $T \in C^1_c(M)$. 
\end{prop}
\begin{proof}
The main tool we will use is a general $L^p$ Hardy inequality in Riemannian manifolds developed by D'Ambrosio and Dipierro in \cite{hardy}. We present the relevant version of it below. 

\begin{thm} \label{hardy ref}
Let $\rho \in C^2(M)$ such that $\Delta_g \rho \geq 0$ such that $\frac{|\nabla \rho|^2}{\Delta \rho} \in L^1_{loc}(M)$. Then for any $R> nm_0$, the inequality 
\begin{equation}
\int_{[R,\infty)\times S^2} |u|^2 \Delta \rho  \, dV_g \leq 4 \int_{[R,\infty)\times S^2} \frac{|\nabla \rho|^2}{\Delta \rho} |\nabla u|^2 dV_g
\end{equation}
holds for all $u \in C^1_c (M)$. 
\end{thm}

Letting $\rho = r^{\tau}$ for $\tau>0$, we compute 

\begin{equation}
\Delta \rho = \tau r^{\tau-2} \left[ \tau+1+ r \left(trK - \frac{2}{r}\right) \right], \quad |\nabla \rho|^2 = \tau^2 r^{2\tau - 2} 
\end{equation}

In light of the Sobolev embeddings in proposition \ref{prop-spaces}, we have that 
\begin{equation}
|trK -\frac{2}{r}| = o(r^{-1+\delta})
\end{equation}

So we can choose $R_0 > nm_0$ depending only on $g$ such that for any $R\geq R_0$,

\begin{equation}
\left| r \left(trK - \frac{2}{r}\right)\right| \leq 1
\end{equation} 
This, in turn, implies that for any $\tau>0$ and any $R\geq R_0$, 
\begin{equation} \Delta \rho \geq \tau^2 r^{\tau-2} \geq0\end{equation}
on $[R,\infty)\times S^2$ . 

\vv
We can then invoke theorem \ref{hardy ref} directly to deduce that the inequality 
\begin{equation} \label{hardy 1}
\int_{[R,\infty)\times S^2}  r^{\tau-2} |u|^2   \, dV_g \leq \frac{4}{\tau^2} \int_{[R,\infty)\times S^2} r^{\tau} |\nabla u|^2 dV_g
\end{equation}
hold for any $R\geq R_0$ and $u \in C^1_c(M)$.

The same inequality holds with $u$ replaced with a tensor field $T$. To see this, we first compute that for any tensor field $T$ on $M$, 
\begin{equation} \label{hardy 2}
|\nabla |T| |^2 \leq |\nabla T|^2
\end{equation}
Using the above and letting $u =|T|$  in equation \eqref{hardy 1}, we immediately deduce that the inequality 
\begin{equation} \label{hardy 2}
\int_{[R,\infty)\times S^2}  r^{\tau-2} |T|^2   \, dV_g \leq \frac{4}{\tau^2} \int_{[R,\infty)\times S^2} r^{\tau} |\nabla T|^2 dV_g
\end{equation}
holds for any $R\geq R_0$ and tensor field $T \in C^1_c(M)$.
\end{proof}
\vv

\begin{cor} \label{main hardy prop}
Let $\delta \in (-1,-\frac{1}{2})$. There exists an $R_0>nm_0$ depending only on $g$ and a positive constant $C$ depending only on $\delta$ such that for any $R\geq R_0$ and any vector field $Z \in \mathcal{X}^3_{\delta}(M)$, 
\begin{equation} \label{hardy in appendix}
\int_{[R,\infty)\times S^2} r^{-2\delta-3} (|Z|^2+ r^2|\nabla Z|^2) dV_g \leq C \int_{[R,\infty)\times S^2} r^{-2(\delta-3)-3} |\nabla^3 Z|^2 dV_g 
\end{equation}
\end{cor}
\begin{proof}
By repeatedly applying proposition \ref{hardy prop} for $\tau = -2\delta-1,-2\delta+1, -2\delta+3$ and $T = Z, \nabla Z, \nabla^2 Z$, we deduce that there exists a positive constant $C = C(\delta)$ such that 
\begin{equation}
\int_{[R,\infty)\times S^2} r^{-2\delta-3} (|Z|^2+ r^2|\nabla Z|^2) dV_g \leq C \int_{[R,\infty)\times S^2} r^{-2(\delta-3)-3} |\nabla^3 Z|^2 dV_g 
\end{equation}
for any $R\geq R_0$ and any vector field $Z \in C^3_c(M)$. The above inequality can be rewritten in terms of the norm on $\mathcal{X}^k_{\delta}(M)$ (see definition \ref{def of Xkdelta}) as follows: 

\begin{equation}
\norm{Z}_{1,\delta}^2 \leq C \norm{\nabla^3 Z}_{0,\delta-3}^2
\end{equation}

The desired inequality then follows from the density of $C^3_c(M)$ in $\mathcal{X}^3_{\delta}(M)$. 
\end{proof}

\subsection{The Legendre functions $P_{\ell}$ and $Q_{\ell}$} \label{legendre appendix}

In this section, we will prove properties of the Legendre functions of the first and second kind, $P_{\ell}$ and $Q_{\ell}$, that are used in the proof of lemma \ref{main-est}. 

\vv
Fix $R>1$ (which is equal to $n-1$ in section $3$). For a positive integer $\ell$, as described by Olver in \cite{olver}, the Legendre functions, $P_{\ell}$ and $Q_{\ell}$, are linearly independent solutions to the ODE

\begin{equation}
(z^2-1) h''(z) + 2z h'(z)-\ell(\ell+1)h(z) = 0, \quad z\in [R,\infty)
\end{equation}

with the following asymptotics as $z\to \infty$,
\begin{equation}
P_{\ell}(z) = O(z^{\ell}), \quad Q_{\ell}(z) = O(z^{-\ell-1})
\end{equation}

We normalize $P_{\ell}$ and $Q_{\ell}$ so that 

\begin{equation}
\lim_{z\to \infty} z^{-\ell} P_{\ell}(z) = 1, \quad \lim_{z\to \infty} z^{\ell+1} Q_{\ell}(z) = 1
\end{equation}
which is different than Olver's. Letting ${\bf P}_{\ell}$ and ${\bf Q}_{\ell}$ be the Legendre functions as defined by Olver, the relation between ours and his can immediately be obtained is as follows (see \cite{olver} chapter 5 section 12 and 13 ):

\begin{equation}
P_{\ell}(z) = \frac{\sqrt{\pi} \Gamma(\ell+1)}{2^{\ell}\Gamma(\ell+\frac{1}{2})} {\bf P}_{\ell}(z), \quad Q_{\ell}(z) = \frac{2^{\ell+1}\Gamma(\ell+\frac{3}{2})}{\sqrt{\pi} \Gamma(\ell+1)} {\bf Q}_{\ell}(z)
\end{equation}

In the following proposition, we will apply the method of Frobenius to obtain the expansion of $P_{\ell}$ and $Q_{\ell}$ in terms of powers of $z$. 

\begin{prop} \label{frobenius}
$P_{\ell}$ and $Q_{\ell}$ admit an expansion of the following form. For $z>1$, 
\begin{equation}
P_{\ell}(z) = \sum_{k=0}^\ell a_k z^{\ell-k}, \quad Q_{\ell}(z) = \sum_{k=0}^{\infty} b_k z^{-\ell-1-k}
\end{equation}

where the coefficients $a_k$ and $b_k$ are defined recursively as follows:

\[ a_0 = b_0 = 1, \quad a_1 = b_1 = 0\]
\[\text{for $k\geq 2$}, \quad a_{k} = \frac{(\ell-k+2)(\ell-k+1)}{k^2-k(2\ell+1)}a_{k-2}, \quad b_{k} = \frac{(\ell+k-1)(\ell+k)}{k(2\ell+k+1)}b_{k-2}  \]
\end{prop}

The expansion of $P_{\ell}$ and $Q_{\ell}$ as described above agree with \cite{watson} pg 302 and 320. 

\begin{prop} \label{prop-estPQ}
There exists a constant $C = C(R)$ such that for any $\ell\geq 1$ and $z\in [R,\infty)$, the following holds 

\begin{equation} \label{estPQa}
z^{-\ell}|P_{\ell}(z)| \leq  C \left(\frac{2z}{z+\sqrt{z^2-1}} \right)^{-\ell},\qquad z^{\ell+1} |Q_{\ell}(z)| \leq C\left(\frac{2z}{z+\sqrt{z^2-1}} \right)^{\ell}
\end{equation}
\begin{equation}\label{estPQ'a}
z^{-(\ell-1)}|P'_{\ell}(z)| \leq C \ell \left(\frac{2z}{z+\sqrt{z^2-1}} \right)^{-\ell}, \qquad z^{\ell+2} |Q'_{\ell}(z)| \leq C \ell \left(\frac{2z}{z+\sqrt{z^2-1}} \right)^{\ell}
\end{equation}
\end{prop}
\begin{proof}

Olver, in \cite{olver} chapter 12 section 12, has established an asymptotic expansion of ${\bf P}_{\ell}$ and ${\bf Q}_{\ell}$ for large degree $\ell$ that is uniformly valid for $z \in (1,\infty)$. Shivakumar and Wong, in \cite{wong}, proved an equivalent expansion of ${\bf P}_{\ell}$ that is more computable; letting $u=\ell+1/2$, he has shown that for $\xi>0$,

\begin{equation}
{\bf P}_{\ell} (\cosh\xi) = \left(\frac{\xi}{\sinh\xi}\right)^{1/2} \bigg( I_0(u\xi) + \epsilon(u,\xi)   \bigg)
\end{equation}
where 
\begin{equation}
|\epsilon(u, \xi) | \leq \frac{\Gamma(3/2)}{2\Gamma(1/2)} \frac{2\xi}{1+\xi} \frac{I_1(u\xi)}{u}
\end{equation}
and $I_0$, $I_1$ are the modified Bessel functions (see \cite{olver} chapter 2 section 10). 

Similarly, Frenzen, in \cite{frenzen}, proved an equivalent expansion of ${\bf Q}_{\ell}$ that is more computable than Olver's; he has shown that for $\xi >0$, 

\begin{equation} \label{Q-expansion}
{\bf Q}_{\ell} (\cosh\xi) = \left(\frac{\xi}{\sinh\xi}\right)^{1/2} \bigg( K_0(u\xi) + \eta(u,\xi)   \bigg)
\end{equation}
where 
\begin{equation}
|\eta(u, \xi) | \leq \frac{\xi}{2(2+\xi)} \frac{K_1(u\xi)}{u}
\end{equation}
and $K_0$, $K_1$ are the modifed Bessel functions (see \cite{olver} chapter 7 section 8). 


\vv

Letting $\xi_{0}$ be the positive number in which $\cosh \xi_0 = R$, Olver's asymptotics of the modified Bessel function in \cite{olver} chapter 12, section 1 implies that for all $\xi\geq \xi_0$,
\begin{equation} \label{I,K}
|I_{0}(u \xi)| + |I_{-1}(u\xi)| \leq C \frac{e^{u\xi}}{2\pi \sqrt{u\xi}}, \quad |K_{0}(u\xi)| + |K_1(u\xi)| \leq C \left( \frac{\pi}{2u \xi} \right)^{1/2} e^{-u\xi}
\end{equation}

for some constant $C>0$ depending only on $R$. We can then compute for $\xi \in [\xi_0, \infty)$, 
\begin{align}
(\cosh\xi)^{-\ell} |P_{\ell}(\cosh \xi)| &= \frac{\sqrt{\pi} \Gamma(\ell+1)}{2^{\ell}\Gamma(\ell+\frac{1}{2})} (\cosh\xi)^{-\ell} |{\bf P}_{\ell}(\cosh \xi)|\\
&\leq C\frac{\sqrt{\pi} \Gamma(\ell+1)}{2^{\ell}\Gamma(\ell+\frac{1}{2})} 2^{\ell} (e^{\xi} + e^{-\xi})^{-\ell} \frac{\sqrt{2\xi}}{\sqrt{e^{\xi} - e^{-\xi}}} \frac{e^{(\ell+1/2)\xi}}{2\pi \sqrt{(\ell+1/2)}\sqrt{\xi}}\\
&\leq C \frac{\Gamma(\ell+1)}{\Gamma(\ell+\frac{1}{2})} \frac{1}{\sqrt{\ell}} \left( \frac{e^{\xi}}{e^{\xi} - e^{-\xi}}\right)^{1/2} \left( \frac{e^{\xi}}{e^{\xi} + e^{-\xi}} \right)^{\ell} \\
& \leq C \frac{\Gamma(\ell+1)}{\Gamma(\ell+\frac{1}{2})} \frac{1}{\sqrt{\ell}} \left( \frac{e^{\xi}}{e^{\xi} + e^{-\xi}} \right)^{\ell} 
\end{align}

\begin{align}
(\cosh\xi)^{\ell+1} |Q_{\ell}(\cosh \xi)| &= \frac{2^{\ell+1}\Gamma(\ell+\frac{3}{2})}{\sqrt{\pi} \Gamma(\ell+1)}(\cosh\xi)^{\ell+1} |{\bf Q}_{\ell}(\cosh \xi)|\\
&\leq C\frac{2^{\ell+1}\Gamma(\ell+\frac{3}{2})}{\sqrt{\pi} \Gamma(\ell+1)}  2^{-\ell-1}(e^{\xi} + e^{-\xi})^{\ell+1} \frac{\sqrt{2\xi}}{\sqrt{e^{\xi} - e^{-\xi}}}  \left( \frac{\pi}{2(\ell+1/2) \xi} \right)^{1/2} e^{-(\ell+1/2)\xi}\\
&\leq C \frac{\Gamma(\ell+\frac{3}{2})}{\Gamma(\ell+1)} \frac{1}{\sqrt{\ell}} \left( \frac{e^{\xi}}{e^{\xi} - e^{-\xi}}\right)^{1/2} \left( \frac{e^{\xi} + e^{-\xi}}{e^{\xi}} \right)^{\ell+1} \\
& \leq C \frac{\Gamma(\ell+\frac{3}{2})}{\Gamma(\ell+1)} \frac{1}{\sqrt{\ell}} \left( \frac{e^{\xi} + e^{-\xi}}{e^{\xi}} \right)^{\ell}
\end{align}
where we have allowed the constant $C$ to change from line to line but remains dependent only on $R$ and not on $\xi$ or $\ell$.

\vv

 In light of Stirling's formula (see \cite{olver}, chapter 3, section 8), the Gamma function enjoys the following asymptotics: 

\begin{equation}
\Gamma(x) = e^{-x} x^x (1+ \mathcal{O}(x^{-1})), \quad \text{as $x \to \infty$}
\end{equation}

In particular, we have 
\begin{align}
\frac{\Gamma(\ell+1)}{\Gamma(\ell+\frac{1}{2})} = \left( 1+ \frac{1}{2(\ell+1/2)} \right)^{\ell+1/2} \frac{\sqrt{\ell+1}}{\sqrt{e}} (1+ \mathcal{O}(\ell^{-1}))  
\end{align}
and 
\begin{align}
\frac{\Gamma(\ell+\frac{3}{2})}{\Gamma(\ell+1)} = \left( 1+ \frac{1}{2(\ell+1)} \right)^{\ell+1} \frac{\sqrt{\ell+\frac{3}{2}}}{\sqrt{e}} (1+ \mathcal{O}(\ell^{-1}))  
\end{align}
as $\ell \to \infty$. It then follows that $\frac{\Gamma(\ell+1)}{\Gamma(\ell+\frac{1}{2})} \frac{1}{\sqrt{\ell}}$ and $\frac{\Gamma(\ell+\frac{3}{2})}{\Gamma(\ell+1)} \frac{1}{\sqrt{\ell}}$ are bounded for $\ell \geq 1$, and we finally conclude that there exists a constant $C$ depending only on $R$ such that for any $z\geq R$, 

\begin{equation}
z^{-\ell} P_{\ell}(z) \leq C \left(\frac{2z}{z+\sqrt{z^2-1}} \right)^{-\ell}, \quad z^{\ell+1} Q_{\ell}(z) \leq C\left(\frac{2z}{z+\sqrt{z^2-1}} \right)^{\ell}
\end{equation}
where we used 
\begin{equation}
 \frac{e^{\xi} + e^{-\xi}}{e^{\xi}} = \frac{2z}{z+\sqrt{z^2-1}}
\end{equation}
for $z=\cosh \xi$. 


We have then shown equation \eqref{estPQa}. The estimate for $P_{\ell}'$ and $Q_{\ell}'$ in equation \eqref{estPQ'a} follows immediately from equation \eqref{estPQa} and the recurrence relations (see \cite{bateman} pg 161 and \cite{watson} pg 318)

\begin{equation} \label{recursive}
(z^2-1) P_{\ell}'(z) = \ell( z P_{\ell}(z) - P_{\ell-1}(z) ), \quad (z^2-1) Q_{\ell}'(z) = \ell( z Q_{\ell}(z) - Q_{\ell-1}(z) )
\end{equation} 
\end{proof}

\bibliographystyle{amsplain}
\bibliography{refs}

\section*{Data Availability Statement}
Data sharing is not applicable to this article as no datasets were generated or analyzed in this paper.

\section*{Declaration} 

\textbf{Conflict of Interest:} The authors have no Conflict of interest to declare that are relevant to the content of this
article.
\end{document}